\documentclass{mcom-l}
\usepackage{xcolor}
\colorlet{siaminlinkcolor}{green!50!black}
\colorlet{siamexlinkcolor}{red!100!black}
\usepackage{amssymb}
\usepackage{bbm}

\usepackage{comment}
\usepackage{empheq}

\usepackage{graphicx,float}

\makeatletter
\@namedef{subjclassname@2020}{\textup{2020} Mathematics Subject Classification}
\makeatother

\makeatletter
\newcommand{\thickhline}{%
	\noalign {\ifnum 0=`}\fi \hrule height 1.2pt
	\futurelet \reserved@a \@xhline
}
\makeatother

\usepackage{amsopn}

\usepackage[hidelinks]{hyperref}
\usepackage[capitalize,nameinlink,noabbrev]{cleveref}

\newtheorem{theorem}{Theorem}[section]
\newtheorem{lemma}[theorem]{Lemma}
\newtheorem{proposition}[theorem]{Proposition}

\theoremstyle{definition}
\newtheorem{definition}[theorem]{Definition}

\theoremstyle{remark}
\newtheorem{remark}[theorem]{Remark}

\numberwithin{equation}{section}
\numberwithin{figure}{section}

\crefname{example}{Example}{Examples}
\crefname{hypothesis}{Hypothesis}{Hypotheses}
\crefname{conj}{Conjecture}{Conjectures}
\crefformat{equation}{\textup{#2(#1)#3}}
\crefrangeformat{equation}{\textup{#3(#1)#4--#5(#2)#6}}
\crefmultiformat{equation}{\textup{#2(#1)#3}}{ and \textup{#2(#1)#3}}
{, \textup{#2(#1)#3}}{, and \textup{#2(#1)#3}}
\crefrangemultiformat{equation}{\textup{#3(#1)#4--#5(#2)#6}}%
{ and \textup{#3(#1)#4--#5(#2)#6}}{, \textup{#3(#1)#4--#5(#2)#6}}{, and \textup{#3(#1)#4--#5(#2)#6}}

\Crefformat{equation}{#2Equation~\textup{(#1)}#3}
\Crefrangeformat{equation}{Equations~\textup{#3(#1)#4--#5(#2)#6}}
\Crefmultiformat{equation}{Equations~\textup{#2(#1)#3}}{ and \textup{#2(#1)#3}}
{, \textup{#2(#1)#3}}{, and \textup{#2(#1)#3}}
\Crefrangemultiformat{equation}{Equations~\textup{#3(#1)#4--#5(#2)#6}}%
{ and \textup{#3(#1)#4--#5(#2)#6}}{, \textup{#3(#1)#4--#5(#2)#6}}{, and \textup{#3(#1)#4--#5(#2)#6}}

\hypersetup{
	allcolors = siaminlinkcolor,
	linkcolor = siamexlinkcolor,
}

\begin{document}

\title[A unified SP-PFEM for anisotropic surface diffusion]{A unified structure-preserving parametric finite element method for
 anisotropic surface diffusion}


\author[W. Bao]{Weizhu Bao}
\address{Department of Mathematics, National University of Singapore, Singapore 119076}
\email{matbaowz@nus.edu.sg}

\author[Y. Li]{Yifei Li}
\address{Department of Mathematics, National University of Singapore, Singapore 119076}
\email{e0444158@u.nus.edu}

\subjclass[2020]{Primary 65M60, 65M12, 53E40, 53E10, 35K55}

\date{}
\keywords{anisotropic surface diffusion,  anisotropic surface energy density, parametric finite element method, structure-preserving, unconditional energy stability}

\begin{abstract}
We propose and analyze a unified structure-preserving parametric finite element method (SP-PFEM) for the anisotropic surface diffusion of curves in two dimensions $(d=2)$ and surfaces in three dimensions $(d=3)$ with an arbitrary anisotropic surface energy density $\gamma(\boldsymbol{n})$, where $\boldsymbol{n}\in \mathbb{S}^{d-1}$ represents the outward unit vector. By introducing a novel unified surface energy matrix $\boldsymbol{G}_k(\boldsymbol{n})$ depending on $\gamma(\boldsymbol{n})$, the Cahn-Hoffman $\boldsymbol{\xi}$-vector and a stabilizing function $k(\boldsymbol{n}):\  \mathbb{S}^{d-1}\to {\mathbb R}$, we obtain a unified and conservative variational formulation for the anisotropic surface diffusion via different surface differential operators, including the surface gradient operator, the surface divergence operator and the surface Laplace-Beltrami operator. A SP-PFEM discretization is presented for the variational problem. In order to establish the unconditional energy stability of the proposed SP-PFEM under a very mild condition on  $\gamma(\boldsymbol{n})$, we propose a new framework via {\sl local energy estimate} for proving energy stability/structure-preserving properties of the parametric finite element method for the anisotropic surface diffusion. This framework sheds light on how to prove unconditional energy stability of other numerical methods for geometric partial differential equations. Extensive numerical results are reported to demonstrate the efficiency and accuracy as well as structure-preserving properties of the proposed SP-PFEM for the anisotropic surface diffusion with arbitrary anisotropic surface energy density $\gamma(\boldsymbol{n})$ arising from different applications.
\end{abstract}

\maketitle


\section{Introduction}
\label{intro}

Surface diffusion is a fundamental model in materials science, which describes the diffusion of atoms or molecules within the surface of a solid material \cite{Mullins57}. In many solid materials, the diffusion rate varies from the crystallographic directions, attributable to differences in surface lattice orientations. This phenomenon, known as the anisotropic effect, is typically described by the anisotropic surface energy density and characterized by anisotropic surface diffusion. Anisotropic surface diffusion plays a critical role in the surface/materials sciences \cite{taylor1992overview,giga2006surface}, such as the growth of thin films \cite{gurtin2002interface,chang1999thermodynamic}, the formation of surface morphological patterns \cite{Thompson12}, and the design in heterogeneous catalysis \cite{hauffe1955application}. It also has many applications in solid-state physics and computer sciences, including solid-state dewetting \cite{jiang2012,Ye10a,wang2015sharp,Naffouti17,ZJB2021,ZJB2020}; producing continuous nanoporous metal coatings \cite{sharipova2022solid}; quantum dots manufacturing \cite{Fonseca14}; and image processing \cite{clarenz2000anisotropic}, among others. 

Let $\Gamma:=\Gamma(t)\subset\mathbb{R}^d$ be the evolving closed and orientable curve in two dimensions (2D) with $d=2$ or surface in three dimensions (3D) with $d=3$, and $\boldsymbol{n}\in \mathbb{S}^{d-1}$ represents the outward unit normal vector of $\Gamma(t)$. Assuming that the anisotropic effect is characterized by the anisotropic surface energy density $\gamma(\boldsymbol{n})>0$, then the total surface energy of $\Gamma$ is defined  as
\begin{equation}\label{Wgamma}
 W(\Gamma):=\int_{\Gamma} \gamma(\boldsymbol{n})\,dA, 
\end{equation} 
where $dA$ represents the area element. By using the thermodynamic variation, one can obtain 
 the chemical potential $\mu$ (also known as the weighted mean curvature $H_{\gamma}$) as
\begin{equation}\label{muHga}
\mu =H_{\gamma}:= \frac{\delta W(\Gamma)}{\delta \Gamma}=\lim_{\varepsilon\to0}
\frac{W(\Gamma^\varepsilon)-W(\Gamma)}{\varepsilon},
\end{equation}
where $\Gamma^{\varepsilon}$ represents a small perturbation of $\Gamma$ 
(see \cite{jiang2016solid,jiang2019sharp} for more details).
Then the anisotropic surface 
 diffusion of $\Gamma(t)$ is formulated as the following geometric 
 flow \cite{jiang2012,Mullins57,Thompson12,Naffouti17}
\begin{equation}\label{eq: aSD, geometric flow}
    V_n = \Delta_{\Gamma} \mu, 
\end{equation}
where $V_n$ denotes the normal velocity of $\Gamma(t)$ and $\Delta_{\Gamma}$ is the surface Laplace-Beltrami operator.

To formulate the chemical potential $\mu$, it is convenient to introduce 
the one-homogeneous extension of the anisotropic surface energy density $\gamma(\boldsymbol{n})$:
\begin{equation}\label{eq: def of gamma p}
\gamma(\boldsymbol{p}):=
    \begin{cases}
     |\boldsymbol{p}|\, \gamma\left(\frac{\boldsymbol{p}}
{|\boldsymbol{p}|}\right),&\forall \boldsymbol{p}=(p_1,\ldots, p_d)^T\in \mathbb{R}^d_*:=\mathbb{R}^d\setminus \{\boldsymbol{0}\};\\
      0, & \boldsymbol{p}=\boldsymbol{0},
    \end{cases}
\end{equation}
where $|\boldsymbol{p}|=\sqrt{p_1^2+\ldots+p_d^2}$. Following this, the Cahn-Hoffman $\boldsymbol{\xi}$-vector is defined as \cite{cahn1974vector}.
\begin{equation}\label{eq: def of xi}
  \boldsymbol{\xi}=(\xi_1, \xi_2, \ldots, \xi_d)^T=\boldsymbol{\xi}(\boldsymbol{n}):=\nabla \gamma(\boldsymbol{p})|_{\boldsymbol{p}=\boldsymbol{n}}.
\end{equation}
By \cite{hoffman1972vector}, the chemical potential $\mu$ can be represented by $\boldsymbol{\xi}$ as
\begin{equation}\label{eq: alter def of mu}
        \mu =\nabla_{\Gamma}\,\cdot \boldsymbol{\xi},
\end{equation}
where $\nabla_\Gamma\,\cdot$ is the surface divergence operator.

When $\gamma(\boldsymbol{n})\equiv 1$, i.e., isotropic case, then $\boldsymbol{\xi}=\boldsymbol{n}$. Consequently, the chemical potential $\mu$ (or the weighted mean curvature $H_\gamma$) becomes $\nabla_\Gamma \,\cdot \boldsymbol{n}$, thereby reducing to the mean curvature $H$ in 3D and the curvature $\kappa$ in 2D. In this case, the anisotropic surface diffusion \eqref{eq: aSD, geometric flow} collapses to the well-known surface diffusion. Like the surface diffusion, the anisotropic surface diffusion \eqref{eq: aSD, geometric flow} is characterized as a fourth-order, highly nonlinear partial differential equation. It possesses two fundamental geometric properties \cite{Cahn94,deckelnick2005computation}: (i) the conservation of the enclosed volume $V(t)$ by $\Gamma(t)$, and (ii) the decrease in total surface energy $W(t)$. In fact, it can be regarded as a $H^{-1}$-gradient flow of the total energy $W(\Gamma)$ in \eqref{Wgamma} \cite{taylor1992ii}. Therefore, finding a numerical approximation of the solution to \eqref{eq: aSD, geometric flow} that can preserve the two geometric properties is a notoriously difficult task.

Various numerical schemes have been developed to simulate anisotropic surface diffusion. These methods include the marker-particle method \cite{du2010tangent}, the finite difference method \cite{bansch2004surface,deckelnick2005computation}, the crystalline method \cite{carter1995shape,girao1996crystalline}, the discontinuous Galerkin finite element method \cite{xu2009local}, and the evolving surface finite element method \cite{kovacs2017convergence,jiang2021}. However, the absence of tangential velocity in the anisotropic surface diffusion \eqref{eq: aSD, geometric flow} often leads these methods to potential mesh point collisions.  To avoid this problem,  various methods with artificial tangential motion have been proposed, including using the DeTurck trick, which utilizes harmonic map heat flow on the initial surface \cite{EF2017}, the minimal deformation rate (MDR) approach \cite{HL2022}, and the minimal deformation (MD) formulation \cite{DL2024}, which employs a harmonic map from the initial surface onto the current surface. However, these adaptations fail to preserve the two geometric properties.  To address this issue, Barrett, Garcke and N{\"u}rnberg proposed a parametric finite element method (PFEM), which allows tangential movement that ensures the mesh points are asymptotically equally distributed for the surface diffusion of curves in 2D \cite{barrett2007parametric}. Moreover, it was proven to be unconditionally energy-stable, and thus is also known as energy-stable PFEM (ES-PFEM). The ES-PFEM was further extended to simultaneously accommodate both curves in 2D and surfaces in 3D, while still maintaining unconditional energy stability \cite{barrett2008parametric}. Recently, by introducing a clever approximation of the unit normal vector $\boldsymbol{n}$, Bao and Zhao proposed a structure-preserving PFEM (SP-PFEM) that not only inherits the unconditionally energy stability, but also preserves the enclosed volume in both 2D and 3D \cite{bao2021structurepreserving,bao2022volume,bao2022structure,BZ2023}. As a result, the PFEM has achieved remarkable success in isotropic surface diffusion and other geometric flows. For further details, we refer to \cite{Barrett2020}.

It is desirable to extend the PFEM to anisotropic surface diffusion with an arbitrary anisotropy while maintaining its structure-preserving/energy-stable property. Although numerous PFEMs for anisotropic surface diffusion with arbitrary anisotropic energy have been proposed \cite{bao2017parametric,jiang2016solid,jiang2019sharp,Hausser07}, they lack a rigorous energy stability analysis. The first energy-stable extension was developed by Barrett, Garcke, and Nürnburg for 2-dimensional curves \cite{barrett2008numerical}, and was subsequently extended to 3-dimensional surfaces \cite{barrett2008variational} for the very specific Riemannian metric anisotropy. Later, Bao and Li proposed an ES-PFEM for curves in 2D \cite{li2020energy} with general anisotropic surface energies. Nevertheless, the energy-stable condition on $\gamma(\boldsymbol{n})$ is rather complicated and restrictive. Recently, by introducing a symmetrized surface energy matrix $\boldsymbol{Z}_{k}(\boldsymbol{n})$, Bao, Jiang and Li proposed a symmetrized ES-PFEM for 2-dimensional curves \cite{bao2021symmetrized} and then 3-dimensional surfaces \cite{bao2022symmetrized}. The symmetrized ES-PFEM was proven to be energy-stable for any symmetric surface energy density, i.e., $\gamma(-\boldsymbol{n})=\gamma(\boldsymbol{n})$. Very recently, by introducing a novel surface energy matrix $\boldsymbol{G}_{k}(\boldsymbol{n})$, Bao and Li developed an ES-PFEM for anisotropic surface diffusion of 2-dimensional curves that only requires $\gamma(-\boldsymbol{n})<3\gamma(\boldsymbol{n})$ \cite{bao2022structure}. However, their analysis can not be extended to 3-dimensional surfaces due to that many surface operators are written in terms of the arclength $s$. To the best of our knowledge, no energy-stable PFEM for anisotropic surface diffusion of 3D surfaces with $\gamma(\boldsymbol{n})\neq \gamma(-\boldsymbol{n})$ has been reported in the literature.

The main aim of this paper is to design a unified structure-preserving PFEM for anisotropic surface diffusion for both 2-dimensional curves and 3-dimensional surfaces with with an arbitrary $\gamma(\boldsymbol{n})$, and to develop a unified analytical framework to prove the volume conservation and unconditional energy dissipation at the full-discretized level. Our contributions in this paper can be summarized in the following two aspects.\\

\vspace{1em}

\textbf{A. Derivation.} We construct a unified SP-PFEM for anisotropic surface diffusion \eqref{eq: aSD, geometric flow} for both curves in 2D and surfaces in 3D, which involves the following key steps
\begin{itemize}
\item We introduce a unified surface energy matrix $\boldsymbol{G}_k(\boldsymbol{n})$ and a stabilizing function $k(\boldsymbol{n})$ based on the $\boldsymbol{\xi}$-vector, see \eqref{eq: def of Gk}. Subsequently, we derive a unified strong/weak formulation that characterizes the chemical potential $\mu$ via $\boldsymbol{G}_k(\boldsymbol{n})$, utilizing the surface differential operator $\nabla_\Gamma$ (Theorem \ref{thm: mu by Gk, strong}, \ref{thm: mu by Gk}).
\item With the unified weak formulation of $\mu$, we give a unified conservative weak formulation for the anisotropic surface diffusion \eqref{eq: weak continuous}.
\item We utilize the $\Delta$-complex to provide a unified spatial approximation of $\Gamma$, accommodating both 2-dimensional curves and 3-dimensional surfaces. This, together with the implicit-explicit Euler method in time, yield our novel unified structure-preserving PFEM full discretization \eqref{eq: full PFEM} for the anisotropic surface diffusion. 
  \end{itemize}
\textbf{B. Analysis.} We significantly improve the energy-stable condition for $\gamma(\boldsymbol{n})$. We establish a novel and comprehensive analytical framework to prove the unconditional energy stability of the proposed SP-PFEM \eqref{eq: full PFEM}. This framework is characterized by three integrated  components: the \textit{local energy estimate}, the unified \textit{minimal stabilizing function} $k_0(\boldsymbol{n})$, and a unified approach for establishing the existence of $k_0(\boldsymbol{n})$. This framework not only enriches theoretical understanding but also sets a new benchmark in practical applications.
\begin{itemize}
\item We establish the energy stability of our unified SP-PFEM \eqref{eq: full PFEM} under the following elegant condition
\begin{equation}\label{eq: energy stable condition}
  \gamma(-\boldsymbol{n})<(5-d)\gamma(\boldsymbol{n})=\begin{cases}
      3\gamma(\boldsymbol{n}),&d=2;\\ 
      2\gamma(\boldsymbol{n}), & d=3.
    \end{cases}\qquad \text{and} \quad \gamma(\boldsymbol{p})\in C^2(\mathbb{R}^{d}_*).
\end{equation}
Our new energy-stable condition \eqref{eq: energy stable condition} is the first in handling non-symmetric $\gamma(\boldsymbol{n})$ in 3D, aligning with existing mild conditions in 2D as per \cite{bao2022structure}. Remarkably, the symmetric anisotropic $\gamma(-\boldsymbol{n})=\gamma(\boldsymbol{n})$ satisfies the condition \eqref{eq: energy stable condition} automatically.
\item We introduce a new and unified concept \textit{local energy estimate} \eqref{eq: energy difference two steps, local}, which is a sufficient condition for the energy stability.
\item We introduce a unified \textit{minimal stabilizing function}, $k_0(\boldsymbol{n})$, defined via the positive semi-definiteness of an auxiliary matrix, which is crucial for establishing the \textit{local energy estimate}. This unified definition contrasts with prior research \cite{bao2021symmetrized,bao2022symmetrized,bao2022structure}, where $k_0(\boldsymbol{n})$ was dependent on dimension-dependent inequalities.
\item We develop a unified approach to establish the existence of $k_0(\boldsymbol{n})$. Firstly, we reduce the existence of $k_0(\boldsymbol{n})$ to the positive semi-definiteness of the auxiliary matrix. Subsequently, we employ the representation of $SO(d)$ to prove the positive semi-definiteness, see Lemma \ref{lem: compactness}.
    Here $SO(d)$ stands for the special orthogonal group in dimension $d$.
\end{itemize}

The rest of this paper is organized as follows. In section 2, we introduce the mathematical formulations and present the unified surface energy matrix $\boldsymbol{G}_k(\boldsymbol{n})$. Utilizing $\boldsymbol{G}_k(\boldsymbol{n})$, we further derive a novel unified strong/weak formulation for the chemical potential $\mu$ and the anisotropic surface diffusion. In Section 3, we present a unified structure-preserving PFEM full discretization, achieved through $\Delta$-complex based unified spatial discretization and implicit-explicit Euler time discretization. Consequently, we state the main result, the structure-preserving property of the unified SP-PFEM.  Section 4 develops a comprehensive analytical framework for energy stability, starting with defining the minimal stabilizing function $k_0(\boldsymbol{n})$ using the auxiliary matrices $\tilde{M}$ and $M$ for $d=2$ and $d=3$, respectively. Assuming the existence of $k_0(\boldsymbol{n})$, we establish the main result by utilizing the local energy estimate. The existence of $k_0(\boldsymbol{n})$ is proved by introducing a unified approach in section 5. Section 6 provides the numerical evidence to demonstrate our structure-preserving analytical results and show efficiency and accuracy of the unified SP-PFEM. Finally, some concluding remarks are drawn in section 7. 

\section{A unified formulation}
\setcounter{equation}{0}

\subsection{Mathematical formulation}
Let $\Gamma:=\Gamma(t)\subset \mathbb{R}^d$ represent the closed orientable $C^2$-evolving curve/surface, and  $\boldsymbol{n}$ be the outward unit normal vector of $\Gamma(t)$. The parameterization of $\Gamma(t)$ is given by $\boldsymbol{X}(\boldsymbol{\rho}, t)$ as follows:
\begin{equation}\label{eq: global para}
  \boldsymbol{X}(\cdot, t):\Gamma_0 \to \mathbb{R}^d, \,(\boldsymbol{\rho}, t)\mapsto \boldsymbol{X}(\boldsymbol{\rho}, t):=(X_1(\boldsymbol{\rho},t), \ldots, X_d(\boldsymbol{\rho}, t))^T,
\end{equation}
where $\boldsymbol{\rho}\in \Gamma_0\subset\mathbb{R}^d$ is the initial closed orientable $C^2$-evolving curve/surface. 

Consider $f$ as a differentiable scalar-valued function on $\Gamma(t)$, the surface gradient operator $\nabla_{\Gamma} f$ is defined as \cite{deckelnick2005computation,bao2022symmetrized}
\begin{equation}\label{eq: surface gradient}
    \nabla_{\Gamma} f=\nabla_{\Gamma(t)} f:=(\underline{D}_1 f, \ldots, \underline{D}_d f)^T.
\end{equation}
For the definitions of $\underline{D}_1, \ldots, \underline{D}_d$, see \cite{deckelnick2005computation}.

The surface Jacobian, surface divergence for a differentiable vector-valued function $\boldsymbol{f}=(f_1, \ldots, f_d)^T\in \mathbb{R}^d$, and the surface Laplace-Beltrami for a second-order differentiable scalar-valued function $f$ defined on $\Gamma(t)$ are
\begin{subequations}
\label{eq: surface grad and divergence}
\begin{align}
\label{eq: surface jacobian}
&\nabla_{\Gamma} \boldsymbol{f}=\nabla_{\Gamma(t)} \boldsymbol{f} := \begin{bmatrix}(\nabla_{\Gamma} f_1)^T\\\vdots\\(\nabla_{\Gamma} f_d)^T\end{bmatrix}= \begin{bmatrix}\underline{D}_1f_1&\cdots&\underline{D}_d f_1\\\vdots&\ddots&\vdots\\\underline{D}_1f_d&\cdots&\underline{D}_d f_d\end{bmatrix},\\
\label{eq: surface divergence}
&\nabla_{\Gamma}\cdot \boldsymbol{f}=\nabla_{\Gamma(t)} \cdot \boldsymbol{f}:=\sum_{i=1}^d\underline{D}_i f_i,\\
\label{eq: surface laplace}
&\Delta_\Gamma f=\Delta_{\Gamma(t)} f:=\nabla_{\Gamma}\, \cdot \,\nabla_{\Gamma} f=\sum_{i=1}^d\underline{D}_i\left(\underline{D}_i f\right).
\end{align}
\end{subequations}

Utilizing the alternate formulation of chemical potential $\mu$ \eqref{eq: alter def of mu}, alongside the definition of $\boldsymbol{\xi}$ \eqref{eq: def of xi}, and the parameterization $\boldsymbol{X}$ \eqref{eq: global para}, the anisotropic surface diffusion equation \eqref{eq: aSD, geometric flow} can be reformulated into the following PDE:
\begin{subequations}
\label{eq: asd pde}
\begin{empheq}[left=\empheqlbrace]{align}
\label{eq: asd pde1}
& \partial_t \boldsymbol{X} =\Delta_{\Gamma} \mu\, \boldsymbol{n},\\
\label{eq: asd pde2}
&\mu = \nabla_{\Gamma} \cdot \boldsymbol{\xi},\, \quad \boldsymbol{\xi}(\boldsymbol{n})=\nabla \gamma(\boldsymbol{p})|_{\boldsymbol{p}=\boldsymbol{n}}.
\end{empheq}
\end{subequations}

To derive an equivalent formulation for \eqref{eq: asd pde}, we introduce the functional space $L^2(\Gamma(t))$ with respect to $\Gamma(t)$ as follows
\begin{equation}\label{eq: def of L2 space}
L^2(\Gamma(t)):=\left\{u: \Gamma(t)\rightarrow \mathbb{R} \ |\   \int_{\Gamma(t)} |u|^2 \, dA <+\infty \right\},
\end{equation}
with the inner product $(\cdot, \cdot)_{\Gamma(t)}$ as
\begin{equation}
  (u, v)_{\Gamma(t)}:=\int_{\Gamma(t)}u\,v\, dA \quad \forall u, v\in L^2(\Gamma(t)).
\end{equation}
The functional spaces $[L^2(\Gamma(t))]^d$ and $[L^2(\Gamma(t))]^{d\times d}$ can be given similarly. In particular, the inner product for two matrix-valued functions $\boldsymbol{U}, \boldsymbol{V}\in [L^2(\Gamma(t))]^{d\times d}$ is emphasized as
\begin{equation}\label{eq: def of L2 inner product, matrix}
    \langle \boldsymbol{U}, \boldsymbol{V}\rangle_{\Gamma(t)}:=\int_{\Gamma(t)}\boldsymbol{U}:\boldsymbol{V}\,dA,
\end{equation}
here $\boldsymbol{U}:\boldsymbol{V}=\text{Tr}(\boldsymbol{V}^T\boldsymbol{U})$ is the Frobenius inner product. 

Furthermore, the Sobolev spaces $H^1(\Gamma(t))$ and $[H^1(\Gamma(t))]^d$ are defined as
\begin{equation}\label{eq: def of H1 space}
H^1(\Gamma(t)):=\left\{u: \Gamma(t)\rightarrow \mathbb{R} \ |\   u\in L^2(\Gamma(t)), \,\nabla_{\Gamma} u\in [L^2(\Gamma(t))]^d\, \right\},
\end{equation}
\begin{equation}\label{eq: def of H1 space, vector}
[H^1(\Gamma(t))]^d:=\left\{\boldsymbol{u}: \Gamma(t)\rightarrow \mathbb{R}^d \ |\   \boldsymbol{u}\in [L^2(\Gamma(t))]^d, \,\nabla_{\Gamma}\, \boldsymbol{u}\in [L^2(\Gamma(t))]^{d\times d}\, \right\}.
\end{equation}

\subsection{A unified surface energy matrix and conservative  formulation}
In order to develop a formulation for the PDE representation \eqref{eq: asd pde} of the anisotropic surface diffusion, it is essential to obtain an appropriate formulation of $\mu$. To achieve this, we introduce a unified surface energy matrix $\boldsymbol{G}_k(\boldsymbol{n})$ as follows:
\begin{definition}[Surface energy matrix]
The unified surface energy matrix $\boldsymbol{G}_k(\boldsymbol{n})$ is given as
\begin{equation}\label{eq: def of Gk}
\boldsymbol{G}_k=\boldsymbol{G}_k(\boldsymbol{n}):=\gamma(\boldsymbol{n})I_d
-\boldsymbol{n}\boldsymbol{\xi}^T+\boldsymbol{\xi}\boldsymbol{n}^T+k(\boldsymbol{n})
\boldsymbol{n}\boldsymbol{n}^T:=\boldsymbol{G}_k^{(s)}+\boldsymbol{G}^{(a)},
\end{equation}
where $I_d$ is the $d\times d$ identity matrix, 
$k(\boldsymbol{n}): \mathbb{S}^{d-1}\to \mathbb{R}_{\geq 0}$ is a stabilizing function,
and $\boldsymbol{G}_k^{(s)}$ is its symmetric part and $\boldsymbol{G}^{(a)}$ is its anti-symmetric part, which are given as
\begin{equation}\label{eq: Gs and Ga}
  \boldsymbol{G}_k^{(s)}:=\gamma(\boldsymbol{n})I_d+k(\boldsymbol{n})\boldsymbol{n}\boldsymbol{n}^T, \,\boldsymbol{G}^{(a)}:=-\boldsymbol{n}\boldsymbol{\xi}^T+\boldsymbol{\xi}\boldsymbol{n}^T, \, \boldsymbol{G}_k=\boldsymbol{G}_k^{(s)}+\boldsymbol{G}^{(a)}.
\end{equation}
\end{definition}

The significance of the unified surface energy matrix $\boldsymbol{G}_k(\boldsymbol{n})$ is demonstrated in the following theorem.
{
\begin{theorem}[Strong formulation of $\mu$]\label{thm: mu by Gk, strong}
  The chemical potential $\mu$ given in \eqref{eq: alter def of mu} can be represented in terms of $\boldsymbol{G}_k(\boldsymbol{n})$ through the following strong formulation:
  \begin{equation}\label{eq: mu by Gk, strong}
    \nabla_{\Gamma} \cdot \left(\boldsymbol{G}_k(\boldsymbol{n})\nabla_\Gamma \boldsymbol{X}\right) = -\mu \boldsymbol{n}.
  \end{equation}
\end{theorem}
\begin{proof}
  We begin by using the identity $\nabla_\Gamma \boldsymbol{X}=I_d-\boldsymbol{n}\boldsymbol{n}^T$ \cite[Lemma 9 (i)]{Barrett2020}. Combining this with \eqref{eq: def of Gk} and the identity $\gamma(\boldsymbol{n}) = \boldsymbol{\xi}\cdot \boldsymbol{n}$, we derive:
  \begin{align*}
    \boldsymbol{G}_k(\boldsymbol{n})\nabla_\Gamma \boldsymbol{X} &= \left(\gamma(\boldsymbol{n})I_d-\boldsymbol{n}\boldsymbol{\xi}^T+\boldsymbol{\xi}\boldsymbol{n}^T+k(\boldsymbol{n})\boldsymbol{n}\boldsymbol{n}^T\right)(I_d-\boldsymbol{n}\boldsymbol{n}^T)\\
    &= \gamma(\boldsymbol{n})I_d - \boldsymbol{n}\boldsymbol{\xi}^T - (\gamma(\boldsymbol{n})I_d - \boldsymbol{n}\boldsymbol{\xi}^T)\boldsymbol{n}\boldsymbol{n}^T\\
    &= \gamma(\boldsymbol{n})I_d-\boldsymbol{n}\boldsymbol{\xi}^T - \gamma(\boldsymbol{n})\boldsymbol{n}\boldsymbol{n}^T+(\boldsymbol{\xi}\cdot\boldsymbol{n})\boldsymbol{n}\boldsymbol{n}^T\\
    &= \gamma(\boldsymbol{n})I_d-\boldsymbol{n}\boldsymbol{\xi}^T.
  \end{align*}
  Next, applying Definition 5 from \cite{Barrett2020}, we obtain:
  \begin{align*}
    \nabla_\Gamma \cdot (\gamma(\boldsymbol{n})I_d) &= \nabla_\Gamma \gamma(\boldsymbol{n}) 
    = \sum_{i=1}^{d-1} \left(\partial_{\boldsymbol{\tau}_i} \gamma(\boldsymbol{n})\right)\boldsymbol{\tau}_i
    = \sum_{i=1}^{d-1} \left( \boldsymbol{\xi}\cdot \partial_{\boldsymbol{\tau}_i} \boldsymbol{n}\right)\boldsymbol{\tau}_i\\
    &= \left(\sum_{i=1}^{d-1}\boldsymbol{\tau}_i (\partial_{\boldsymbol{\tau}_i} \boldsymbol{n})^T\right) \boldsymbol{\xi}
    = (\nabla_\Gamma \boldsymbol{n})^T \boldsymbol{\xi},
  \end{align*}
  where $\{\boldsymbol{\tau}_1,\ldots, \boldsymbol{\tau}_{d-1}, \boldsymbol{n}\}$ forms an orthonormal basis of $\mathbb{R}^d$.
  Finally, we employ the chain rule of surface divergence and the symmetry of $\nabla_\Gamma \boldsymbol{n}$ from \cite{Barrett2020}, together with \eqref{eq: alter def of mu}, to deduce:
  \begin{align*}
    \nabla_\Gamma\cdot \left(\boldsymbol{G}_k(\boldsymbol{n})\nabla_\Gamma \boldsymbol{X}\right)
    &=\nabla_\Gamma \cdot \left(\gamma(\boldsymbol{n})I_d -\boldsymbol{n}\boldsymbol{\xi}^T\right)\\
    &= (\nabla_\Gamma \boldsymbol{n})^T \boldsymbol{\xi} - (\nabla_\Gamma\cdot \boldsymbol{\xi})\boldsymbol{n} - (\nabla_\Gamma \boldsymbol{n})\boldsymbol{\xi} \\
    &= (\nabla_\Gamma \boldsymbol{n}) \boldsymbol{\xi} - \mu \boldsymbol{n} - (\nabla_\Gamma \boldsymbol{n})\boldsymbol{\xi}\\
    &= -\mu \boldsymbol{n},
  \end{align*}
  which establishes the desired identity \eqref{eq: mu by Gk, strong}.
\end{proof}

Having established the unified strong formulation of $\mu$, we now present its unified weak formulation, which is crucial for the numerical scheme.
\begin{theorem}[Weak formulation of $\mu$]\label{thm: mu by Gk}
Let $\Gamma\subset\mathbb{R}^d$ be a closed orientable $C^2$-curve/surface with the outward unit normal vector $\boldsymbol{n}=(n_1, \ldots, n_d)^T$. For any $\boldsymbol{\omega}=(\omega_1, \ldots, \omega_d)^T \in [H^1(\Gamma)]^d$, the following identity holds:
\begin{equation}\label{eq: mu by Gk}
  \left(\mu\boldsymbol{n}, \boldsymbol{\omega}\right)_{\Gamma}=\langle \boldsymbol{G}_k(\boldsymbol{n})\nabla_{\Gamma} \boldsymbol{X}, \nabla_{\Gamma} \boldsymbol{\omega}\rangle_{\Gamma}.
\end{equation}
\end{theorem}

\begin{proof}
  We begin with the strong formulation \eqref{eq: mu by Gk, strong}. Multiplying both sides by a test function $\boldsymbol{\omega}$ and integrating over $\Gamma$, we obtain:
  \begin{align}
  \label{newmunomega}
    \int_{\Gamma}\mu\boldsymbol{n}\cdot \boldsymbol{\omega}\,dA = -\int_{\Gamma}\Bigl(\nabla_{\Gamma}\cdot (\boldsymbol{G}_k(\boldsymbol{n})\nabla_{\Gamma} \boldsymbol{X})\Bigr)\cdot \boldsymbol{\omega}\,dA.
  \end{align}
  To simplify the right-hand side of \eqref{newmunomega}, we employ the divergence theorem for matrix-valued functions, as stated in Lemma \ref{exlem: divergence theorem for matrix-valued functions} in the Appendix A with 
  $\boldsymbol{F}=\boldsymbol{G}_k(\boldsymbol{n})$ in \eqref{matrxintegr}, to deduce that
  \begin{align*}
  -\int_{\Gamma}\Bigl(\nabla_{\Gamma}\cdot (\boldsymbol{G}_k(\boldsymbol{n})\nabla_{\Gamma} \boldsymbol{X})\Bigr)\cdot \boldsymbol{\omega}\,dA=\int_{\Gamma}\left(\boldsymbol{G}_k(\boldsymbol{n})\nabla_{\Gamma} \boldsymbol{X}\right):\nabla_{\Gamma}\boldsymbol{\omega}\,dA.
  \end{align*}
  Combining these two equations, we obtain the desired weak formulation \eqref{eq: mu by Gk}.
  \end{proof}}

  To obtain a unified weak formulation of \eqref{eq: asd pde}, we rewrite \eqref{eq: asd pde1} as $\boldsymbol{n}\cdot \partial_t \boldsymbol{X} = \Delta_{\Gamma} \mu$, multiply by a test function $\phi \in H^1(\Gamma(t))$, integrate over $\Gamma(t)$, and apply integration by parts. We then combine this with the unified weak form of \eqref{eq: asd pde2} derived from Theorem \ref{thm: mu by Gk}. The resulting unified conservative weak formulation is as follows: Let the initial closed and orientable curve/surface be $\Gamma_0$ and the function $\boldsymbol{X}_0(\boldsymbol{\rho})=\boldsymbol{\rho}, \forall \boldsymbol{\rho}\in \Gamma_0$. Find the solution $(\boldsymbol{X}(\cdot, t), \mu(\cdot, t))\in [H^1(\Gamma(t))]^d\times H^1(\Gamma(t))$, such that $\boldsymbol{X}(\cdot, 0)=\boldsymbol{X}_0(\cdot)$ and 
\begin{subequations}\label{eq: weak continuous}
\begin{align}
\label{eq: weak continuous 1}
&\left(\partial_t \boldsymbol{X}\cdot \boldsymbol{n}, \phi\right)_{\Gamma(t)}+\left(\nabla_{\Gamma}\mu, \nabla_{\Gamma} \phi\right)_{\Gamma(t)}=0,\quad \forall \phi\in H^1(\Gamma(t)),\\
\label{eq: weak continuous 2}
&\left(\mu\boldsymbol{n}, \boldsymbol{\omega}\right)_{\Gamma(t)}-\langle \boldsymbol{G}_k(\boldsymbol{n})\nabla_{\Gamma} \boldsymbol{X}, \nabla_{\Gamma} \boldsymbol{\omega}\rangle_{\Gamma(t)}=0,\quad \forall \boldsymbol{\omega}\in [H^1(\Gamma(t))]^d.
\end{align}
\end{subequations}

For the unified conservative weak formulation \eqref{eq: weak continuous}, in the same way as Theorem 2.2 in \cite{bao2022symmetrized}, it can be shown that the two geometric properties are well preserved.
\begin{proposition}\label{thm: two geometric properties, weak}
Let $\Gamma(t)$ be the solution of the unified conservative weak formulation \eqref{eq: weak continuous}. Denote $V(t)$ as the enclosed volume and $W(t)$ as the total energy of the closed orientable evolving curve/surface $\Gamma(t)$, respectively, which are formally given by
\begin{equation}
    V(t):=\frac{1}{d}\int_{\Gamma(t)}\boldsymbol{X}\cdot \boldsymbol{n}\, dA, \qquad W(t):=\int_{\Gamma(t)}\gamma(\boldsymbol{n})\,dA.
\end{equation} 
Then the enclosed volume $V(t)$ is conserved, and the total energy $W(t)$ is dissipative, i.e.,
\begin{equation}\label{eq: two geometric properties, weak}
  V(t)\equiv V(0)=\frac{1}{d}\int_{\Gamma_0}\boldsymbol{X}_0\cdot \boldsymbol{n}\, dA, \quad W(t)\le W(t^\prime)\le W(0), \ \forall t\ge t^\prime\ge 0.
\end{equation}
\end{proposition}

\section{A unified SP-PFEM}

\setcounter{equation}{0}

\subsection{A unified SP-PFEM discretization}
Choose a time step $\tau>0$, with $t_m=m\tau$ representing the discretized time level for $m=0, 1, \ldots$, and $\Gamma(t)$ at $t=t_m$ is approximated by  $\Gamma^m$. For a unified discretization of \eqref{eq: weak continuous}, we utilize a closed orientable $\Delta$-complex to approximate $\Gamma^m$, which is comprised by disjoint $(d-1)$-simplices $\sigma_j^m=[\boldsymbol{q}_{j_1}^m, \ldots, \boldsymbol{q}_{j_d}^m]$ for $1\leq j\leq J$, i.e., 
\begin{equation}
  \Gamma^m:=\cup_{j=1}^J \sigma_j^m.
\end{equation}
For the detailed definition of the $\Delta$-complex, we refer to \cite{Hatcher02}. Moreover, each $(d-1)$-simplex $\sigma_j \subset \mathbb{R}^d$ is associated with a direction vector $\mathcal{J}\{\sigma_j\}$, aligned with the orientation $[\boldsymbol{q}_{j_1}, \ldots, \boldsymbol{q}_{j_d}]$ as 
  \begin{equation}\label{eq: def of direction d}
      \mathcal{J}_j=\mathcal{J}\{\sigma_j\}:=(\boldsymbol{q}_{j_2}-\boldsymbol{q}_{j_1})\wedge\cdots\wedge (\boldsymbol{q}_{j_d}-\boldsymbol{q}_{j_1}).
  \end{equation}
Here $\wedge$ is the wedge product \cite{Hatcher02}, and this $\mathcal{J}_j$ satisfies
\begin{equation}\label{eq: wedge product}
    \mathcal{J}_j\cdot \boldsymbol{u}=\det [\boldsymbol{q}_{j_2}-\boldsymbol{q}_{j_1}, \ldots, \boldsymbol{q}_{j_d}-\boldsymbol{q}_{j_1}, \boldsymbol{u}], \qquad \forall \boldsymbol{u}\in \mathbb{R}^d,
\end{equation}
see \cite[Definition 45]{Barrett2020}. Specifically, for $d=2$, the $1$-simplex $\sigma_j=[\boldsymbol{q}_{j_1}, \boldsymbol{q}_{j_2}]$ is a line segment with vertices $\boldsymbol{q}_{j_1}$ and $\boldsymbol{q}_{j_2}$, and its direction vector $\mathcal{J}\{\sigma_j\}$ is defined as (c.f. Figure \ref{fig: illu})
\begin{equation}\label{eq: def of direction 2d}
  \mathcal{J}\{\sigma_j\}:=-(\boldsymbol{q}_{j_2}-\boldsymbol{q}_{j_1})^\perp,
 \end{equation}
  where $(u_1, u_2)^\perp=(u_2, -u_1), \forall \boldsymbol{u}=(u_1, u_2)\in \mathbb{R}^2$. 

For $d=3$, the $2$-simplex $\sigma_j=[\boldsymbol{q}_{j_1}, \boldsymbol{q}_{j_2}, \boldsymbol{q}_{j_3}]$ is a triangle with vertices $\boldsymbol{q}_{j_1}$, $\boldsymbol{q}_{j_2}$, and $\boldsymbol{q}_{j_3}$, and its direction vector $\mathcal{J}\{\sigma_j\}$ is given by  (c.f. Figure \ref{fig: illu})
\begin{equation}\label{eq: def of direction}
  \mathcal{J}\{\sigma_j\}:=(\boldsymbol{q}_{j_2}-\boldsymbol{q}_{j_1})\times 
  (\boldsymbol{q}_{j_3}-\boldsymbol{q}_{j_1}).
\end{equation}

 \begin{figure}[htp!]
\centering
\includegraphics[width=0.5\textwidth]{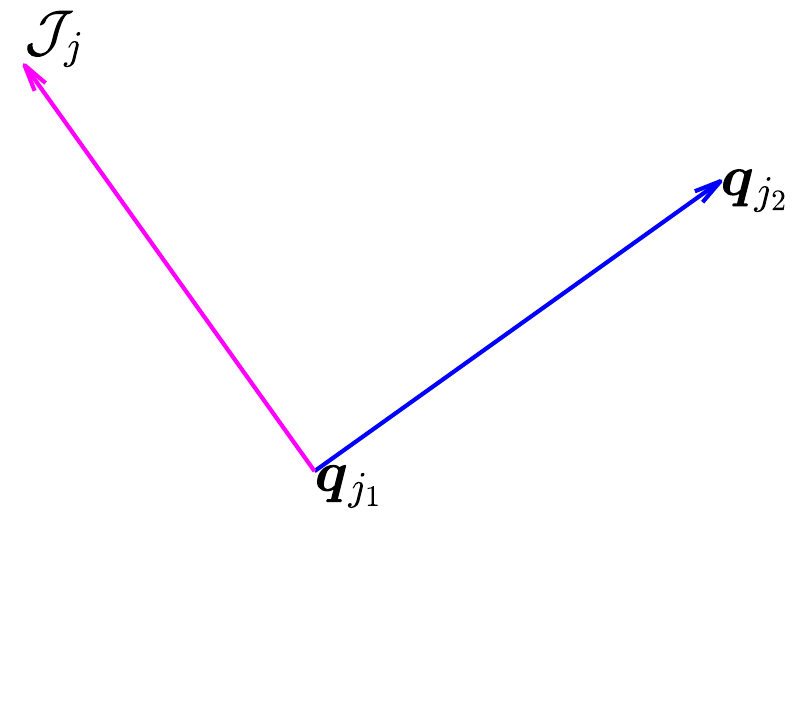}\includegraphics[width=0.5\textwidth]{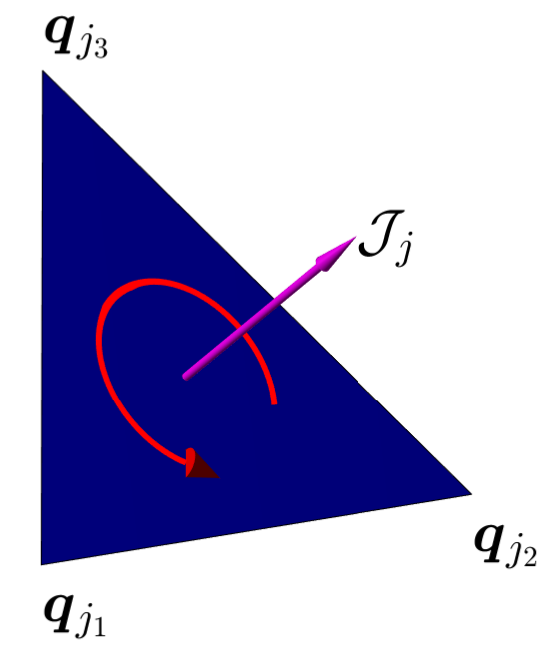}
\caption{Plot of the direction vector $\mathcal{J}$ for  2D (left) and for 3D (right).}
\label{fig: illu}
\end{figure}

Employing the direction vector $\mathcal{J}\{\sigma_j\}$ enables the representation of the area $|\sigma_j|$ and outward normal vector $\boldsymbol{n}_j$ for each $(d-1)$-simplex $\sigma_j$ as follows:
\begin{equation}\label{eq: normal vector}
  |\sigma_j|:=\frac{1}{d-1}|\mathcal{J}\{\sigma_j\}|, \qquad \boldsymbol{n}_j:=\frac{\mathcal{J}\{\sigma_j\}}{|\mathcal{J}\{\sigma_j\}|}.
\end{equation}

The finite element space for $\Gamma^m$ is defined as
\begin{equation}\label{eq: def of K space}
  \mathbb{K}^m=\mathbb{K}(\Gamma^m):=\Bigl\{u\in C(\Gamma^m)\Big|\, u|_{\sigma_{j}^m}\in \mathcal{P}^1(\sigma_{j}^m),\, \forall 1\leq j\leq J\Bigr\},
\end{equation}
here $\mathcal{P}^1(\sigma_{j}^m)$ is the set of polynomials on $\sigma_{j}^m$ with degree no higher than $1$. For $u, v\in \mathbb{K}^m$, the mass-lumped inner product $(u,v)^h_{\Gamma^m}$ is defined as
\begin{equation}\label{eq: mass lumped inner product}
  \left(u, v\right)_{\Gamma^m}^h:=\frac{1}{d}\sum_{j=1}^J\sum_{i=1}^d|\sigma_j^m|\,u((\boldsymbol{q}_{j_i}^m)^-)\,v((\boldsymbol{q}_{j_i}^m)^-),
\end{equation}
where $u((\boldsymbol{q}_{j_i}^m)^-)=\lim\limits_{\substack{\boldsymbol{q}\to \boldsymbol{q}_{j_i}^m\\ \boldsymbol{q}\in \sigma_j^m}}u(\boldsymbol{q})$ and $|\sigma_j^m|:=\frac{1}{d-1}|\mathcal{J}\{\sigma_j^m\}|$. This definition holds true for $[\mathbb{K}^m]^d$, $[\mathbb{K}^m]^{d\times d}$, and applies to the piecewise constant functions as well. Similar to the continuous situation, we emphasize  the mass-lumped inner product for two matrix-valued functions $\boldsymbol{U}, \boldsymbol{V}$ as follows
\begin{equation}\label{eq: mass lumped inner product, matrix}
  \langle\boldsymbol{U}, \boldsymbol{V}\rangle_{\Gamma^m}^h:=\frac{1}{d}\sum_{j=1}^J\sum_{i=1}^d|\sigma_{j}^m|\,\boldsymbol{U}((\boldsymbol{q}_{j_i}^m)^-):\,\boldsymbol{V}((\boldsymbol{q}_{j_i}^m)^-).
\end{equation}

Suppose the initial closed orientable $C^2$-evolving curve/surface $\Gamma_0$ is approximated by the closed orientable $\Delta$-complex $\Gamma^0=\cup_{j=1}^J \sigma_j^0$ with $\sigma_j^0=[\boldsymbol{q}_{j_1}^0, \ldots, \boldsymbol{q}_{j_d}^0]$. Applying backward-Euler discretization in time and the PFEM discretization in space to the unified weak form \eqref{eq: weak continuous}, a semi-implicit unified SP-PFEM for anisotropic surface diffusion is derived for both $d=2$ and $d=3$, as follows:

 For each $m=0,1,2,\ldots$, find the solution $(\boldsymbol{X}^{m+1}, \mu^{m+1})\in [\mathbb{K}^m]^d\times \mathbb{K}^m$ such that 
\begin{subequations}\label{eq: full PFEM}
\begin{align}
\label{eq: full PFEM 1}
&\left(\frac{\boldsymbol{X}^{m+1}-\boldsymbol{X}^m}{\tau}\cdot \boldsymbol{n}^{m+\frac{1}{2}}, \phi\right)_{\Gamma^m}^h+\left(\nabla_{\Gamma}\mu^{m+1}, \nabla_{\Gamma} \phi\right)_{\Gamma^m}^h=0, \quad \forall \phi\in \mathbb{K}^m,\\
\label{eq: full PFEM 2}
&\left(\mu^{m+1}\boldsymbol{n}^{m+\frac{1}{2}}, \boldsymbol{\omega}\right)_{\Gamma^m}^h-\langle \boldsymbol{G}_k(\boldsymbol{n}^m)\nabla_{\Gamma} \boldsymbol{X}^{m+1}, \nabla_{\Gamma} \boldsymbol{\omega}\rangle_{\Gamma^m}^h=0, \quad \forall \boldsymbol{\omega}\in [\mathbb{K}^m]^d.
\end{align}
\end{subequations}
Here $\boldsymbol{X}^m(\boldsymbol{q}_{j_i}^m)=\textbf{id}(\boldsymbol{q}_{j_i}^m)=\boldsymbol{q}_{j_i}^m$, the vertex $\boldsymbol{q}_{j_i}^{m+1}:=\boldsymbol{X}^{m+1}(\boldsymbol{q}_{j_i}^m)$, the $(d-1)$ simplex $\sigma_j^{m+1}$ is given by $\sigma_j^{m+1}:=[\boldsymbol{q}_{j_1}^{m+1}, \ldots, \boldsymbol{q}_{j_d}^{m+1}]=\boldsymbol{X}^{m+1}(\sigma_j^m)$, and the closed orientable $\Delta$-complex $\Gamma^{m+1}$ is given by $\cup_{j=1}^J \sigma_j^{m+1}=\boldsymbol{X}^{m+1}(\Gamma^m)$. 

The discretized surface gradient operator $\nabla_{\Gamma}$ for a $1$-simplex $\sigma=[\boldsymbol{q}_1, \boldsymbol{q}_2]$ in 2D becomes
\begin{equation}\label{eq: def of surface gradient dis 1 2d}
  \nabla_{\Gamma} f|_{\sigma}:=\left(f(\boldsymbol{q}_{2})-f(\boldsymbol{q}_{1})\right) \frac{\boldsymbol{q}_{2}-\boldsymbol{q}_{1}}{|\sigma|^2},\qquad \forall f \in \mathcal{P}^1(\sigma).
\end{equation}
And for a $2$-simplex $\sigma=[\boldsymbol{q}_{1}, \boldsymbol{q}_2, \boldsymbol{q}_3]$ in 3D, it is
\begin{align}\label{eq: def of surface gradient dis 1}
  \nabla_{\Gamma} f|_{\sigma}:=&\left[f(\boldsymbol{q}_{1})(\boldsymbol{q}_{2}-\boldsymbol{q}_{3})+f(\boldsymbol{q}_{2})
  (\boldsymbol{q}_{3}-\boldsymbol{q}_{1})+f(\boldsymbol{q}_{3})(\boldsymbol{q}_{1}-\boldsymbol{q}_{2})\right]\nonumber\\
    &\times \frac{\boldsymbol{n}}{2|\sigma|},\qquad \forall f \in \mathcal{P}^1(\sigma).
\end{align}
The surface Jacobian for a vector-valued function $\boldsymbol{f}=(f_1, f_2, \ldots, f_d)^T$ is
\begin{equation}\label{eq: def of surface jacob dis}
    \nabla_{\Gamma} \boldsymbol{f}|_{\sigma}:= \begin{bmatrix}(\nabla_{\Gamma} f_1)^T\\\vdots\\(\nabla_{\Gamma} f_d)^T\end{bmatrix} \qquad \forall \boldsymbol{f} \in [\mathcal{P}^1(\sigma)]^d.
\end{equation}
The vector $\boldsymbol{n}^{m+\frac{1}{2}}$ is determined as \cite{bao2021structurepreserving}
\begin{equation}
  \boldsymbol{n}^{m+\frac{1}{2}}|_{\sigma_j^m}:=\begin{cases}
      \frac{1}{2}\frac{1}{|\sigma_j^m|}(\mathcal{J}\{\sigma_j^m\}+\mathcal{J}\{\sigma_j^{m+1}\}),&d=2,\\
      \frac{\mathcal{J}\{\sigma_j^m\}+4\mathcal{J}\{\sigma_j^{m+\frac{1}{2}}\}+\mathcal{J}\{\sigma_j^{m+1}\}}
      {12|\sigma_j^m|}, & d=3;
    \end{cases}
\end{equation}
where $\sigma_j^{m+\frac{1}{2}}:=\frac{1}{2}(\sigma_j^m+\sigma_j^{m+1})=\Big[\frac{\boldsymbol{q}_{j_1}^m+\boldsymbol{q}_{j_1}^{m+1}}{2}, \ldots, \frac{\boldsymbol{q}_{j_d}^m+\boldsymbol{q}_{j_d}^{m+1}}{2}\Big]$.

\begin{remark}
The Newton's method is utilized to numerically solve the semi-implicit unified SP-PFEM \eqref{eq: full PFEM}. Notably, the nonlinearity is exclusively attributed to $\boldsymbol{n}^{m+\frac{1}{2}}$, a semi-implicit approximation of $\boldsymbol{n}$ introduced in \cite{bao2021structurepreserving}. This smart approximation is crucial for preserving the exact volume conservation at the full-discretized level. The other terms, especially the integration domain $\Gamma^m$, are explicitly defined. As a result, the unified SP-PFEM \eqref{eq: full PFEM} achieves high performance in practical computation.
\end{remark}

\subsection{Main result}
 Suppose the enclosed volume and surface energy for the solution $\Gamma^m=\cup_{j=1}^J \sigma_j^m$ of \eqref{eq: full PFEM} to be $V^m$ and $W^m$, respectively, which are given as
\begin{subequations}\label{eq: full-dis geo}
\begin{align}
\label{eq: full-dis volume}
&V^m:=\frac{1}{d}\left(\boldsymbol{X}^m, \boldsymbol{n}^m\right)_{\Gamma^m}^h=\frac{1}{d^2}\sum_{j=1}^J\sum_{i=1}^d|\sigma_j^m|\,\boldsymbol{q}_{j_i}^m\cdot \boldsymbol{n}_j^m,\\
\label{eq: full-dis energy}
&W^m:=\left(\gamma(\boldsymbol{n}^m), 1\right)^h_{\Gamma^m}=\sum_{j=1}^J |\sigma_j^m|\gamma(\boldsymbol{n}_j^m).
\end{align}
\end{subequations}
Our main result is the structure-preserving property of the unified SP-PFEM \eqref{eq: full PFEM}:

\begin{theorem}[structure-preserving]\label{thm: main}
Consider dimensions $d=2, 3$. For any $\gamma(\boldsymbol{n})$ satisfying \eqref{eq: energy stable condition}, the unified SP-PFEM \eqref{eq: full PFEM} is volume conservative and unconditional energy dissipative with sufficiently large $k(\boldsymbol{n})$, i.e.
\begin{subequations}\label{eq: full-dis geo preserve}
\begin{align}
\label{eq: full-dis volume conservation}
&V^{m+1}=V^m=\ldots=V^0,\\
\label{eq: full-dis energy dissipation}
&W^{m+1}\leq W^m\leq \ldots\leq W^0, \qquad \forall m=0,1,\ldots
\end{align}
\end{subequations}
\end{theorem}

The proof of volume conservation for $d=2$ and $d=3$, analogous to Theorem 2.1 and 3.1 in \cite{bao2021structurepreserving}, is omitted for brevity. However, an in-depth analysis is required for the proof of unconditional energy stability in \eqref{eq: full-dis energy dissipation}, which will be addressed in the following section.

\section{Proof of unconditional energy stability}

\setcounter{equation}{0}

To prove \eqref{eq: full-dis energy dissipation}, it is important to establish the following energy estimate for the energy difference $W^{m+1}-W^m$ between two subsequent time steps:
\begin{equation}\label{eq: energy difference two steps, global}
  \langle \boldsymbol{G}_k(\boldsymbol{n}^m)\nabla_{\Gamma} \boldsymbol{X}^{m+1},\, \nabla_{\Gamma} (\boldsymbol{X}^{m+1}-\boldsymbol{X}^m)\rangle_{\Gamma^m}^h\geq W^{m+1}-W^m.
\end{equation}
We aim to demonstrate that the local version of \eqref{eq: energy difference two steps, global}, applicable between $\sigma_j^m$ and $\sigma_j^{m+1}=\boldsymbol{X}^{m+1}(\sigma_j^m)$, is valid. This concept is illustrated in Figure \ref{fig: illu_local}, where the left and right images represent the 2D and 3D cases, respectively. We name this concept as \textit{local energy estimate}, which is formulated by the following lemma:

\begin{figure}[htp!]
\centering
\includegraphics[width=0.4\textwidth]{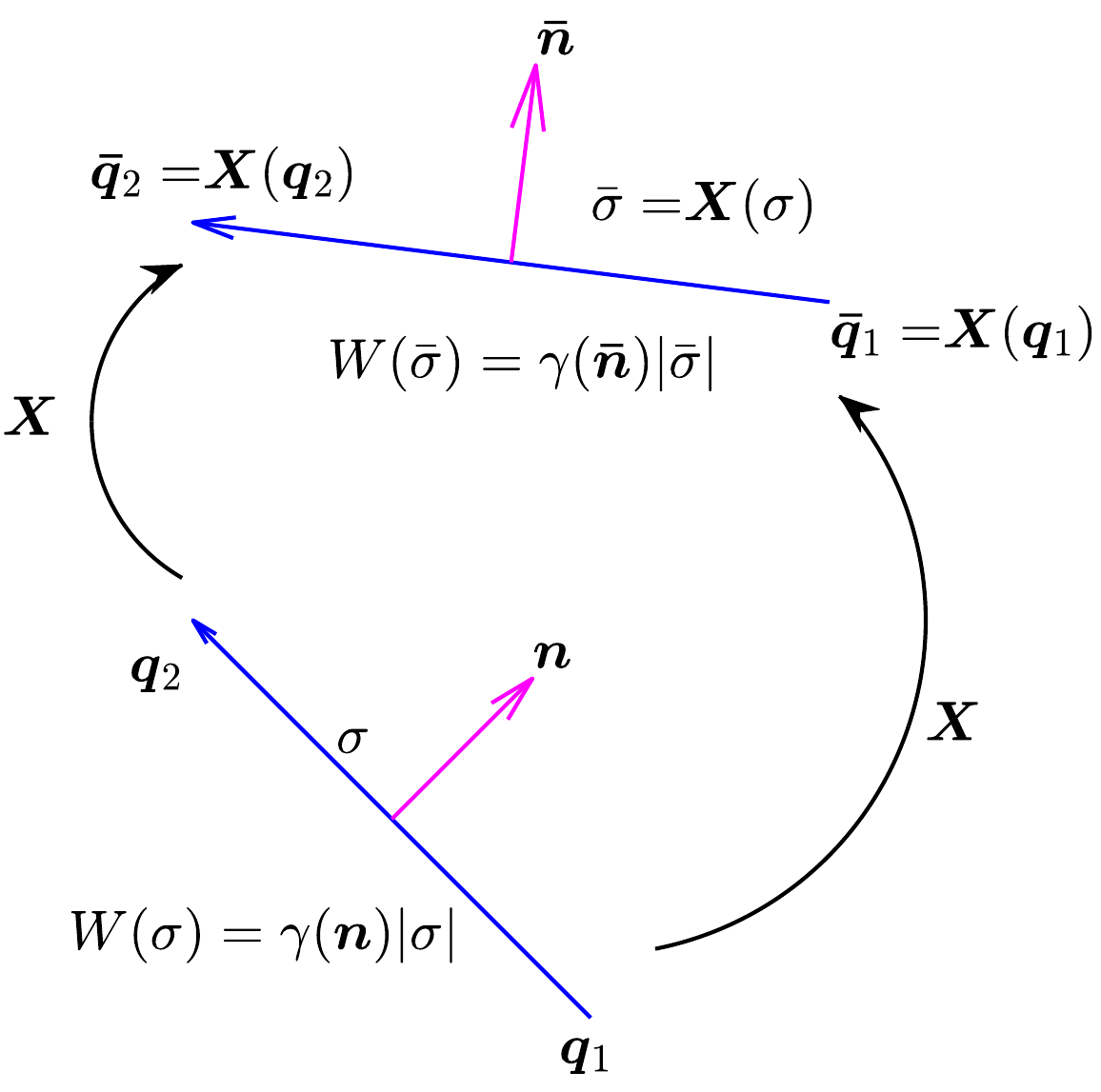}\includegraphics[width=0.6\textwidth]{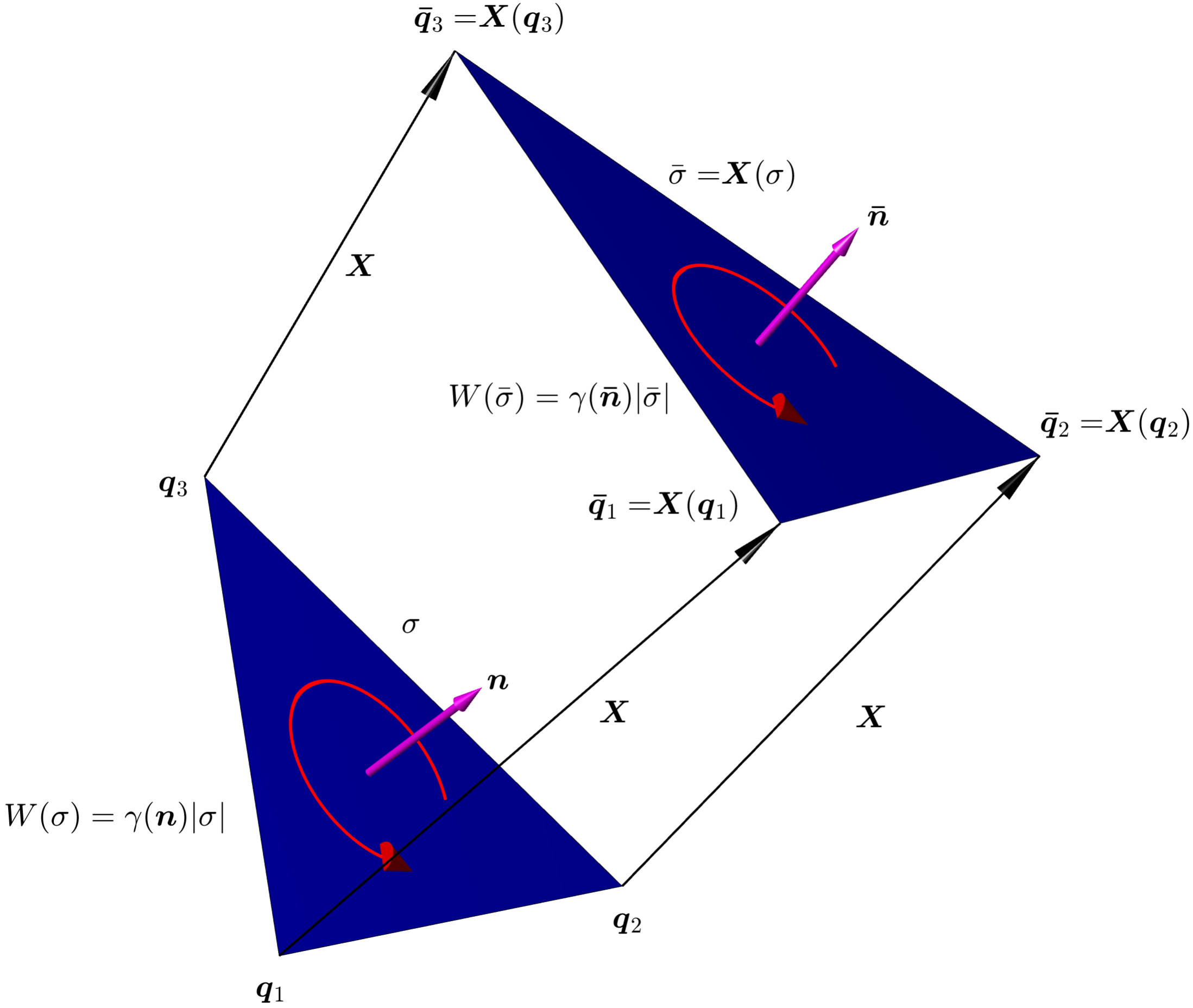}
\caption{Illustration of the local energy estimate Lemma \ref{lem: local est} for 2D (left) and for 3D (right).}
\label{fig: illu_local}
\end{figure}

\begin{lemma}[local energy estimate]\label{lem: local est}
Let $\sigma=[\boldsymbol{q}_1,  \ldots , \boldsymbol{q}_d], \bar{\sigma}=[\bar{\boldsymbol{q}}_1, \ldots, \bar{\boldsymbol{q}}_d]$ be two $(d-1)$-simplices in $\mathbb{R}^d$. Assume that $\boldsymbol{X}: \mathbb{R}^d\to \mathbb{R}^d$ is a continuous differentiable function satisfying 
\begin{equation}
    \boldsymbol{X}(\boldsymbol{q}_i)=\bar{\boldsymbol{q}}_i, \,\forall 1\leq i\leq d, \qquad \boldsymbol{X}|_{\sigma}\in [\mathcal{P}^1(\sigma)]^d.
\end{equation}
Then for dimensions $d=2, 3$ and sufficiently large $k(\boldsymbol{n})$, the following local energy estimate hold
\begin{align}\label{eq: energy difference two steps, local}
  |\sigma| \left(\boldsymbol{G}_k (\boldsymbol{n})\nabla_{\Gamma} \boldsymbol{X}|_{\sigma}\right):(\nabla_{\Gamma}\boldsymbol{X}|_{\sigma}-\nabla_{\Gamma} \textnormal{\textbf{id}}|_{\sigma})\geq \gamma(\bar{\boldsymbol{n}})|\bar{\sigma}|-\gamma(\boldsymbol{n})|\sigma|.
\end{align}
\end{lemma}

\begin{remark}\label{remark: 4.1}
It is worthwhile to mention a necessary condition for the local energy estimate \eqref{eq: energy difference two steps, local}. Let $\boldsymbol{X}=-p\, \textnormal{\textbf{id}}$ with $p>0$ in \eqref{eq: energy difference two steps, global}, then it is easy to verify that $\bar{\boldsymbol{n}}=-\boldsymbol{n}$, $\nabla_\Gamma \boldsymbol{X}=-p\nabla_\Gamma \textnormal{\textbf{id}}$, and $|\bar{\sigma}|=p^{d-1}|\sigma|>0$. In this special case, the local energy estimate \eqref{eq: energy difference two steps, local} gives

\begin{align}\label{eq: energy difference two steps, local necessary}
 |\sigma| \left(\boldsymbol{G}_k (\boldsymbol{n})\nabla_{\Gamma} \boldsymbol{X}|_{\sigma}\right):(\nabla_{\Gamma}\boldsymbol{X}|_{\sigma}-\nabla_{\Gamma} \textnormal{\textbf{id}}|_{\sigma})&\geq \gamma(\bar{\boldsymbol{n}})|\bar{\sigma}|-\gamma(\boldsymbol{n})|\sigma|\nonumber\\
 \Longleftrightarrow \quad p(p+1)|\sigma| \left(\boldsymbol{G}_k (\boldsymbol{n})\nabla_{\Gamma} \textnormal{\textbf{id}}|_{\sigma}\right):\nabla_{\Gamma} \textnormal{\textbf{id}}|_{\sigma}&\geq \gamma(-\boldsymbol{n})p^{d-1}|\bar{\sigma}|-\gamma(\boldsymbol{n})|\sigma|\nonumber\\
  \Longleftrightarrow \qquad \quad\qquad \qquad \quad\,\,(d-1)p(p+1)\gamma(\boldsymbol{n})&\geq \gamma(-\boldsymbol{n})p^{d-1}-\gamma(\boldsymbol{n}).
\end{align}
If $d=2$, by taking $p=1$, \eqref{eq: energy difference two steps, local necessary} implies $\gamma(-\boldsymbol{n})\leq 3\gamma(\boldsymbol{n})$; and if $d=3$, by taking the limit $p\to\infty$, \eqref{eq: energy difference two steps, local necessary} implies $\gamma(-\boldsymbol{n})\leq 2\gamma(\boldsymbol{n})$. Therefore, our energy-stable condition \eqref{eq: energy stable condition} is almost necessary to the local energy estimate!
\end{remark}

\subsection{The minimal stabilizing function}
For any unit normal vector $\boldsymbol{n}\in \mathbb{S}^{d-1}$, it can be assigned with $d-1$ unit vectors $\boldsymbol{\tau}_1, \boldsymbol{\tau}_2, \ldots, \boldsymbol{\tau}_{d-1}\in \mathbb{S}^{d-1}$, such that $\{\boldsymbol{\tau}_1, \boldsymbol{\tau}_2, \ldots, \boldsymbol{\tau}_{d-1}, \boldsymbol{n}\}$ form an orthonormal basis and $\det \begin{bmatrix}\boldsymbol{\tau}_1, \boldsymbol{\tau}_2, \ldots, \boldsymbol{\tau}_{d-1}, \boldsymbol{n}\end{bmatrix}=1$. We thus define two auxiliary $(d-1)\times (d-1)$ matrices $P_{\alpha}(U, \boldsymbol{n}), Q(U, \boldsymbol{n})$ for any $U\in SO(d)$ and $\alpha\in \mathbb{R}_{\geq 0}$ as follows
\begin{eqnarray}
    &&P_{\alpha}(U, \boldsymbol{n}):=\gamma(\boldsymbol{n})I_{_{d-1}}+\Bigl(\alpha(U\boldsymbol{\tau}_i\cdot \boldsymbol{n})(U\boldsymbol{\tau}_j\cdot \boldsymbol{n})\Bigr)_{1\leq i,j\leq d-1},\\
    &&Q(U, \boldsymbol{n}):=\Bigl(\gamma(\boldsymbol{n})(U\boldsymbol{\tau}_i\cdot \boldsymbol{\tau}_j)+(U\boldsymbol{\tau}_i\cdot \boldsymbol{n})(\boldsymbol{\tau}_j\cdot \boldsymbol{\xi})\Bigr)_{1\leq i,j\leq d-1},
\end{eqnarray}
 where $i$ is the row index and $j$ is the column index. Here we adopt $U\boldsymbol{\tau}_i\cdot \boldsymbol{n}:=(U\boldsymbol{\tau}_i)\cdot \boldsymbol{n}$ for simplicity without confusion.

By utilizing the auxiliary matrix $P_{\alpha}, Q$, we thus define the unified minimal stabilizing function $k_0(\boldsymbol{n})$ as
\begin{equation}\label{eq: def of k0}
    k_0(\boldsymbol{n}):=\inf\left\{
\alpha \Biggm| \begin{array}{l}
\text{Tr}\left(L^T(P_{\alpha}(U, \boldsymbol{n})L-Q(U, \boldsymbol{n}))\right)\geq \gamma(U\boldsymbol{n})\prod\limits_{i=1}^{d-1}l_{ii}-\gamma(\boldsymbol{n})\\
\forall U\in SO(d), \forall L=(l_{ij})_{_{1\leq j\leq i\leq d-1}} \text{ lower triangular}, l_{ii}>0
\end{array}
\right\}.
\end{equation}

The following theorem ensures the existence of $k_0(\boldsymbol{n})$.
\begin{theorem}\label{thm: existence of k0}
 For any $\gamma(\boldsymbol{n})$ satisfying \eqref{eq: energy stable condition}, the minimal stabilizing function $k_0(\boldsymbol{n})$, as given in \eqref{eq: def of k0}, is well-defined for $d=2$ and $d=3$.
\end{theorem}

We refer to Section 5.2 for the proof of Theorem \ref{thm: existence of k0} in the case of $d=2$ and to Section 5.3 for $d=3$.

\begin{remark}\label{remark: 2.1}
In fact, similar to the 2D case in \cite{bao2022structure}, the regularity condition in \eqref{eq: energy stable condition} for Theorem 3.1 and Theorem 3.2 can be relaxed to $\gamma(\boldsymbol{p})$ is piecewise $C^2(\mathbb{R}^d_*)$. 
\end{remark}

\subsection{Proof of the local energy estimate} 

To verify the local energy estimate \eqref{eq: energy difference two steps, local}, we need to represent $\nabla_{\Gamma}\boldsymbol{X}|_{\sigma}$ appropriately.

\begin{lemma}\label{lem: QR}
Let $\sigma=[\boldsymbol{q}_1,  \ldots , \boldsymbol{q}_d], \bar{\sigma}=[\bar{\boldsymbol{q}}_1, \ldots, \bar{\boldsymbol{q}}_d]$ be two $(d-1)$-simplices, and $\{\boldsymbol{\tau}_1, \boldsymbol{\tau}_2, \ldots, \boldsymbol{\tau}_{d-1}, \boldsymbol{n}\}$ be an orthonormal basis of $\mathbb{R}^d$ with  $\det \begin{bmatrix}\boldsymbol{\tau}_1,\ldots, \boldsymbol{\tau}_{d-1},$ $\boldsymbol{n}\end{bmatrix}=1$. Then for any continuously differentiable function $\boldsymbol{X}: \mathbb{R}^d\to \mathbb{R}^d, \boldsymbol{X}|_{\sigma}\in [\mathcal{P}^1(\sigma)]^d$ satisfying $\boldsymbol{X}(\boldsymbol{q}_j)=\bar{\boldsymbol{q}}_j,\,\forall 1\leq j\leq d$, there exists a matrix $U\in SO(d)$ and a lower triangular matrix $L=(l_{ij})_{1\leq j\leq i\leq d-1}$ with $l_{ii}>0, \forall 1\leq i\leq d-1$, such that the surface Jacobian $\nabla_{\Gamma}\boldsymbol{X}|_{\sigma}$ is
\begin{equation}\label{eq: QR 2d}
  \nabla_{\Gamma} \boldsymbol{X}|_{\sigma}=U\left(\sum_{1\leq j\leq i\leq d-1}l_{ij}\boldsymbol{\tau}_i\boldsymbol{\tau}_j^T+\boldsymbol{n}\boldsymbol{n}^T\right)\nabla_{\Gamma} \textnormal{\textbf{id}}\, |_{\sigma}.
\end{equation}
Furthermore, the two $(d-1)$-simplices $\sigma$ and $\bar{\sigma}$ are related as:
\begin{equation}\label{eq: QR 2d 21}
  \bar{\boldsymbol{q}}_j-\bar{\boldsymbol{q}}_1 = U\left(\sum_{1\leq j\leq i\leq d-1}l_{ij}\boldsymbol{\tau}_i\boldsymbol{\tau}_j^T+\boldsymbol{n}\boldsymbol{n}^T\right)\left(\boldsymbol{q}_j-\boldsymbol{q}_1\right), \, \forall 1\leq j\leq d,
\end{equation}
and
\begin{equation}\label{eq: QR 2d 22}
     \bar{\boldsymbol{n}}=U\left(\sum_{1\leq j\leq i\leq d-1}l_{ij}\boldsymbol{\tau}_i\boldsymbol{\tau}_j^T+\boldsymbol{n}\boldsymbol{n}^T\right)\boldsymbol{n}=U\boldsymbol{n}.
\end{equation}

\end{lemma}

\begin{proof}
First, we consider the two matrices
\begin{equation*}
     B=\begin{bmatrix}\boldsymbol{q}_2-\boldsymbol{q}_1, \ldots, \boldsymbol{q}_d-\boldsymbol{q}_1, \boldsymbol{n}\end{bmatrix}, \quad \bar{B}=\begin{bmatrix}\bar{\boldsymbol{q}}_2-\bar{\boldsymbol{q}}_1, \ldots, \bar{\boldsymbol{q}}_d-\bar{\boldsymbol{q}}_1, \bar{\boldsymbol{n}}\end{bmatrix}.
 \end{equation*}
By combining the property of wedge product \eqref{eq: wedge product}, \eqref{eq: def of direction d}, and \eqref{eq: normal vector}, we derive that
  \begin{equation*}
       \det B=(\boldsymbol{q}_{2}-\boldsymbol{q}_{1})\wedge\cdots\wedge (\boldsymbol{q}_{d}-\boldsymbol{q}_{1})\cdot \boldsymbol{n}=\mathcal{J}\{\sigma\}\cdot \boldsymbol{n}=(d-1)|\sigma|>0,
   \end{equation*} 
   and thus $B$ is invertible. Similarly, we have $\det \bar{B}>0$. Let $A=\bar{B}B^{-1}$, we know that $\det A=\det \bar{B}\left(\det B\right)^{-1}>0$, and
   \begin{equation}\label{eq: property of A}
       A(\boldsymbol{q}_j-\boldsymbol{q}_1)=\bar{\boldsymbol{q}}_j-\bar{\boldsymbol{q}}_1,\,\ \forall 1\leq j\leq d, \qquad A\boldsymbol{n}=\bar{\boldsymbol{n}}.
   \end{equation} Moreover, let $\boldsymbol{b}= \bar{\boldsymbol{q}}_1-A\boldsymbol{q}_1$, we obtain
   \begin{align*}
       \boldsymbol{X}(\boldsymbol{q}_j)-\left(A\boldsymbol{q}_j+\boldsymbol{b}\right)&=\bar{\boldsymbol{q}}_j-\left(A(\boldsymbol{q}_j-\boldsymbol{q}_1)+\left(A\boldsymbol{q}_1+\boldsymbol{b}\right)\right)\\
       &=\bar{\boldsymbol{q}}_j-\left(\bar{\boldsymbol{q}}_j-\bar{\boldsymbol{q}}_1+\bar{\boldsymbol{q}}_1\right)\\
       &=\boldsymbol{0}, \quad \forall 1\leq j\leq d.
   \end{align*} 
   By noting the definition of surface Jacobian for vector-valued function \eqref{eq: def of surface jacob dis}, the surface gradient \eqref{eq: def of surface gradient dis 1 2d} in 2D and \eqref{eq: def of surface gradient dis 1} in 3D, we know that the surface Jacobian of $\boldsymbol{X}$ can be formulated by $A$ as
 \begin{equation}\label{eq: affine res}
     \nabla_{\Gamma} \boldsymbol{X}|_{\sigma}=\nabla_{\Gamma}\left(A\,\textnormal{\textbf{id}}+\boldsymbol{b}\right)|_{\sigma}=A\nabla_{\Gamma}\textnormal{\textbf{id}}|_{\sigma}.
 \end{equation}

Next, since $\{\boldsymbol{\tau}_1, \boldsymbol{\tau}_2, \ldots, \boldsymbol{\tau}_{d-1}, \boldsymbol{n}\}$ is an orthonormal basis of $\mathbb{R}^d$ with \\ $\det \begin{bmatrix}\boldsymbol{\tau}_1,\ldots, \boldsymbol{\tau}_{d-1}, \boldsymbol{n}\end{bmatrix}=1$, the matrix $\begin{bmatrix}\boldsymbol{\tau}_1,\ldots, \boldsymbol{\tau}_{d-1}, \boldsymbol{n}\end{bmatrix}\in SO(d)$ is an orthogonal matrix. By adopting QR factorization for the matrix $A\begin{bmatrix}\boldsymbol{\tau}_1,\ldots, \boldsymbol{\tau}_{d-1}, \boldsymbol{n}\end{bmatrix}$, we know there exists an orthogonal matrix $Q$ and a lower triangular matrix $\hat{L}=(l_{ij})_{1\leq j\leq i\leq d}$ with $l_{ii}>0, \,\forall 1\leq i\leq d$, such that
\begin{align}\label{eq: QR for Ax}
    A&=\left(A\begin{bmatrix}\boldsymbol{\tau}_1,\ldots, \boldsymbol{\tau}_{d-1}, \boldsymbol{n}\end{bmatrix}\right)\begin{bmatrix}\boldsymbol{\tau}_1,\ldots, \boldsymbol{\tau}_{d-1}, \boldsymbol{n}\end{bmatrix}^T\nonumber\\
  &=Q\hat{L}\begin{bmatrix}\boldsymbol{\tau}_1,\ldots, \boldsymbol{\tau}_{d-1}, \boldsymbol{n}\end{bmatrix}^T\nonumber\\
  &=\left(Q\begin{bmatrix}\boldsymbol{\tau}_1,\ldots, \boldsymbol{\tau}_{d-1}, \boldsymbol{n}\end{bmatrix}^T\right)\left(\begin{bmatrix}\boldsymbol{\tau}_1,\ldots, \boldsymbol{\tau}_{d-1}, \boldsymbol{n}\end{bmatrix}\hat{L}\begin{bmatrix}\boldsymbol{\tau}_1,\ldots, \boldsymbol{\tau}_{d-1}, \boldsymbol{n}\end{bmatrix}^T\right)\nonumber\\
  &=U\left(\sum_{1\leq j\leq i\leq d-1}l_{ij}\boldsymbol{\tau}_i\boldsymbol{\tau}_j^T+\sum_{j=1}^{d-1}l_{dj}\boldsymbol{n}\boldsymbol{\tau}_j^T+l_{dd}\boldsymbol{n}\boldsymbol{n}^T\right),
\end{align}
Here $U=\left(Q\begin{bmatrix}\boldsymbol{\tau}_1,\ldots, \boldsymbol{\tau}_{d-1}, \boldsymbol{n}\end{bmatrix}^T\right)$. We know $U$ is orthogonal from the fact $Q$ and $\begin{bmatrix}\boldsymbol{\tau}_1,\ldots, \boldsymbol{\tau}_{d-1}, \boldsymbol{n}\end{bmatrix}$ are orthogonal. Moreover, by taking the determinant of each side of \eqref{eq: QR for Ax} and noticing the fact $\det A>0$, we know that $\det{U}=\frac{\det{A}}{\det{\hat{L}}}=\frac{\det{A}}{l_{11}\ldots l_{dd}}>0$. Therefore, we conclude that $U\in SO(d)$.

Comparing \eqref{eq: QR for Ax} with the desired identity \eqref{eq: QR 2d} and noting \eqref{eq: affine res}, it suffices to show that $l_{dd}=1$ and $l_{dj}=0, \,\forall 1\leq j\leq d-1$. From \eqref{eq: property of A}, $A\boldsymbol{n}=\bar{\boldsymbol{n}}$ is a unit vector, \eqref{eq: QR for Ax} together with $U\in SO(d)$ yields that
\begin{equation*}
    1=\bar{\boldsymbol{n}}\cdot \bar{\boldsymbol{n}}=A\boldsymbol{n}\cdot A\boldsymbol{n}=\left(U(l_{dd}\, \boldsymbol{n})\right)\cdot \left(U (l_{dd}\, \boldsymbol{n})\right)=l_{dd}^2.
\end{equation*}
We know that $l_{dd}=1$ since $l_{dd}>0$. 

To see $l_{dj}=0$, we notice that $\det B>0$ implies that $\boldsymbol{\tau}_j\in \operatorname{span}\{\boldsymbol{q}_2-\boldsymbol{q}_1, \ldots, \boldsymbol{q}_{d}-\boldsymbol{q}_1, \boldsymbol{n}\}=\mathbb{R}^d$. Moreover, since $\boldsymbol{\tau}_j\cdot \boldsymbol{n}=0$, we further deduce that $\boldsymbol{\tau}_j\in \operatorname{span}\{\boldsymbol{q}_2-\boldsymbol{q}_1, \ldots, \boldsymbol{q}_{d}-\boldsymbol{q}_1\}$. This and \eqref{eq: property of A} conclude that $A\boldsymbol{\tau}_j \in \operatorname{span}\{A(\boldsymbol{q}_2-\boldsymbol{q}_1), \ldots, A(\boldsymbol{q}_{d}-\boldsymbol{q}_1)\}= \operatorname{span}\{\bar{\boldsymbol{q}}_2-\bar{\boldsymbol{q}}_1, \ldots, \bar{\boldsymbol{q}}_{d}-\bar{\boldsymbol{q}}_1\}$. Therefore, together with \eqref{eq: QR for Ax} and the identity $\bar{\boldsymbol{n}}=A\boldsymbol{n}$, we have
\begin{align*}
    0=A\boldsymbol{\tau}_j\cdot \bar{\boldsymbol{n}}=A\boldsymbol{\tau}_j\cdot A\boldsymbol{n}=\left(U\left(\sum_{i=j}^{d-1}l_{ij}\boldsymbol{\tau}_i+l_{dj}\boldsymbol{n}\right)\right)\cdot \left(U(l_{dd}\, \boldsymbol{n})\right)=l_{dj}\,l_{dd}.
\end{align*}
Thus $l_{dj}=0, \forall 1\leq j\leq d-1$. Therefore, \eqref{eq: QR for Ax} can be further simplified as
\begin{equation}\label{eq: QR for A final}
    A=U\left(\sum_{1\leq j\leq i\leq d-1}l_{ij}\boldsymbol{\tau}_i\boldsymbol{\tau}_j^T+\boldsymbol{n}\boldsymbol{n}^T\right),
\end{equation}
\eqref{eq: QR 2d}, \eqref{eq: QR 2d 21} and \eqref{eq: QR 2d 22} are the direct results of \eqref{eq: property of A}, \eqref{eq: affine res} and \eqref{eq: QR for A final}.  
\end{proof}

By using the representation of $\nabla_{\Gamma} \boldsymbol{X}|_{\sigma}$ given in Lemma \ref{lem: QR}, we can finally establish the local energy estimate.

{\em Proof of the local energy estimate, Lemma \ref{lem: local est}}\quad
First we take $\boldsymbol{\tau}_1, \ldots, \boldsymbol{\tau}_{d-1}$ as in the definition of $P_{\alpha}(U, \boldsymbol{n})$ and $Q(U, \boldsymbol{n})$. We know that $\{\boldsymbol{\tau}_1, \ldots, \boldsymbol{\tau}_{d-1}, \boldsymbol{n}\}$ form an orthonormal basis and $\det \begin{bmatrix}\boldsymbol{\tau}_1,\ldots,  \boldsymbol{\tau}_{d-1}, \boldsymbol{n}\end{bmatrix}=1$. Therefore, from Lemma \ref{lem: QR}, there exists a matrix $U\in SO(d)$ and a lower triangular matrix $L=(l_{ij})_{1\leq j\leq i\leq d-1}$ with $l_{ii}>0, \forall 1\leq i\leq d-1$, such that
\begin{equation}\label{eq: tmp xxad}
  \nabla_{\Gamma} \boldsymbol{X}|_{\sigma}=U\left(\sum_{1\leq j\leq i\leq d-1}l_{ij}\boldsymbol{\tau}_i\boldsymbol{\tau}_j^T+\boldsymbol{n}\boldsymbol{n}^T\right)\nabla_{\Gamma} \textnormal{\textbf{id}}\, |_{\sigma}.
\end{equation}

By using Lemma 3.6 in \cite{bao2022symmetrized} and Lemma 9 (i) in \cite{Barrett2020}, we know that
\begin{equation}\label{eq: nabla id 2d}
  \nabla_{\Gamma} \textnormal{\textbf{id}}|_{\sigma}=I_d-\boldsymbol{n}\boldsymbol{n}^T=\sum_{i=1}^{d-1}\boldsymbol{\tau}_i\boldsymbol{\tau}_i^T.
\end{equation}
This, together with \eqref{eq: tmp xxad} yields that
\begin{align}\label{eq: nabla X 2d}
  \nabla_{\Gamma}\boldsymbol{X}|_{\sigma}&=U\left(\sum_{1\leq j\leq i\leq d-1}l_{ij}\boldsymbol{\tau}_i\boldsymbol{\tau}_j^T+\boldsymbol{n}\boldsymbol{n}^T\right)(I_d-\boldsymbol{n}\boldsymbol{n}^T)\nonumber\\
  &=\sum_{1\leq j\leq i\leq d-1}l_{ij}(U\boldsymbol{\tau}_i)\boldsymbol{\tau}_j^T.
\end{align}

The left-hand side of \eqref{eq: energy difference two steps, local} is composed of the subtraction of two components. Utilizing \eqref{eq: nabla X 2d} and the definition of $\boldsymbol{G}_k(\boldsymbol{n})$ as given in \eqref{eq: def of Gk}, and using the fact \eqref{eq: Gs and Ga} that $\boldsymbol{G}_k(\boldsymbol{n})$ is the sum of its symmetric part $\boldsymbol{G}_k^{(s)}$ and anti-symmetric part $\boldsymbol{G}^{(a)}$, we can simplify the first component as follows
\begin{align}\label{eq: lhs 1 est 2d}
&|\sigma| \left(\boldsymbol{G}_k (\boldsymbol{n})\nabla_{\Gamma} \boldsymbol{X}|_{\sigma}\right):(\nabla_{\Gamma}\boldsymbol{X}|_{\sigma})\nonumber\\
&=|\sigma|\,\left(\boldsymbol{G}_k^{(s)}(\boldsymbol{n}) \nabla_{\Gamma} \boldsymbol{X}|_{\sigma}\right):(\nabla_{\Gamma} \boldsymbol{X}|_{\sigma})\nonumber\\
&=|\sigma|\,\text{Tr}\left[\left(\sum_{1\leq q\leq p\leq d-1}l_{pq}\boldsymbol{\tau}_q(U\boldsymbol{\tau}_p)^T\right)(\gamma(\boldsymbol{n})I_d+k(\boldsymbol{n})\boldsymbol{n}\boldsymbol{n}^T)\right.\nonumber\\
&\qquad\qquad\qquad\left. \left(\sum_{1\leq j\leq i\leq d-1}l_{ij}(U\boldsymbol{\tau}_i)\boldsymbol{\tau}_j^T\right)\right]\nonumber\\
&=|\sigma|\left[\gamma(\boldsymbol{n})\sum_{1\leq j\leq i\leq d-1}l_{ij}^2+k(\boldsymbol{n})\sum_{1\leq j\leq i, p\leq d-1}l_{ij}l_{pj}(U\boldsymbol{\tau}_i\cdot \boldsymbol{n})(U\boldsymbol{\tau}_p\cdot \boldsymbol{n})\right]\nonumber\\
&=|\sigma|\text{Tr}\left(L^T(P_{k(\boldsymbol{n})}(U, \boldsymbol{n})L)\right).
\end{align}
The second component is detailed as:
\begin{align}\label{eq: lhs 2 est 2d}
&|\sigma| \left(\boldsymbol{G}_k (\boldsymbol{n})\nabla_{\Gamma} \boldsymbol{X}|_{\sigma}\right):(\nabla_{\Gamma}\textnormal{\textbf{id}}|_{\sigma}) \nonumber\\
&=|\sigma|\,\text{Tr}\left[(I_d-\boldsymbol{n}\boldsymbol{n}^T)(\gamma(\boldsymbol{n})I_d-\boldsymbol{n}\boldsymbol{\xi}^T+\boldsymbol{\xi}\boldsymbol{n}^T+k(\boldsymbol{n})\boldsymbol{n}\boldsymbol{n}^T)\right.\nonumber\\
&\qquad\qquad \qquad \left. \left(\sum_{1\leq j\leq i\leq d-1}l_{ij}(U\boldsymbol{\tau}_i)\boldsymbol{\tau}_j^T\right)\right]\nonumber\\
&=|\sigma|\,\text{Tr}\left[\left(\sum_{i=1}^{d-1}\boldsymbol{\tau}_i\boldsymbol{\tau}_i^T\right)(\gamma(\boldsymbol{n})I_d+\boldsymbol{\xi}\boldsymbol{n}^T) \left(\sum_{1\leq j\leq i\leq d-1}l_{ij}(U\boldsymbol{\tau}_i)\boldsymbol{\tau}_j^T\right)\right]\nonumber\\
&=|\sigma|\left[\sum_{1\leq j\leq i\leq d-1} l_{ij}\Bigl(\gamma(\boldsymbol{n})(U\boldsymbol{\tau}_i\cdot \boldsymbol{\tau}_j)+(U\boldsymbol{\tau}_i\cdot \boldsymbol{n})(\boldsymbol{\tau}_j\cdot \boldsymbol{\xi})\Bigr)\right]\nonumber\\
&=|\sigma|\text{Tr}\left(L^TQ(U, \boldsymbol{n})\right).
\end{align}

For the right-hand side of \eqref{eq: energy difference two steps, local}, $\gamma(\bar{\boldsymbol{n}})|\bar{\sigma}|-\gamma(\boldsymbol{n})|\sigma|$, it suffices to deal with $\gamma(\bar{\boldsymbol{n}})|\bar{\sigma}|$. We have already known that $\bar{n}=U\boldsymbol{n}$ by Lemma \ref{lem: QR}. From \eqref{eq: def of direction d}, \eqref{eq: wedge product} and \eqref{eq: normal vector}, we deduce that 
\begin{align}\label{eq: rhs est 2d 1}
  \gamma (\bar{\boldsymbol{n}})\,|\bar{\sigma}|&=\gamma (U\boldsymbol{n})\,\frac{1}{d-1} \mathcal{J}\{\bar{\sigma}\}\cdot \boldsymbol{\bar{n}}\nonumber\\
  &=\gamma (U\boldsymbol{n})\frac{1}{d-1}\det [\bar{\boldsymbol{q}}_{2}-\bar{\boldsymbol{q}}_{1}, \ldots, \bar{\boldsymbol{q}}_{d}-\bar{\boldsymbol{q}}_{1}, \boldsymbol{\bar{n}}]\nonumber\\
  &=\gamma(U\boldsymbol{n})\frac{1}{d-1}\det \left(U\left(\sum_{1\leq j\leq i\leq d-1}l_{ij}\boldsymbol{\tau}_i\boldsymbol{\tau}_j^T+\boldsymbol{n}\boldsymbol{n}^T\right)\right)\nonumber\\
  &\quad\times \det [\boldsymbol{q}_2-\boldsymbol{q}_1, \ldots, \boldsymbol{q}_d-\boldsymbol{q}_1, \boldsymbol{n}]\nonumber\\
  &=\gamma(U\boldsymbol{n})|\sigma|\det \left(\sum_{1\leq j\leq i\leq d-1}l_{ij}\boldsymbol{\tau}_i\boldsymbol{\tau}_j^T+\boldsymbol{n}\boldsymbol{n}^T\right).
\end{align}
Here we use the identity $\bar{\boldsymbol{q}}_j-\bar{\boldsymbol{q}}_1=A(\boldsymbol{q}_j-\boldsymbol{q}_1), \forall 1\leq j\leq d$ from \eqref{eq: QR 2d 21}, and the fact $\det [\bar{\boldsymbol{q}}_{2}-\bar{\boldsymbol{q}}_{1}, \ldots, \bar{\boldsymbol{q}}_{d}-\bar{\boldsymbol{q}}_{1}, \boldsymbol{\bar{n}}]=\det [A(\boldsymbol{q}_{2}-\boldsymbol{q}_1), \ldots, A(\boldsymbol{q}_d-\boldsymbol{q}_1), A\boldsymbol{n}]=\det A \, \det [\boldsymbol{q}_{j_2}-\boldsymbol{q}_{j_1}, \ldots, \boldsymbol{q}_{j_d}-\boldsymbol{q}_{j_1}, \boldsymbol{n}]$, where $A=U\left(\sum\limits_{1\leq j\leq i\leq d-1}l_{ij}\boldsymbol{\tau}_i\boldsymbol{\tau}_j^T+\boldsymbol{n}\boldsymbol{n}^T\right)$.
Furthermore, we observe that
\begin{equation*}
    \sum_{1\leq j\leq i\leq d-1}l_{ij}\boldsymbol{\tau}_i\boldsymbol{\tau}_j^T+\boldsymbol{n}\boldsymbol{n}^T=\begin{bmatrix}\boldsymbol{\tau}_1,\ldots, \boldsymbol{\tau}_{d-1}, \boldsymbol{n}\end{bmatrix} \begin{bmatrix}l_{11} & &\\  \vdots &\ddots&\\l_{d1}&\ldots&l_{dd} \end{bmatrix}\begin{bmatrix}\boldsymbol{\tau}_1,\ldots, \boldsymbol{\tau}_{d-1}, \boldsymbol{n}\end{bmatrix}^T.
\end{equation*}
Here $l_{di}=0, \, \forall 1\leq i\leq d-1$ and $l_{dd}=1$.
Therefore, \eqref{eq: rhs est 2d 1} can be further simplified as
\begin{equation}\label{eq: rhs est 2d}
    \gamma (\bar{\boldsymbol{n}})\,|\bar{\sigma}|=\gamma(U\boldsymbol{n})|\sigma|\prod\limits_{i=1}^{d-1}l_{ii}.
\end{equation}
Finally, by substituting \eqref{eq: lhs 1 est 2d}, \eqref{eq: lhs 2 est 2d}, \eqref{eq: rhs est 2d} into the local energy estimate \eqref{eq: energy difference two steps, local}, we deduce that the local energy estimate \eqref{eq: energy difference two steps, local} is equivalent to
\begin{align}\label{eq: aux final 2d}
&|\sigma| \left(\boldsymbol{G}_k (\boldsymbol{n})\nabla_{\Gamma} \boldsymbol{X}|_{\sigma}\right):(\nabla_{\Gamma}\boldsymbol{X}|_{\sigma}-\nabla_{\Gamma} \textnormal{\textbf{id}}|_{\sigma})- \left(\gamma(\bar{\boldsymbol{n}})|\bar{\sigma}|-\gamma(\boldsymbol{n})|\sigma|\right)\nonumber\\
&=|\sigma|\left(\text{Tr}\left(L^T(P_{k(\boldsymbol{n})}(U, \boldsymbol{n})L-Q(U, \boldsymbol{n}))\right)- \left(\gamma(U\boldsymbol{n})\prod\limits_{i=1}^{d-1}l_{ii}-\gamma(\boldsymbol{n})\right)\right)\nonumber\\
  &\geq 0.
\end{align}
From the unified definition of $k_0(\boldsymbol{n})$ \eqref{eq: def of k0}, we know that $\gamma(U\boldsymbol{n})\prod\limits_{i=1}^{d-1}l_{ii}-\gamma(\boldsymbol{n})$ $\le \text{Tr}\left(L^T(P_{k(\boldsymbol{n})}(U, \boldsymbol{n})L-Q(U, \boldsymbol{n}))\right)$ for all $U\in SO(d)$, lower triangular matrix $L=(l_{ij})_{1\leq j\leq i\leq d-1}$ with $l_{ii}>0, \forall 1\leq i\leq d-1$, and $k(\boldsymbol{n})\geq k_0(\boldsymbol{n})$. Theorem \ref{thm: existence of k0} indicates that for dimensions $d=2, 3$, the unified minimal stabilizing function $k_0(\boldsymbol{n})<\infty$ is well-defined. Therefore, we can choose sufficient large $k(\boldsymbol{n})$ satisfying $k(\boldsymbol{n})\ge k_0(\boldsymbol{n})$ such that the desired local energy estimate \eqref{eq: energy difference two steps, local} is validated.
\endproof

\subsection{Proof of main result}
With the help of the local energy estimate \eqref{eq: energy difference two steps, local} in Lemma \ref{lem: local est}, we are finally able to finish the unconditional energy stability part \eqref{eq: full-dis energy dissipation} of the main result \ref{thm: main}.

{\em Proof of unconditional energy stability}. \quad Suppose $k(\boldsymbol{n})$ is sufficiently large, such that $k(\boldsymbol{n})\geq k_0(\boldsymbol{n})$, and the local energy estimate \eqref{eq: energy difference two steps, local} holds. For each $1\leq j\leq J$, we apply Lemma \ref{lem: local est} for $\sigma=[\boldsymbol{q}_{j_1}^m, \ldots, \boldsymbol{q}_{j_d}^m]$, $\bar{\sigma}=[\boldsymbol{q}_{j_1}^{m+1},\ldots, \boldsymbol{q}_{j_d}^{m+1}]$, and $\boldsymbol{X}=\boldsymbol{X}^{m+1}$. Consequently, the local energy estimate \eqref{eq: energy difference two steps, local} gives
\begin{align*}
  &|\sigma_j^m| \left(\boldsymbol{G}_k (\boldsymbol{n}_j^m)\nabla_{\Gamma} \boldsymbol{X}^{m+1}|_{\sigma_j^m}\right):(\nabla_{\Gamma}\boldsymbol{X}^{m+1}|_{\sigma_j^m}-\nabla_{\Gamma} \boldsymbol{X}^m|_{\sigma_j^m})\\
  &\geq \gamma(\boldsymbol{n}_j^{m+1})|\sigma_j^{m+1}|-\gamma(\boldsymbol{n}_j^m)|\sigma_j^m|.
\end{align*} 
By taking summation of this inequality for $j$ from $1$ to $J$, and applying the mass-lumped inner product \eqref{eq: mass lumped inner product, matrix} and the definition for $W^m$ \eqref{eq: full-dis energy}, we get
\begin{align}\label{eq: apply lem}
  &\langle \boldsymbol{G}_k(\boldsymbol{n}^m)\nabla_{\Gamma} \boldsymbol{X}^{m+1},\, \nabla_{\Gamma} (\boldsymbol{X}^{m+1}-\boldsymbol{X}^m)\rangle_{\Gamma^m}^h\nonumber\\
  &=\frac{1}{d}\sum_{j=1}^J|\sigma_j^m|\sum_{i=1}^d \Bigl(\boldsymbol{G}_k (\boldsymbol{n}^m_j)\nabla_{\Gamma} \boldsymbol{X}^{m+1} ( (\boldsymbol{q}_{j_i}^m)^-))\Bigr.\nonumber\\
  &\qquad\qquad\qquad\quad\Bigl.:\nabla_{\Gamma}\boldsymbol{X}^{m+1} ( (\boldsymbol{q}_{j_i}^m)^-)-\nabla_{\Gamma} \boldsymbol{X}^{m}((\boldsymbol{q}_{j_i}^m)^-))\Bigr)\nonumber\\
  &=\sum_{j=1}^J|\sigma_j^m| \left(\boldsymbol{G}_k (\boldsymbol{n}_j^m)\nabla_{\Gamma} \boldsymbol{X}^{m+1}|_{\sigma_j^m}\right):(\nabla_{\Gamma}\boldsymbol{X}^{m+1}|_{\sigma_j^m}-\nabla_{\Gamma} \boldsymbol{X}^m|_{\sigma_j^m})\nonumber\\
  &\geq \sum_{j=1}^J\left(\gamma(\boldsymbol{n}^{m+1}_j)|\sigma_j^{m+1}|-\gamma(\boldsymbol{n}^m_j)|\sigma_j^m|\right)=W^{m+1}-W^m, \qquad m\ge0.
\end{align}
Choosing $\phi=\mu^{m+1}$ in \eqref{eq: full PFEM 1} and $\boldsymbol{\omega}=\boldsymbol{X}^{m+1}-\boldsymbol{X}^m$ in \eqref{eq: full PFEM 2}, together with \eqref{eq: apply lem} yields that
\begin{align}
W^{m+1}-W^m&\leq \langle \boldsymbol{G}_k (\boldsymbol{n}^m)\nabla_{\Gamma}\boldsymbol{X}^{m+1}, \nabla_{\Gamma} (\boldsymbol{X}^{m+1}-\boldsymbol{X}^m)\rangle_{\Gamma^m}^h\nonumber\\
&=\left(\mu^{m+1}\boldsymbol{n}^{m+\frac{1}{2}}, \boldsymbol{X}^{m+1}-\boldsymbol{X}^m\right)_{\Gamma^m}^h\nonumber\\
&=-\tau\left(\nabla_{\Gamma}\mu^{m+1}, \nabla_{\Gamma}\mu^{m+1}\right)_{\Gamma^m}^h\leq 0,\qquad m\ge0,
\end{align}
which validates the unconditional energy stability
\eqref{eq: full-dis energy dissipation} in Theorem \ref{thm: main}.
\endproof

\section{Existence of the minimal stabilizing function}

\setcounter{equation}{0}

In this section, we first reduce the existence of the minimal stabilizing function $k_0(\boldsymbol{n})$ for dimensions $d=2, 3$ to the positive semi-definiteness of an auxiliary matrix. For any unit normal vector $\boldsymbol{n}\in \mathbb{S}^{d-1}$, we take $\boldsymbol{\tau}_1, \ldots, \boldsymbol{\tau}_{d-1}$ as in the definition of $P_{\alpha}(U, \boldsymbol{n})$ and $Q(U, \boldsymbol{n})$. 

When $d=2$, we define the auxiliary $2\times 2$ symmetric matrix $\tilde{M}(U, \alpha)$ for any $U\in SO(2)$ and $\alpha\in \mathbb{R}_{\geq 0}$ as follows
\begin{equation}\label{eq: def of M 2d}
  \tilde{M}(U, \alpha):=\begin{bmatrix}
  \gamma(\boldsymbol{n})+\alpha(U\boldsymbol{\tau}_1\cdot \boldsymbol{n})^2 &*\\
    -\frac{1}{2}\left(\gamma(\boldsymbol{n})(U\boldsymbol{\tau}_1\cdot \boldsymbol{\tau}_1)+(U\boldsymbol{\tau}_1\cdot \boldsymbol{n})(\boldsymbol{\tau}_1\cdot \boldsymbol{\xi})+\gamma(U\boldsymbol{n})\right)&\gamma(\boldsymbol{n})
  \end{bmatrix},
\end{equation}
here the entries above the main diagonal are abbreviated to $*$ since $\tilde{M}(U, \alpha)$ is symmetric.

It is straightforward to check that for $d=2$, it holds that
\begin{align*}
    &\text{Tr}\left(L^T(P_{\alpha}(U, \boldsymbol{n})L-Q(U, \boldsymbol{n}))\right)\geq \gamma(U\boldsymbol{n})\prod\limits_{i=1}^{d-1}l_{ii}-\gamma(\boldsymbol{n})\\
    \iff & \begin{bmatrix} l_{11}& 1\end{bmatrix} \tilde{M}(U, \alpha) \begin{bmatrix} l_{11}& 1\end{bmatrix}^T \geq 0.
\end{align*}

Therefore, by utilizing the auxiliary matrix $\tilde{M}(U, \alpha)$, the unified definition of the minimal stabilizing function $k_0(\boldsymbol{n})$ \eqref{eq: def of k0} is equivalent to
\begin{equation}\label{eq: def of k0 2d}
  k_0(\boldsymbol{n}):=\inf\Big\{\alpha\Big|\,\,\tilde{M}(U, \alpha) \text{ is positive semi-definite }\quad \forall U\in SO(2)\Big\}.
\end{equation}

When $d=3$, we define the auxiliary $4\times 4$ symmetric matrix $M(U, \alpha)$ for any $U\in SO(3)$ and $\alpha\in \mathbb{R}_{\geq 0}$ as follows
\begin{equation}\label{eq: def of M}
  \begin{bmatrix}
  \gamma(\boldsymbol{n})+\alpha(U\boldsymbol{\tau}_1\cdot \boldsymbol{n})^2 &*&*&*\\
  -\frac{1}{2}\gamma(U\boldsymbol{n})&\gamma(\boldsymbol{n})+\alpha(U\boldsymbol{\tau}_2\cdot \boldsymbol{n})^2&*&*\\
  \alpha(U\boldsymbol{\tau}_1\cdot \boldsymbol{n})(U\boldsymbol{\tau}_2\cdot \boldsymbol{n})&0&\gamma(\boldsymbol{n})+\alpha(U\boldsymbol{\tau}_2\cdot \boldsymbol{n})^2&*\\
  M_{41}&M_{42}&M_{43}&\gamma(\boldsymbol{n})
  \end{bmatrix},
\end{equation}
and $M_{41}, M_{42}, M_{43}$ are 
\begin{subequations}\label{eq: def of Ms}
\begin{align}
\label{eq: def of M31}
  &M_{41}=-\frac{1}{2}(\gamma(\boldsymbol{n})(U\boldsymbol{\tau}_1\cdot \boldsymbol{\tau}_1)+(U\boldsymbol{\tau}_1\cdot \boldsymbol{n})(\boldsymbol{\tau}_1\cdot \boldsymbol{\xi})), \\
\label{eq: def of M32}
  &M_{42}=-\frac{1}{2}(\gamma(\boldsymbol{n})(U\boldsymbol{\tau}_2\cdot \boldsymbol{\tau}_2)+(U\boldsymbol{\tau}_2\cdot \boldsymbol{n})(\boldsymbol{\tau}_2\cdot \boldsymbol{\xi})),\\
\label{eq: def of M43}
  &M_{43}=-\frac{1}{2}(\gamma(\boldsymbol{n})(U\boldsymbol{\tau}_2\cdot \boldsymbol{\tau}_1)+(U\boldsymbol{\tau}_2\cdot \boldsymbol{n})(\boldsymbol{\tau}_1\cdot \boldsymbol{\xi})).
\end{align}
\end{subequations}

Similarly, it is straightforward to check that for $d=3$, it holds that
\begin{align*}
    &\text{Tr}\left(L^T(P_{\alpha}(U, \boldsymbol{n})L-Q(U, \boldsymbol{n}))\right)\geq \gamma(U\boldsymbol{n})\prod\limits_{i=1}^{d-1}l_{ii}-\gamma(\boldsymbol{n})\\
    \iff & \begin{bmatrix} l_{11}& l_{22}&l_{21}& 1\end{bmatrix} M(U, \alpha)  \begin{bmatrix} l_{11}& l_{22}&l_{21}& 1\end{bmatrix}^T \geq 0.
\end{align*}
The unified definition of the minimal stabilizing function $k_0(\boldsymbol{n})$ \eqref{eq: def of k0} is equivalent to
\begin{equation}\label{eq: def of k0 3d}
  k_0(\boldsymbol{n}):=\inf\Big\{\alpha\Big|\,\,M(U, \alpha) \text{ is positive semi-definite }\quad \forall U\in SO(3)\Big\}.
\end{equation}

We then propose a unified approach for showing the positive semi-definiteness. And after that, we adopt this unified approach to show the positive semi-definiteness of $\tilde{M}$ and $M$ for dimensions $d=2$ and $d=3$, respectively.

\subsection{A unified approach}
\begin{lemma}\label{lem: compactness}
Let $A:\ SO(d)\times \mathbb{R}\to \mathbb{R}^{m\times m}$ and $D:\ SO(d)\to \mathbb{R}^{m\times m}$
be two $m\times m$ symmetric continuous matrices satisfying  the following $3$ conditions
\begin{itemize}
\item (i) Linearity in $\alpha$
\begin{equation}\label{eq: A is linear on alpha}
    A(U, \alpha)=A(U, 0)+\alpha D(U), \qquad D(U) \text{ is positive semi-definite}.
\end{equation}
\item (ii) There exists a constant $k_{m-1}\geq 0$, such that 
\begin{equation}\label{eq: An-1 is spd}
    A_{m-1}(U, \alpha) \text{ is positive-definite}, \qquad \forall U\in  SO(d), \alpha\geq k_{m-1},
\end{equation}
where $A_{m-1}$ is the $(m-1)$th leading principle minor of $A$. 
\item (iii) For any $U\in SO(d)$, there exists a constant $k_{m, U}\geq k_{m-1}$ and an open neighbourhood $\mathcal{U}_{U}$ of $U$, such that
\begin{equation}\label{eq: local spd}
    \det(A(\tilde{U}, k_{m, U}))\geq 0, \qquad \forall \tilde{U}\in \mathcal{U}_{U}.
\end{equation}
\end{itemize}
Then there exists a finite constant $k_m\geq k_{m-1}$, such that for any $\alpha\geq k_m$, it holds
\begin{equation}\label{eq: A is spd}
     A(U, \alpha)\text{ is positive semi-definite}, \qquad \forall U\in SO(d).
\end{equation}
\end{lemma}
\begin{proof}

We begin by showing that for an \(m \times m\) matrix \(A\), if the \((m-1)\)th leading principal minor \(A_{m-1}\) is positive definite and \(\det(A) \geq 0\), then \(A\) is positive semi-definite.

Consider the Schur complement of \(A_{m-1}\) in \(A\), denoted by \(A / A_{m-1}\). According to \cite[Appendix A.5.5]{boyd2004convex}, if \(A_{m-1}\) is positive definite, then
\[
A \text{ is positive semi-definite} \iff A / A_{m-1} \geq 0.
\]

Since \(\det(A) \geq 0\) and \(\det(A_{m-1}) > 0\), we have
\[
A / A_{m-1} = \frac{\det(A)}{\det(A_{m-1})} \geq 0.
\]
Thus, \(A\) is positive semi-definite.

Next, we consider any \(U \in SO(d)\) and \(\alpha \geq k_{m, U}\). From condition (iii) \eqref{eq: local spd}, we have
\[
\det(A(\tilde{U}, k_{m, U})) \geq 0 \quad \forall \tilde{U} \in \mathcal{U}_{U}.
\]
Moreover, by condition (ii) \eqref{eq: An-1 is spd} and the fact that \(k_{m, U} \geq k_{m-1}\), it follows that
\[
A_{m-1}(\tilde{U}, k_{m, U}) \text{ is positive definite} \quad \forall \tilde{U} \in SO(d).
\]
Therefore, \(A(\tilde{U}, k_{m, U})\) is positive semi-definite for any \(\tilde{U} \in \mathcal{U}_{U}\).

Finally, by condition (i) \eqref{eq: A is linear on alpha}, we have
\begin{equation}
A(\tilde{U}, \alpha) = A(\tilde{U}, k_{m, U}) + (\alpha - k_{m, U})D(\tilde{U}).
\end{equation}
Thus $\det(A(\tilde{U}, \alpha))\geq 0$ for all $\alpha\geq k_{m, U}, \tilde{U}\in \mathcal{U}_U$. And \eqref{eq: A is spd} is a direct result of the compactness of $SO(d)$ \cite{fulton2013representation} and the open cover theorem.
\end{proof}

Our objective is to establish that the two auxiliary matrices $\tilde{M}$, which dimension is $d=2$,   and $M$, which dimension is $d=3$, are positive semi-definite. Lemma \ref{lem: compactness} offers a straightforward and unified method to show the positive semi-definiteness of a matrix in arbitrary dimensions. By applying Lemma \ref{lem: compactness} with $A=\tilde{M}$ or $A=M$, we find that the condition (i) is already fulfilled.

For the condition (ii) \eqref{eq: An-1 is spd}, the positive-definiteness of $A_{m-1}(U, \alpha)$ can be verified by examining its leading principal minors. This is elaborated in Lemma \ref{lem: F1 and M1 2d} for $d=2$, and Lemma \ref{lem: F1 and M1} for $d=3$, respectively.

When it comes to verifying the condition (iii) \eqref{eq: local spd}, the analysis diverges into two cases. The simpler case is when $\det(A(U, k_{m, U}))>0$ for some positive $k_{m, U}$, the condition (iii) \eqref{eq: local spd} is ensured by the continuity of $\det(A(\cdot, k_{m, U}))$. For detailed computations in this case, we refer to Lemma \ref{lem: F4 and M4 2d} for $d=2$ and Lemma \ref{lem: F4 and M4} for $d=3$.

The most difficult case arises when $\det(A(U, k_{m, U}))=0$.  We need to prove that $U$ is a local minimum of the determine function $\det(A(\cdot, k_{m, U}))$. Typically, this can be verified by demonstrating that its gradient is zero and its Hessian matrix is positive-definite. However, finding the gradient of a determinant function is quite complicated, Jacobi's formula is introduced to simplify the calculations.

\begin{lemma}[Jacobi's formula]\label{lem: jacobi}
Suppose $A=(a_{i,j})_{m\times m}$ be a matrix of functions, we have
\begin{equation}\label{eq: jacobi 1}
  \frac{\partial\det(A)}{\partial \beta }=\textnormal{Tr}\left(\textnormal{adj}(A)\frac{\partial A}{\partial \beta}\right).
\end{equation}
\begin{equation}\label{eq: jacobi 2}
  \frac{\partial^2\det(A)}{\partial \beta \partial \phi}=\textnormal{Tr}\left(\textnormal{adj}(A)\frac{\partial^2A}{\partial \beta \partial \phi}\right)+\sum_{i\neq j}\det \begin{bmatrix}
  a_{1,1}&a_{1,2}&\ldots&a_{1,m}\\
  \vdots&\vdots&\ldots&\vdots\\
  \frac{\partial a_{i,1}}{\partial \beta}&\frac{\partial a_{i,2}}{\partial \beta}
  &\ldots&\frac{\partial a_{i,m}}{\partial \beta}\\
  \vdots&\vdots&\ldots&\vdots\\
  \frac{\partial a_{j,1}}{\partial \phi}&\frac{\partial a_{j,2}}{\partial \phi}
  &\ldots&\frac{\partial a_{j,m}}{\partial \phi}\\
  \vdots&\vdots&\ldots&\vdots\\
  a_{m,1}&a_{m,2}&\ldots&a_{m,m}
  \end{bmatrix}.
\end{equation}
Here $\textnormal{adj}(A)$ is the adjunct matrix of $A$.
\end{lemma}
We refer the proof of Jacobi's formula to \cite{bellman1997introduction}.

Another challenge is that $\det(A(\cdot, k_{m, U}))>0$ is a function defined on $SO(d)$ rather than the Euclid space $\mathbb{R}^d$. In order to define the open neighbourhood of $U\in SO(d)$, the gradient, and the Hessian matrix of $\det(A(\cdot, k_{m, U}))$ with respect to $U$, we need to provide a group representation of $SO(d)$. The following two lemmas are adapted from \cite{fulton2013representation}. 
\begin{lemma}[Group representation of $SO(2)$]\label{lem: rep 2d}
For any $U\in SO(2)$, there exists $\theta\in[0,2\pi]$, such that
\begin{equation}\label{eq: rep of O 2d}
  U \begin{bmatrix} \boldsymbol{\tau}, \boldsymbol{n}\end{bmatrix}=\begin{bmatrix} \boldsymbol{\tau}, \boldsymbol{n}\end{bmatrix}U(\theta),
\end{equation}
where
\begin{equation}\label{eq: O angles 2d}
  U(\theta):={\begin{bmatrix}\cos \theta &\sin \theta \\-\sin \theta &\cos \theta \end{bmatrix}}.
\end{equation}
Moreover, we have
\begin{subequations}\label{eq: dO at 0 2d}
\begin{align}
\label{eq: O at 0 2d}
&U \begin{bmatrix} \boldsymbol{\tau}, \boldsymbol{n}\end{bmatrix}\Big|_{\theta=0}=\begin{bmatrix} \boldsymbol{\tau}, \boldsymbol{n}\end{bmatrix}I_2,\quad
\frac{d}{d \theta} U\begin{bmatrix} \boldsymbol{\tau}, \boldsymbol{n}\end{bmatrix}\Big|_{\theta=0} =\begin{bmatrix} \boldsymbol{\tau}, \boldsymbol{n}\end{bmatrix}{\begin{bmatrix}0&1\\-1&0 \end{bmatrix}},\\
\label{eq: Othetatheta at 0 2d}
&\frac{d^2}{d \theta^2}  U\begin{bmatrix} \boldsymbol{\tau}, \boldsymbol{n}\end{bmatrix}\Big|_{\theta=0}=\begin{bmatrix} \boldsymbol{\tau}, \boldsymbol{n}\end{bmatrix}{\begin{bmatrix}-1&0\\0&-1\end{bmatrix}}.
\end{align}
\end{subequations}
\end{lemma}

\begin{lemma}[Group representation of $SO(3)$]\label{lem: rep}
For any $U\in SO(3)$, there exists $\boldsymbol{\Phi} = (\phi, \theta, \psi)^T \in [0,2\pi]^3$, such that
\begin{equation}\label{eq: rep of O}
  U \begin{bmatrix} \boldsymbol{\tau}_1, \boldsymbol{\tau}_2, \boldsymbol{n}\end{bmatrix}=\begin{bmatrix} \boldsymbol{\tau}_1, \boldsymbol{\tau}_2, \boldsymbol{n}\end{bmatrix}U(\boldsymbol{\Phi}),
\end{equation}
where $U(\boldsymbol{\Phi})$ is given as
\begin{equation*}
  {\begin{bmatrix}\cos \theta \cos \psi &-\cos \phi \sin \psi +\sin \phi \sin \theta \cos \psi &\sin \phi \sin \psi +\cos \phi \sin \theta \cos \psi \\\cos \theta \sin \psi &\cos \phi \cos \psi +\sin \phi \sin \theta \sin \psi &-\sin \phi \cos \psi +\cos \phi \sin \theta \sin \psi \\-\sin \theta &\sin \phi \cos \theta &\cos \phi \cos \theta \end{bmatrix}}.
\end{equation*}
Moreover, for $\beta, \varphi\in \{\phi, \theta, \psi\}$, we have
\begin{subequations}\label{eq: dO at 0}
\begin{align}
\label{eq: O at 0}
&U\begin{bmatrix} \boldsymbol{\tau}_1, \boldsymbol{\tau}_2, \boldsymbol{n}\end{bmatrix}\Big|_{\boldsymbol{\Phi}=\boldsymbol{0}}=\begin{bmatrix} \boldsymbol{\tau}_1, \boldsymbol{\tau}_2, \boldsymbol{n}\end{bmatrix}I_3,\\
\label{eq: Ophi at 0}
&\frac{\partial }{\partial \phi} U\begin{bmatrix} \boldsymbol{\tau}_1, \boldsymbol{\tau}_2, \boldsymbol{n}\end{bmatrix}\Big|_{\boldsymbol{\Phi}=\boldsymbol{0}}=\begin{bmatrix} \boldsymbol{\tau}_1, \boldsymbol{\tau}_2, \boldsymbol{n}\end{bmatrix}{\begin{bmatrix}0&0&0\\0&0&-1\\0&1&0 \end{bmatrix}},\\
\label{eq: Otheta at 0}
&\frac{\partial}{\partial \theta}U \begin{bmatrix} \boldsymbol{\tau}_1, \boldsymbol{\tau}_2, \boldsymbol{n}\end{bmatrix}\Big|_{\boldsymbol{\Phi}=\boldsymbol{0}}=\begin{bmatrix} \boldsymbol{\tau}_1, \boldsymbol{\tau}_2, \boldsymbol{n}\end{bmatrix}{\begin{bmatrix}0&0&1\\0&0&0\\-1&0&0 \end{bmatrix}},
\end{align}
\end{subequations}
\begin{subequations}\label{eq: dO at 01}
\begin{align}
\label{eq: Opsi at 0}
&\frac{\partial}{\partial \psi}U \begin{bmatrix} \boldsymbol{\tau}_1, \boldsymbol{\tau}_2, \boldsymbol{n}\end{bmatrix}\Big|_{\boldsymbol{\Phi}=\boldsymbol{0}}=\begin{bmatrix} \boldsymbol{\tau}_1, \boldsymbol{\tau}_2, \boldsymbol{n}\end{bmatrix}{\begin{bmatrix}0&-1&0\\1&0&0\\0&0&0 \end{bmatrix}},\\
\label{eq: Opsipsi at 0}
&\frac{\partial^2}{\partial \psi^2}U \begin{bmatrix} \boldsymbol{\tau}_1, \boldsymbol{\tau}_2, \boldsymbol{n}\end{bmatrix}\Big|_{\boldsymbol{\Phi}=\boldsymbol{0}}=\begin{bmatrix} \boldsymbol{\tau}_1, \boldsymbol{\tau}_2, \boldsymbol{n}\end{bmatrix}{\begin{bmatrix}-1&0&0\\0&-1&0\\0&0&0 \end{bmatrix}},\\
\label{eq: Oother at 0}
&\frac{\partial^2 }{\partial \beta\partial \varphi} \left(U\boldsymbol{\tau}_1\cdot \boldsymbol{n}\right)^2\Big|_{\boldsymbol{\Phi}=\boldsymbol{0}}=2\delta_{\beta\phi}\delta_{\varphi\phi}, \quad \frac{\partial^2 }{\partial \beta\partial \varphi} \left(U\boldsymbol{\tau}_2\cdot \boldsymbol{n}\right)^2\Big|_{\boldsymbol{\Phi}=\boldsymbol{0}}=2\delta_{\beta\theta}\delta_{\varphi\theta}.
\end{align}
\end{subequations}
Here $\delta$ is the Kronecker delta, i.e.,  $\delta_{ij} = 1$ if $i=j$, otherwise $0$.
\end{lemma}

With the adept application of Lemma \ref{lem: jacobi} for the gradient and Hessian matrix of a determinant, and the proper group representation for $SO(d)$ (Lemma \ref{lem: rep 2d} for $d=2$ and Lemma \ref{lem: rep} for $d=3$), the challenging case of $\det(A(\cdot, k_{m, U}))=0$ is handled in Lemma \ref{lem: F3 near 0 2d} for $d=2$ and Lemma \ref{lem: F4 near 0} for $d=3$.

\medskip
\subsection{Existence of the minimal stabilizing function in 2D}
We denote the leading principle minors of $\tilde{M}(U, \alpha)$ as $\tilde{M}_1(U, \alpha), \tilde{M}_2(U, \alpha)$, respectively.

Now we are going to prove the existence of $k_0(\boldsymbol{n})$ by applying Lemma \ref{lem: compactness}. 

\begin{lemma}\label{lem: F1 and M1 2d}
For any $\gamma(\boldsymbol{p})\in C^2(\mathbb{R}^2_*)$, there exists a $k_{1}<\infty$, such that $\forall U\in SO(2), \alpha\geq k_1$, there holds 
\begin{equation}\label{eq: F1 and M1 global 2d}
    \tilde{M}_1(U, \alpha)\text{ is positive-definite}.
\end{equation}
\end{lemma}
\begin{proof}
We choose $k_1=0$. It is easy to check $\det(\tilde{M}_1(U, \alpha))=\gamma(\boldsymbol{n})+\alpha(U\boldsymbol{\tau}\cdot \boldsymbol{n})^2> 0$, and thus we know that $\tilde{M}_1(U, \alpha)$ is positive-definite.
\end{proof}

\begin{lemma}\label{lem: F4 and M4 2d}
For any $\gamma(\boldsymbol{p})\in C^2(\mathbb{R}^2_*)$ with $\gamma(-\boldsymbol{n})<3\gamma(\boldsymbol{n})$ and $\forall I_2 \neq U\in SO(2)$, there exists a constant $k_1\leq k_{2, U}<\infty$ with the open neighbourhood $\mathcal{U}_{U}$ of $U$, such that 
\begin{equation}\label{eq: F4 and M4 global 2d}
    \det(\tilde{M}_2(\tilde{U}, k_{2, U}))\geq 0, \quad \forall \tilde{U}\in \mathcal{U}_{U}.
\end{equation}
\end{lemma}
\begin{proof}
First from Lemma \ref{lem: F1 and M1 2d}, we know that there exists a constant $ k_1\geq 0$, such that $\tilde{M}_1(U, \alpha)$ is positive-definite $\alpha\geq k_1$. 

Suppose $(U_0\boldsymbol{\tau} \cdot \boldsymbol{n})^2\neq 0$, we have 
\begin{equation}
    \det(\tilde{M}_2(U_0, \alpha))=\gamma(\boldsymbol{n})(U_0\boldsymbol{\tau}\cdot \boldsymbol{n})^2\alpha+\mathcal{O}(1).
\end{equation}
Thus for such $U_0$, there exists a constant $k_1\leq k_{2, U_0}<\infty$ and an open neighbourhood $\mathcal{U}_{U_0}$ of $U_0$, such that $\det(\tilde{M}_2(U, k_{2, U_0}))\geq 0, \forall U\in \mathcal{U}_{U_0}$.

If $(U_1\boldsymbol{\tau} \cdot \boldsymbol{n})^2= 0$, we know that $U_1\boldsymbol{n}=\pm \boldsymbol{n}$. If $U_1\boldsymbol{n}=\boldsymbol{n}$, then it must be $I_2$. So the last case is $U_1\boldsymbol{n}=-\boldsymbol{n}$. From the fact $\gamma(-\boldsymbol{n})<3\gamma(\boldsymbol{n})$ and $U_1\boldsymbol{\tau}=-\boldsymbol{\tau}$, we have
\begin{equation}
    \det(\tilde{M}_2(U_1, \alpha))=\frac{3\gamma(\boldsymbol{n})-\gamma(-\boldsymbol{n})}{4}(\gamma(\boldsymbol{n})+\gamma(-\boldsymbol{n}))>0.
\end{equation}
Thus there is an open neighbourhood $\mathcal{U}_{U_1}$ of $U_1$ and a $k_{2, U_1}<\infty$, such that $\forall U\in \mathcal{U}_{U_1}$, it holds $\det(\tilde{M}_2(U, k_{2, U_1}))\geq 0$. 
\end{proof}

To discuss $U$ near $I_2$, by using Lemma \ref{lem: rep 2d}, it suffices to consider the $U=U(\theta)$ when $\theta$ near $0$.

\begin{lemma}\label{lem: F3 near 0 2d}
For any $\gamma(\boldsymbol{p})\in C^2(\mathbb{R}^2_*)$, there exists a constant $k_1\leq k_{2,I_2}<\infty$ with the open neighbourhood $\mathcal{U}$ of $U(0)=I_2$ such that 
\begin{equation}\label{eq: F3 near 0 2d}
    \det(\tilde{M}_2(U, k_{2,I_2}))\geq 0,\quad \forall U\in \mathcal{U}.
\end{equation}
\end{lemma}

\begin{proof}
First by applying the chain rule, noticing , $\nabla \nabla \gamma(\boldsymbol{p})|_{\boldsymbol{p}
=\boldsymbol{n}}={\bf H}_{\gamma}(\boldsymbol{n})$, $\nabla \gamma(\boldsymbol{p})|_{\boldsymbol{p}=\boldsymbol{n}}=\boldsymbol{\xi}(\boldsymbol{n})$, 
together with \eqref{eq: dO at 0 2d}, we obtain that
\begin{subequations}\label{eq: gamma U basic 2d}
\begin{align}
    \label{eq: gamma U ori 2d}
    \gamma(U\boldsymbol{n})\Big|_{\theta = 0}&=\gamma(\boldsymbol{n}), \\
    \label{eq: gamma U dtheta 2d}
    \frac{d \gamma(U\boldsymbol{n})}{d \theta}\Big|_{\theta = 0}&=\boldsymbol{\xi}\cdot \boldsymbol{\tau},\\
    \label{eq: gamma U dthetatheta 2d}
    \frac{d^2 \gamma(U\boldsymbol{n})}{d \theta^2}\Big|_{\theta = 0}&=\left(\frac{d U}{d \theta}\Big|_{\theta = 0} \boldsymbol{n}\right)\cdot {\bf H}_{\gamma}(\boldsymbol{n})\cdot \left(\frac{d U}{d \theta}\Big|_{\theta = 0} \boldsymbol{n}\right)+\boldsymbol{\xi}\cdot \left(\frac{d^2 U}{d \theta^2}\Big|_{\theta = 0} \boldsymbol{n}\right)\nonumber\\
    &=\boldsymbol{\tau}\cdot ({\bf H}_{\gamma}(\boldsymbol{n}) \boldsymbol{\tau})-\gamma(\boldsymbol{n}).
\end{align}
\end{subequations}
By the definition of $\tilde{M}_2(U, \alpha)$, \eqref{eq: dO at 0 2d}, \eqref{eq: gamma U basic 2d}, and the definition of adjunct matrix,  we know that
\begin{subequations}\label{eq: A basic 2d}
\begin{align}
    \label{eq: A basic ori 2d}
    &\tilde{M}_2(U, \alpha)\Big|_{\theta = 0}=\gamma(\boldsymbol{n})\begin{bmatrix}
    1&-1\\-1&1
    \end{bmatrix}, \\
    \label{eq: A basic adj 2d}
    &\text{adj}(\tilde{M}_2(U, \alpha))\Big|_{\theta = 0}=\gamma(\boldsymbol{n})\begin{bmatrix}
    1\\1
    \end{bmatrix}\begin{bmatrix}
    1&1
    \end{bmatrix},\\
    \label{eq: A basic dtheta 2d}
    &\frac{d \tilde{M}_2(U, \alpha)}{d \theta}\Big|_{\theta = 0}=\begin{bmatrix}
    0&0\\0&0
    \end{bmatrix},\\
    \label{eq: A basic dthetatheta 2d}
    &\frac{d^2 \tilde{M}_2(U, \alpha)}{d \theta^2}\Big|_{\theta = 0}=\begin{bmatrix}
    2\alpha&*\\-\frac{1}{2}(-2\gamma(\boldsymbol{n})+\boldsymbol{\tau}\cdot ({\bf H}_{\gamma}(\boldsymbol{n}) \boldsymbol{\tau}))&0
    \end{bmatrix}.
\end{align}
\end{subequations}
\eqref{eq: jacobi 1}, \eqref{eq: jacobi 2} in Lemma \ref{lem: jacobi} and \eqref{eq: A basic ori 2d}-\eqref{eq: A basic dthetatheta 2d} suggest that
\begin{equation}\label{eq: F3 ori and first 2d}
    \det(\tilde{M}_2((U, \alpha)))\Big|_{\theta = 0}=0, \quad \frac{d \det(\tilde{M}_2((U, \alpha)))}{d\theta}\Big|_{\theta = 0}=0,
\end{equation}
and
\begin{equation}\label{eq: F3 Hessian 2d}
    \frac{d^2 \det(\tilde{M}_2(U, \alpha))}{d\theta^2}\Big|_{\theta = 0}=\gamma(\boldsymbol{n})\left(2\alpha+2\gamma(\boldsymbol{n})-\boldsymbol{\tau}\cdot ({\bf H}_{\gamma}(\boldsymbol{n}) \boldsymbol{\tau})\right).
\end{equation}
\eqref{eq: F3 Hessian 2d} implies that there exists a $k_1\leq k_{2, I_2}<\infty$, such that $\frac{d^2 \det(\tilde{M}_2(U, k_{2, I_2}))}{d\theta^2}\Big|_{\theta = 0}>0$. By the continuity of $\frac{d^2 \det(\tilde{M}_2(U, k_{2, I_2}))}{d\theta^2}$, we know that there exists an open neighbourhood $\mathcal{U}$ of $I_2$, such that $\frac{d^2 \det(\tilde{M}_2(U, k_{2, I_2}))}{d\theta^2}\geq 0, \forall U\in \mathcal{U}$. Thus by Taylor expansion and \eqref{eq: F3 ori and first 2d}, we know that there exists a $\det(\tilde{M}_2(U, k_{2,I_2}))\geq 0, \forall  U\in \mathcal{U}$, which validates \eqref{eq: F3 near 0 2d}. 
\end{proof}

\medskip

{\em Proof of Theorem \ref{thm: existence of k0} in 2D}. The condition (ii) \eqref{eq: An-1 is spd} is proved by Lemma \ref{lem: F1 and M1 2d}, and the condition (iii) \eqref{eq: local spd} is the result of Lemma \ref{lem: F4 and M4 2d} and Lemma \ref{lem: F3 near 0 2d}. 

For the condition (i), it is obvious that $\tilde{M}_2(U, \alpha)=\tilde{M}_2(U, 0)+\alpha\tilde{D}$, where $\tilde{D}=\text{diag}(\left(U\boldsymbol{\tau}\cdot \boldsymbol{n}\right)^2, 0)$ is positive semi-definite. 
Therefore by Lemma \ref{lem: compactness}, we derive that 
\begin{equation}
    k_2<\infty, \quad k_2  \in \Big\{\alpha\Big|\,\,\tilde{M}(U, \alpha) \text{ is positive semi-definite }\quad \forall U\in SO(2)\Big\}.
\end{equation}
Thus such a set is nonempty. On the other hand, let $U\boldsymbol{\tau}\cdot \boldsymbol{n}=1$ and $\tilde{\alpha}=-2\gamma(\boldsymbol{n})$. We know that $\det(\tilde{M}_1(U, \tilde{\alpha}))=\gamma(\boldsymbol{n})+\tilde{\alpha}(U\boldsymbol{\tau}\cdot \boldsymbol{n})^2=-\gamma(\boldsymbol{n})<0$, and the set is also bounded below. Therefore the set has a finite infimum $k_0(\boldsymbol{n})$.

\subsection{Existence of the minimal stabilizing function in 3D}
Similarly, we denote the leading principle minors of $M(U, \alpha)$ are denoted as $M_1(U, \alpha), M_2(U, \alpha), M_3(U, \alpha)$ and $M_4(U, \alpha)$, respectively.

To apply Lemma \ref{lem: compactness}, we first need to show $M_3(U, \alpha)$ is positive-definite for large enough $\alpha$.
 
\begin{lemma}\label{lem: F1 and M1}
For any $\gamma(\boldsymbol{p})\in C^2(\mathbb{R}^3_*)$ with $\gamma(-\boldsymbol{n})<2\gamma(\boldsymbol{n})$, there exists a constant  $k_{3}<\infty$, such that $\forall U\in SO(3), \alpha\geq k_3$, there holds 
\begin{equation}\label{eq: F1 and M1 global}
    M_3(U, \alpha)\text{ is positive-definite}.
\end{equation}
\end{lemma}
\begin{proof}
By checking the leading principle minors, $M_3(U, \alpha)$ is positive-definite if and only if $\det(M_1), \det(M_2), \det(M_3)>0$. It is easy to verify that
\begin{subequations}
\begin{align*}
    \det(M_1(U, \alpha))&=\gamma(\boldsymbol{n})+\alpha(U\boldsymbol{\tau}_1\cdot \boldsymbol{n})^2, \\
    \det(M_2(U, \alpha))&= \alpha^2 (U\boldsymbol{\tau}_1 \cdot \boldsymbol{n})^2(U\boldsymbol{\tau}_2 \cdot \boldsymbol{n})^2+\alpha ((U\boldsymbol{\tau}_1 \cdot \boldsymbol{n})^2+(U\boldsymbol{\tau}_2 \cdot \boldsymbol{n})^2)\gamma(\boldsymbol{n})\nonumber\\
    &\qquad +\frac{4\gamma(\boldsymbol{n})^2-\gamma(U\boldsymbol{n})^2}{4},\\
    \det(M_3(U, \alpha))&= \left(\alpha ((U\boldsymbol{\tau}_1 \cdot \boldsymbol{n})^2+(U\boldsymbol{\tau}_2 \cdot \boldsymbol{n})^2)\gamma(\boldsymbol{n})+\frac{4\gamma(\boldsymbol{n})^2-\gamma(U\boldsymbol{n})^2}{4}\right)\nonumber\\
    &\qquad (\gamma(\boldsymbol{n})+\alpha(U\boldsymbol{\tau}_2\cdot\boldsymbol{n})^2).
\end{align*}
\end{subequations}
Thus for $\alpha\geq 0$, we know $\det(M_1(U, \alpha))>0, \alpha^2 (U\boldsymbol{\tau}_1 \cdot \boldsymbol{n})^2(U\boldsymbol{\tau}_2 \cdot \boldsymbol{n})^2\geq 0$. Also, $\det(M_2(U, \alpha)), \det(M_3(U, \alpha))$ are nondecreasing with respect to $\alpha$. Moreover, if $\alpha ((U\boldsymbol{\tau}_1 \cdot \boldsymbol{n})^2+(U\boldsymbol{\tau}_2 \cdot \boldsymbol{n})^2)\gamma(\boldsymbol{n})+\frac{4\gamma(\boldsymbol{n})^2-\gamma(U\boldsymbol{n})^2}{4}>0$, we can deduce that \\ $\det(M_2(U, \alpha)), \det(M_3(U, \alpha))>0$.

Suppose $(U_1\boldsymbol{\tau}_1 \cdot \boldsymbol{n})^2+(U_1\boldsymbol{\tau}_2 \cdot \boldsymbol{n})^2>0$. Then for such $U_1\in SO(3)$, we know that there exists a constant $k_{3, U_1}\geq 0$ with an open neighbourhood $\mathcal{U}_{U_1}$ of $U_1$, such that for all $\tilde{U}\in \mathcal{U}_{U_1}$ and $\alpha>k_{3, U_1}$
\begin{equation}\label{eq: lem C1 aux 1}
    \alpha ((\tilde{U}_1\boldsymbol{\tau}_1 \cdot \boldsymbol{n})^2+(\tilde{U}_1\boldsymbol{\tau}_2 \cdot \boldsymbol{n})^2)\gamma(\boldsymbol{n})+\frac{4\gamma(\boldsymbol{n})^2-\gamma(\tilde{U}_1\boldsymbol{n})^2}{4}>0.
\end{equation}

On the contrary, $(U_2\boldsymbol{\tau}_1 \cdot \boldsymbol{n})^2+(U_2\boldsymbol{\tau}_2 \cdot \boldsymbol{n})^2=0$ implies both $U_2\boldsymbol{\tau}_1 \cdot \boldsymbol{n}=0$ and $U_2\boldsymbol{\tau}_2 \cdot \boldsymbol{n}=0$, we know that $U_2\boldsymbol{n}=\pm \boldsymbol{n}$. In this case, $\alpha ((U\boldsymbol{\tau}_1 \cdot \boldsymbol{n})^2+(U\boldsymbol{\tau}_2 \cdot \boldsymbol{n})^2)\gamma(\boldsymbol{n})+\frac{4\gamma(\boldsymbol{n})^2-\gamma(U\boldsymbol{n})^2}{4}$ becomes
\begin{equation}\label{eq: lem C1 aux 2}
    \frac{4\gamma(\boldsymbol{n})^2-\gamma(U\boldsymbol{n})^2}{4}\geq \min\left\{\frac{3\gamma(\boldsymbol{n})^2}{4}, \frac{4\gamma(\boldsymbol{n})^2-\gamma(-\boldsymbol{n})^2}{4}\right\}>0.
\end{equation}
And we can simply choose $k_{3, U_2}=0$. By applying the open cover theorem and \eqref{eq: lem C1 aux 1}, \eqref{eq: lem C1 aux 2}, and the above analysis, we deduce the desired result. 
\end{proof}

\begin{lemma}\label{lem: F4 and M4}
For any $\gamma(\boldsymbol{p})\in C^2(\mathbb{R}^3_*)$ with $\gamma(-\boldsymbol{n})<2\gamma(\boldsymbol{n})$, and $\forall U\in SO(3)$, $U\neq I_3=U(0, 0, 0)$, $U\neq {\rm diag(-1,-1,1)}=U(0, 0, \pi)$, there exists a constant $k_3\leq k_{4, U}<\infty$ with the open neighbourhood $\mathcal{U}_{U}$ of $U$, such that 
\begin{equation}\label{eq: F4 and M4 global}
    \det(M_4(\tilde{U}, k_{4, U}))\geq 0, \quad \forall \tilde{U}\in \mathcal{U}_{U}.
\end{equation}
\end{lemma}
\begin{proof}
First from Lemma \ref{lem: F1 and M1}, we know that there exists a constant $k_3\geq 0$, such that $M_3(U, \alpha)$ is positive-definite $\alpha\geq k_3$. 
Suppose $(U_0\boldsymbol{\tau}_2 \cdot \boldsymbol{n})^2\neq 0$, we have 
\begin{align}
    &\det(M_4(U_0, \alpha))\nonumber\\
    &\ =(U_0\boldsymbol{\tau}_2\cdot \boldsymbol{n})^2\gamma(\boldsymbol{n})^2\left((U_0\boldsymbol{\tau}_1\cdot \boldsymbol{n})^2+(U_0\boldsymbol{\tau}_2\cdot \boldsymbol{n})^2\right)\alpha^2\nonumber\\
    &\quad -(U_0\boldsymbol{\tau}_2\cdot \boldsymbol{n})^2\left((U_0\boldsymbol{\tau}_1\cdot \boldsymbol{n})M_{43}-(U_0\boldsymbol{\tau}_2\cdot \boldsymbol{n})M_{41}\right)^2\alpha^2+\mathcal{O}(\alpha)\nonumber\\
    &\ =(U_0\boldsymbol{\tau}_2\cdot \boldsymbol{n})^2\gamma(\boldsymbol{n})^2\left((U_0\boldsymbol{\tau}_1\cdot \boldsymbol{n})^2+(U_0\boldsymbol{\tau}_2\cdot \boldsymbol{n})^2\right)\alpha^2\nonumber\\
    &\quad -\frac{(U_0\boldsymbol{\tau}_2\cdot \boldsymbol{n})^2\gamma(\boldsymbol{n})^2}{4}[(U_0\boldsymbol{\tau}_1\cdot \boldsymbol{n})(U_0\boldsymbol{\tau}_2\cdot \boldsymbol{\tau}_1)-(U_0\boldsymbol{\tau}_2\cdot \boldsymbol{n})(U_0\boldsymbol{\tau}_1\cdot \boldsymbol{\tau}_1)]^2\alpha^2
    +\mathcal{O}(\alpha)\nonumber\\
    &\ \geq \frac{(U_0\boldsymbol{\tau}_2\cdot \boldsymbol{n})^2\gamma(\boldsymbol{n})^2}{2}\left((U_0\boldsymbol{\tau}_1\cdot \boldsymbol{n})^2+(U_0\boldsymbol{\tau}_2\cdot \boldsymbol{n})^2\right)\alpha^2+\mathcal{O}(\alpha).\nonumber
\end{align}
Thus for such $U_0$, there exists a constant 
$k_3\leq k_{4, U_0}<\infty$ and an open neighbourhood $\mathcal{U}_{U_0}$ of $U_0$,
 such that $F_4(U, k_{4, U_0})\geq 0, \forall U\in \mathcal{U}_{U_0}$.

Next, suppose $(U_1\boldsymbol{\tau}_1 \cdot \boldsymbol{n})^2\neq 0, (U_1\boldsymbol{\tau}_2 \cdot \boldsymbol{n})^2=0$, we have
\begin{align}
    \det(M_4(U_1, \alpha))&=\gamma(\boldsymbol{n})(U_1\boldsymbol{\tau}_1 \cdot \boldsymbol{n})^2\left(\gamma(\boldsymbol{n})^2-M_{42}^2-M_{43}^2\right)\alpha+\mathcal{O}(1)\nonumber\\
    &\geq \frac{1}{2}\gamma(\boldsymbol{n})^3(U_1\boldsymbol{\tau}_1 \cdot \boldsymbol{n})^2\alpha+\mathcal{O}(1).\nonumber
\end{align}
By the same argument, we know that there exists a constant $k_3\leq k_{4,U_1}<\infty$ and an open neighbourhood $\mathcal{U}_{U_1}$ of $U_1$, such that $\det(M_4(U, k_{4, U_1}))\geq 0, \forall U\in \mathcal{U}_{U_1}$.

If both $(U_2\boldsymbol{\tau}_1 \cdot \boldsymbol{n})^2= 0$ and $(U_2\boldsymbol{\tau}_2 \cdot \boldsymbol{n})^2= 0$, we know that $U_2\boldsymbol{n}=\pm \boldsymbol{n}$. First we assume that $U_2(\boldsymbol{\Phi})\boldsymbol{n}=\boldsymbol{n}$, i.e. $\phi=\theta=0$. In this case, from Lemma \ref{lem: rep} and \eqref{eq: rep of O} we obtain
\begin{equation}\label{eq: rep of O plus}
U_2\boldsymbol{\tau}_1\cdot \boldsymbol{\tau}_1=\cos\psi,\, U_2\boldsymbol{\tau}_1\cdot \boldsymbol{\tau}_2=\sin\psi,\, U_2\boldsymbol{\tau}_2\cdot \boldsymbol{\tau}_2=\cos\psi,\, U_2\boldsymbol{\tau}_2\cdot \boldsymbol{\tau}_1=-\sin\psi.
\end{equation}
For any $\alpha\geq k_3$, by applying \eqref{eq: rep of O plus} we have 
\begin{equation}
    \det(M_4(U_2, \alpha))=\frac{9\sin^2\psi}{16}\gamma(\boldsymbol{n})^4.
\end{equation}
The condition $U\neq I_3=U(0, 0, 0), U\neq U(0, 0, \pi)={\rm diag(-1,-1,1)}$, implies $\psi\neq 0,\pi$, thus we know that $\det(M_4(U_2, k_3))>0$. By the same argument, there exists such open neighbourhood $\mathcal{U}_{U_2}$ of $U_2$ and the $k_3= k_{4, U_2}<\infty$. 

The last case is $U_3(\boldsymbol{\Phi})\boldsymbol{n}=-\boldsymbol{n}$, we assume that $\phi=\pi, \theta=0$. For any $\alpha>0$, from the fact $\gamma(-\boldsymbol{n})<2\gamma(\boldsymbol{n})$ and Lemma \ref{lem: rep}, we have 
\begin{align}
    \det(M_4(U_3, \alpha))=&\gamma(\boldsymbol{n})^2\frac{2\gamma(\boldsymbol{n})-\gamma(-\boldsymbol{n})}{32}
    \bigl(\gamma(\boldsymbol{n})(10-2\cos(2\psi))\bigr.\nonumber\\
    &\ \bigl.+\gamma(-\boldsymbol{n})(7-2\cos(2\psi))\bigr)>0.
\end{align}
By the same argument, there is an open neighbourhood $\mathcal{U}_{U_3}$ of $U_3$ and a constant 
$k_3= k_{4, U_3}<\infty$, such that $\forall U\in \mathcal{U}_{U_3}$, it holds $\det(M_4(U, k_{4, U_3}))\geq 0$. 
\end{proof}

To discuss $U$ near $I_3$ or $U$ near $U(0,0,\pi)={\rm diag(-1,-1,1)}$, it suffices to consider the $U=U(\boldsymbol{\Phi})$ when $\boldsymbol{\Phi}$ near $\boldsymbol{0}$ or $(0,0,\pi)^T$. 

\begin{lemma}\label{lem: F4 near 0}
For any $\gamma(\boldsymbol{p})\in C^2(\mathbb{R}^3_*)$ with $\gamma(-\boldsymbol{n})<2\gamma(\boldsymbol{n})$, there exists $k_3\leq k_{4,I_3}<\infty, k_3\leq k_{4,U(0,0,\pi)}<\infty$ with the open neighbourhood $\mathcal{U}$ of $I_3$, $\mathcal{V}$ of $U(0,0,\pi)$ such that 
\begin{equation}\label{eq: F4 near 0}
    F_4(U, k_{4,I_3})\geq 0,\quad \forall U(\boldsymbol{\Phi})\in \mathcal{U};
\end{equation}
\begin{equation}\label{eq: F4 near 0 2}
    F_4(U, k_{4,U(0,0,\pi)})\geq 0,\quad \forall U(\boldsymbol{\Phi})\in \mathcal{V}.
\end{equation}
\end{lemma}
\begin{proof}
First by applying the chain rule, noticing $\nabla \gamma(\boldsymbol{p})|_{\boldsymbol{p}=\boldsymbol{n}}=\boldsymbol{\xi}(\boldsymbol{n}), \\\nabla \nabla \gamma(\boldsymbol{p})|_{\boldsymbol{p}=\boldsymbol{n}}={\bf H}_{\gamma}(\boldsymbol{n})$, together with \eqref{eq: dO at 0}, we obtain that
\begin{subequations}\label{eq: gamma U basic}
\begin{align}
    \label{eq: gamma U ori}
    \gamma(U\boldsymbol{n})\Big|_{\boldsymbol{\Phi}=\boldsymbol{0}}&=\gamma(\boldsymbol{n}), \\
    \label{eq: gamma U dphi}
    \frac{\partial \gamma(U\boldsymbol{n})}{\partial \phi}\Big|_{\boldsymbol{\Phi}=\boldsymbol{0}}&=\nabla \gamma(U \boldsymbol{n})\Big|_{\boldsymbol{\Phi}=\boldsymbol{0}}\cdot \left(\frac{\partial U}{\partial \phi}\Big|_{\boldsymbol{\Phi}=\boldsymbol{0}} \boldsymbol{n}\right)=-\boldsymbol{\xi}\cdot \boldsymbol{\tau}_2,\\
    \label{eq: gamma U dtheta}
    \frac{\partial \gamma(U\boldsymbol{n})}{\partial \theta}\Big|_{\boldsymbol{\Phi}=\boldsymbol{0}}&=\boldsymbol{\xi}\cdot \boldsymbol{\tau}_1,
    \qquad \frac{\partial \gamma(U\boldsymbol{n})}{\partial \psi}\Big|_{\boldsymbol{\Phi}=\boldsymbol{0}}=\boldsymbol{0},\\
    \label{eq: gamma U dpsipsi}
    \frac{\partial^2 \gamma(U\boldsymbol{n})}{\partial \psi^2}\Big|_{\boldsymbol{\Phi}=\boldsymbol{0}}&=\left(\frac{\partial U}{\partial \psi}\Big|_{\boldsymbol{\Phi}=\boldsymbol{0}} \boldsymbol{n}\right)\cdot {\bf H}_{\gamma}(\boldsymbol{n})\cdot \left(\frac{\partial U}{\partial \psi}\Big|_{\boldsymbol{\Phi}=\boldsymbol{0}} \boldsymbol{n}\right)\nonumber\\
    &\quad +\boldsymbol{\xi}\cdot \left(\frac{\partial^2 U}{\partial \psi^2}\Big|_{\boldsymbol{\Phi}=\boldsymbol{0}} \boldsymbol{n}\right)\nonumber\\
    &=\boldsymbol{0}\cdot ({\bf H}_{\gamma}(\boldsymbol{n}) \boldsymbol{0})+\boldsymbol{\xi}\cdot \boldsymbol{0}=\boldsymbol{0}.
\end{align}
\end{subequations}
By definition of $M_4(U, \alpha)$, \eqref{eq: dO at 0}, \eqref{eq: gamma U basic}, and the definition of adjunct matrix,  we know that
\begin{subequations}\label{eq: M4 basic}
\begin{align}
    \label{eq: M4 basic ori}
    &M_4(U, \alpha)\Big|_{\boldsymbol{\Phi}=\boldsymbol{0}}=\gamma(\boldsymbol{n})\begin{bmatrix}
    1&-1/2&0&-1/2\\-1/2&1&0&-1/2\\0&0&1&0\\-1/2&-1/2&0&1
    \end{bmatrix}, \\
    \label{eq: M4 basic adj}
    &\text{adj}(M_4(U, \alpha))\Big|_{\boldsymbol{\Phi}=\boldsymbol{0}}=\frac{3}{4}\gamma(\boldsymbol{n})^3\begin{bmatrix}
    1\\1\\0\\1
    \end{bmatrix}\begin{bmatrix}
    1&1&0&1
    \end{bmatrix},\\
    \label{eq: M4 basic dphi}
    &\frac{\partial M_4(U, \alpha)}{\partial \phi}\Big|_{\boldsymbol{\Phi}=\boldsymbol{0}}=\frac{1}{2}\begin{bmatrix}
    0&\boldsymbol{\tau}_2\cdot \boldsymbol{\xi}&0&0\\ \boldsymbol{\tau}_2\cdot \boldsymbol{\xi}&0&0&-\boldsymbol{\tau}_2\cdot \boldsymbol{\xi}\\0&0&0&-\boldsymbol{\tau}_1\cdot \boldsymbol{\xi}\\0&-\boldsymbol{\tau}_2\cdot \boldsymbol{\xi}&-\boldsymbol{\tau}_1\cdot \boldsymbol{\xi}&0
    \end{bmatrix},\\
    \label{eq: M4 basic dtheta}
    &\frac{\partial M_4(U, \alpha)}{\partial \theta}\Big|_{\boldsymbol{\Phi}=\boldsymbol{0}}=\frac{\boldsymbol{\tau}_1\cdot \boldsymbol{\xi}}{2}\begin{bmatrix}
    0&-1&0&1\\-1&0&0&0\\0&0&0&0\\1&0&0&0
    \end{bmatrix},
\end{align}
\end{subequations}    
\begin{subequations}\label{eq: M4 basicc}
\begin{align}   
    \label{eq: M4 basic dpsi}
    \frac{\partial M_4(U, \alpha)}{\partial \psi}\Big|_{\boldsymbol{\Phi}=\boldsymbol{0}}=\frac{\gamma(\boldsymbol{n})}{2}\begin{bmatrix}
    0&0&0&0\\0&0&0&0\\0&0&0&1\\0&0&1&0
    \end{bmatrix},\\ 
    \label{eq: M4 basic dpsipsi}
    \frac{\partial^2 M_4(U, \alpha)}{\partial \psi^2}\Big|_{\boldsymbol{\Phi}=\boldsymbol{0}}=\frac{\gamma(\boldsymbol{n})}{2}\begin{bmatrix}
    0&0&0&1\\0&0&0&1\\0&0&0&0\\1&1&0&0
    \end{bmatrix}.
\end{align}
\end{subequations}
Combining \eqref{eq: jacobi 1} in Lemma \ref{lem: jacobi} and \eqref{eq: M4 basic} and
\eqref{eq: M4 basicc}, we get
\begin{equation}\label{eq: F4 ori and first}
    \det(M_4(U, \alpha))\Big|_{\boldsymbol{\Phi}=\boldsymbol{0}}=0, \quad \frac{\partial \det(M_4(U, \alpha))}{\partial \beta}\Big|_{\boldsymbol{\Phi}=\boldsymbol{0}}=0, \,\,\forall \beta\in \{\phi, \theta, \psi\}.
\end{equation}
Obviously, $M_4(U, 0)$ is independent of $\alpha$. From \eqref{eq: M4 basic}, we observe that $M_4(U, \alpha)\Big|_{\boldsymbol{\Phi}=\boldsymbol{0}}$ and $\frac{\partial M_4(U, \alpha)}{\partial \beta}\Big|_{\boldsymbol{\Phi}=\boldsymbol{0}}$ for $\beta\in \{\phi, \theta, \psi\}$ are also independent of $\alpha$. Thus for any $\beta, \varphi\in \{\phi, \theta, \psi\}$, we define the constant $C^1_{4, \beta, \varphi}, C^2_{4, \beta, \varphi}$ as follows
\begin{subequations}\label{eq: def of C4s}
\begin{align}
\label{eq: def of C14}
    &C^1_{4, \beta, \varphi}:=\frac{3}{4}\gamma(\boldsymbol{n})^3\begin{bmatrix}
    1\\1\\0\\1
    \end{bmatrix}^T \frac{\partial^2 M_4(U, 0)}{\partial \beta \partial \varphi}\Big|_{\boldsymbol{\Phi}=\boldsymbol{0}} \begin{bmatrix}
    1\\1\\0\\1
    \end{bmatrix},\\
\label{eq: def of C24}
    &C^2_{4, \beta, \varphi}:=\sum_{i\neq j}\det \begin{bmatrix}
    M_{1,1}&M_{1,2}&M_{1,3}&M_{1,4}\\
    \frac{\partial M_{i,1}}{\partial \beta}&\frac{\partial M_{i,2}}{\partial \beta}&\frac{\partial M_{i,3}}{\partial \beta}&\frac{\partial M_{i,3}}{\partial \beta}\\
    \frac{\partial M_{j,1}}{\partial \varphi}&\frac{\partial M_{j,2}}{\partial \varphi}&\frac{\partial M_{j,3}}{\partial \varphi}&\frac{\partial M_{j,3}}{\partial \varphi}\\
    M_{4,1}&M_{4,2}&M_{4,3}&M_{4,4}
    \end{bmatrix}.
\end{align}
\end{subequations}

From the definition of $M(U, \alpha)$, we know that
\begin{align}\label{eq: decomp of M4}
    M_4(U, \alpha)&:=M_4(U, 0)+\alpha D(U).
\end{align}
with
\begin{align*}
D(U)=\begin{bmatrix}
    (U\boldsymbol{\tau}_1\cdot \boldsymbol{n})^2 &0&(U\boldsymbol{\tau}_1\cdot \boldsymbol{n})(U\boldsymbol{\tau}_2\cdot \boldsymbol{n})&0\\
    0&(U\boldsymbol{\tau}_2\cdot \boldsymbol{n})^2&0&0\\
    (U\boldsymbol{\tau}_1\cdot \boldsymbol{n})(U\boldsymbol{\tau}_2\cdot \boldsymbol{n})&0&(U\boldsymbol{\tau}_2\cdot \boldsymbol{n})^2&0\\
    0&0&0&0
    \end{bmatrix}.
\end{align*}

Using \eqref{eq: jacobi 2} in Lemma \ref{lem: jacobi}, together with \eqref{eq: M4 basic adj}, \eqref{eq: Oother at 0}, \eqref{eq: def of C4s}, \eqref{eq: decomp of M4}, we deduce for any $\beta, \varphi\in \{\phi, \theta, \psi\}$, the second order derivative of $F_4$ as follows
\begin{align*}
    &\frac{\partial^2 \det(M_4(U, \alpha))}{\partial \beta \partial \varphi}\Big|_{\boldsymbol{\Phi}=\boldsymbol{0}}\nonumber\\
    &=\text{tr}\left(\text{adj}(M_4(U, \alpha))\frac{\partial^2M_4(U, \alpha)}{\partial \beta \partial \varphi}\right)+\sum_{i\neq j}\det \begin{bmatrix}
    M_{1,1}&M_{1,2}&M_{1,3}&M_{1,4}\\
    \frac{\partial M_{i,1}}{\partial \beta}&\frac{\partial M_{i,2}}{\partial \beta}&\frac{\partial M_{i,3}}{\partial \beta}&\frac{\partial M_{i,3}}{\partial \beta}\\
    \frac{\partial M_{j,1}}{\partial \varphi}&\frac{\partial M_{j,2}}{\partial \varphi}&\frac{\partial M_{j,3}}{\partial \varphi}&\frac{\partial M_{j,3}}{\partial \varphi}\\
    M_{4,1}&M_{4,2}&M_{4,3}&M_{4,4}
    \end{bmatrix}\nonumber\\
    &=\frac{3}{4}\gamma(\boldsymbol{n})^3\begin{bmatrix}
    1\\1\\0\\1
    \end{bmatrix}\cdot \frac{\partial^2 (M_4(U, 0)+\alpha D(U))}{\partial \beta \partial \varphi}\Big|_{\boldsymbol{\Phi}=\boldsymbol{0}} \begin{bmatrix}
    1\\1\\0\\1
    \end{bmatrix}+C^2_{4, \beta, \varphi}\nonumber\\
    &=C^1_{4, \beta, \varphi}+C^2_{4, \beta, \varphi}+\frac{3\alpha}{4}\gamma(\boldsymbol{n})^3(2\delta_{\beta\phi}\delta_{\varphi\phi}+2\delta_{\beta\theta}\delta_{\varphi\theta}).
\end{align*}
We note only $\frac{\partial^2 \det(M_4(U, \alpha))}{\partial \phi^2}\Big|_{\boldsymbol{\Phi}=\boldsymbol{0}}, \frac{\partial^2 \det(M_4(U, \alpha))}{\partial \theta^2}\Big|_{\boldsymbol{\Phi}=\boldsymbol{0}}$ depend on $\alpha$. Hence the Hessian matrix ${\bf H}_{\det(M_4(U, \alpha))}\Big|_{\boldsymbol{\Phi}=\boldsymbol{0}}$ can be written as
\begin{equation}\label{eq: F4 Hessian}
    {\bf H}_{\det(M_4(U, \alpha))}\Big|_{\boldsymbol{\Phi}=\boldsymbol{0}}=(C^1_{4, \beta, \varphi}+C^2_{4, \beta, \varphi})_{\beta, \varphi\in \{\phi, \theta, \psi\}}+\frac{3\alpha}{2}\gamma(\boldsymbol{n})^3\text{diag}(1,1, 0).
\end{equation}

Moreover, by combining \eqref{eq: M4 basic adj}, \eqref{eq: M4 basic dpsi}, \eqref{eq: M4 basic dpsipsi}, \eqref{eq: def of C4s}, \eqref{eq: F4 Hessian} for \\ $\frac{\partial^2 \det(M_4(U, \alpha))}{\partial \psi^2}\Big|_{\boldsymbol{\Phi}=\boldsymbol{0}}$,  we have
\begin{equation}\label{eq: F4 dpsipsi}
    \frac{\partial^2 \det(M_4(U, \alpha))}{\partial \psi^2}\Big|_{\boldsymbol{\Phi}=\boldsymbol{0}}=C^1_{4, \psi, \psi}+C^2_{4, \psi, \psi}=\frac{9}{8}\gamma(\boldsymbol{n})^4>0.
\end{equation}
This together with \eqref{eq: F4 Hessian} imply that there exists a $k_{3}\leq k_{4,I_3}<\infty$, such that ${\bf H}_{\det(M_4(U, k_{4,I_3}))}\Big|_{\boldsymbol{\Phi}=\boldsymbol{0}}$ is positive-definite. By the continuity of ${\bf H}_{\det(M_4(U, \alpha))}$, we know that there is an open neighbourhood $\mathcal{U}$ of $I_3$, such that $\forall U(\boldsymbol{\Phi})\in \mathcal{U}$, it holds
\begin{equation}
    {\bf H}_{\det(M_4(U, k_{4,I_3}))}\Big|_{U=U(\boldsymbol{\Phi})} \text{ is positive semi-definite}.
\end{equation}
Thus by Taylor expansion, we know that $\det(M_4(U, k_{4,I_3}))\geq 0, \forall U(\boldsymbol{\Phi})\in \mathcal{U}$, which validates \eqref{eq: F4 near 0}. 

The proof of \eqref{eq: F4 near 0 2} is similar. 
\end{proof}

\medskip
Similar to the proof of Theorem \ref{thm: existence of k0} for $d=2$, Theorem \ref{thm: existence of k0} for $d=3$ is also a direct result of Lemma \ref{lem: compactness} together with Lemma \ref{lem: F1 and M1}, Lemma  \ref{lem: F4 and M4} and Lemma \ref{lem: F4 near 0}.  

\section{Numerical results}

\setcounter{equation}{0}

In this section, we present numerical results for the proposed unified SP-PFEM \eqref{eq: full PFEM} for time evolution of surfaces in 3D. We demonstrate the efficiency of the method using a convergence test and verify the main result \eqref{thm: main} with a conservation law test. And we also apply \eqref{eq: full PFEM} to show the morphological evolution of several non-symmetric anisotropic energies.

\begin{table}[htp!]
\label{tb: convergence rate}
\begin{center}
 \begin{tabular}{@{\extracolsep{\fill}}|c|cc|cc|cc|}\hline
$(h,\tau)$ & $e^h (1)$\, \footnotesize{Case 1\,\,} & \footnotesize{order} & $e^h (1)$\, \footnotesize{Case 2\,\,} & \footnotesize{order} & $e^h(1)$\, \footnotesize{Case 3\,} & \footnotesize{order} \\ \hline
$(h_0,\tau_0)$ & $1.48 \times 10^{-1}$ & - & $1.56 \times 10^{-1}$ & - & $1.63 \times 10^{-1}$ & - \\ \hline
$\left(\frac{h_0}{2}, \frac{\tau_0}{4}\right)$ & $3.68 \times 10^{-2}$ & 2.01 & $3.87 \times 10^{-2}$ & 2.01 & $3.98 \times 10^{-2}$ & 2.03 \\ \hline
$\left(\frac{h_0}{2^2}, \frac{\tau_0}{4^2}\right)$ & $8.95 \times 10^{-3}$ & 2.04 & $9.73 \times 10^{-3}$ & 1.99 & $9.53 \times 10^{-3}$ & 2.06 \\ \hline

 \end{tabular}
 \vspace{5mm}
\begin{tabular}{@{\extracolsep{\fill}}|c|cc|cc|cc|}\hline
$(h,\tau)$ & $e^h (1)$\, \footnotesize{Case 1'} & \footnotesize{order} & $e^h (1)$\, \footnotesize{Case 2'} & \footnotesize{order} & $e^h (1)$\, \footnotesize{Case 3'} & \footnotesize{order} \\ \hline
$(h_0,\tau_0)$ & $1.63 \times 10^{-1}$ & - & $1.65 \times 10^{-1}$ & - & $1.66 \times 10^{-1}$ & - \\ \hline
$\left(\frac{h_0}{2}, \frac{\tau_0}{4}\right)$ & $3.95 \times 10^{-2}$ & 2.04 & $4.23 \times 10^{-2}$ & 1.96 & $4.04 \times 10^{-2}$ & 2.04 \\ \hline
$\left(\frac{h_0}{2^2}, \frac{\tau_0}{4^2}\right)$ & $9.66 \times 10^{-3}$ & 2.03 & $1.01 \times 10^{-2}$ & 2.07 & $9.76 \times 10^{-3}$ & 2.05 \\ \hline

 \end{tabular}\vspace{1em}
\end{center}
\caption{Numerical errors of $e_{h,\tau}(t=1)$ with $k(\boldsymbol{n})=k_0(\boldsymbol{n})$ (upper part) and $k(\boldsymbol{n})=\sup\limits_{\boldsymbol{n}\in \mathbb{S}^2}k_0(\boldsymbol{n})$ (lower part) for Cases 1-3, while $h_0:=2^{-1}$ and $\tau_0:=\frac{2^{-1}}{25}$. Here Case $i$/ Case $i$' means the anisotropic energy in Case $i$ with $k(\boldsymbol{n})=k_0(\boldsymbol{n})/ k(\boldsymbol{n})=\sup\limits_{\boldsymbol{n}\in \mathbb{S}^2}k_0(\boldsymbol{n})$, respectively.}
 \end{table}

For the spatial discretization, the initial surface $S_0$ is approximated by the polyhedral mesh $\Gamma_{h, \tau}(0)=\Gamma^0=\cup_{j=1}^J\sigma_j^0$ with the mesh size parameter $h$ via the \textit{CFDTool}. The time step $\tau$ corresponding to the mesh $\Gamma^0$ is chosen as $\tau=\frac{2}{25}h^2$. To solve the implicit unified SP-PFEM \eqref{eq: full PFEM}, we employ the Newton iteration proposed in \cite{bao2022symmetrized}, where the tolerance $\varepsilon$ is chosen as $10^{-12}$.

\begin{figure}[htp!]
\centering
\includegraphics[width=1\textwidth]{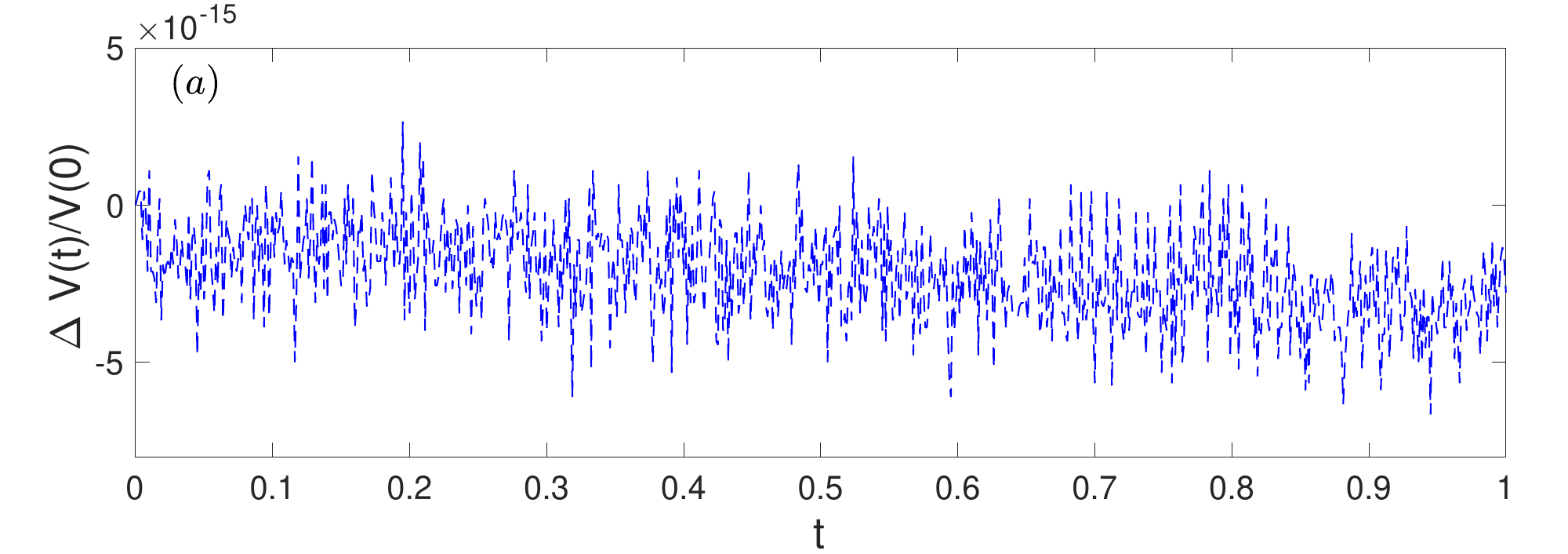}
\includegraphics[width=1\textwidth]{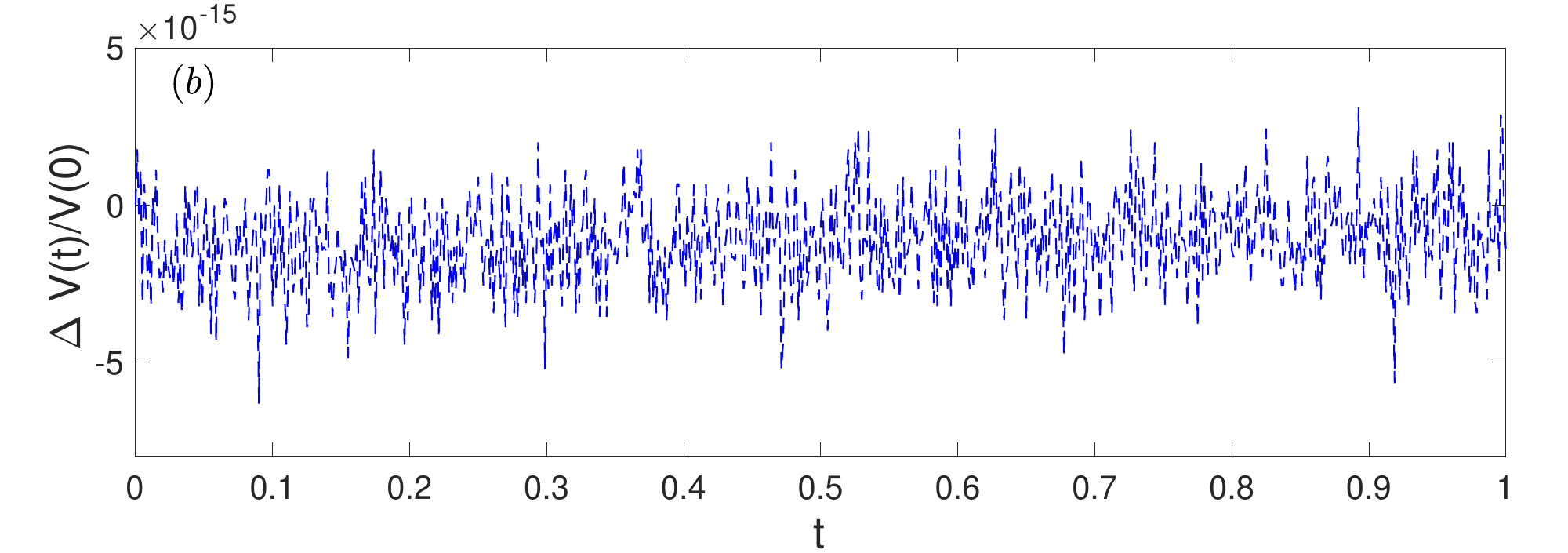}
\includegraphics[width=1\textwidth]{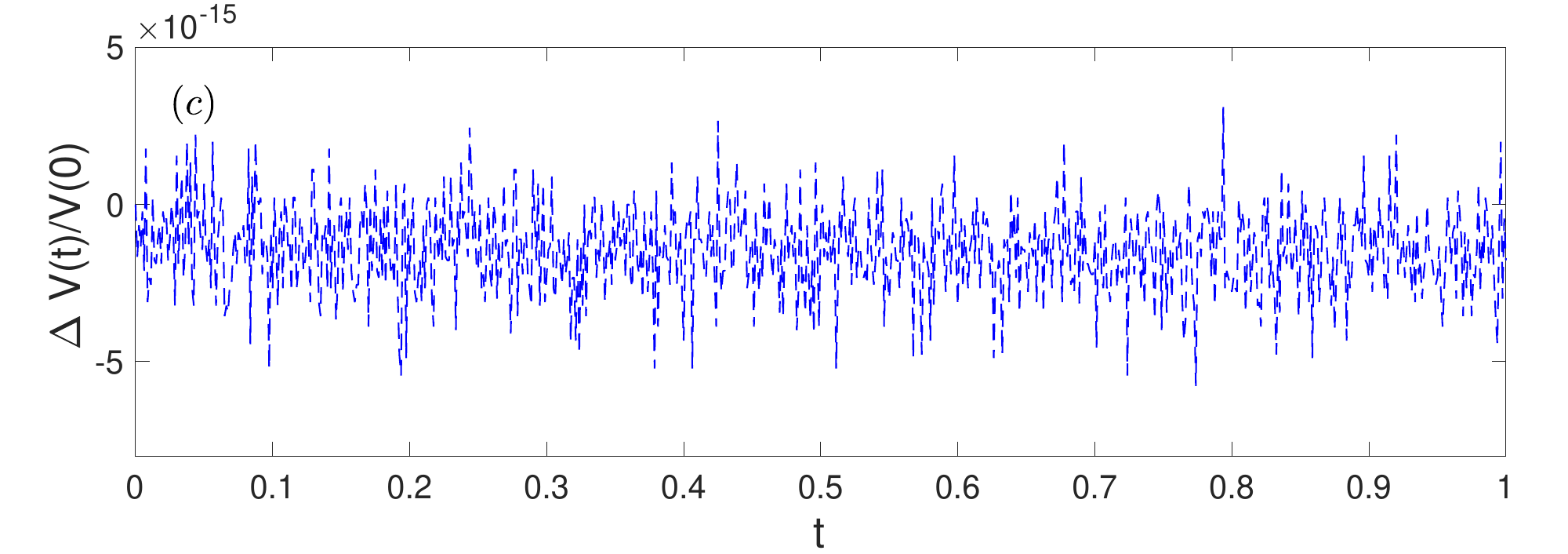}
\caption{Plot of the normalized volume change $\frac{\Delta V(t)}{V(0)}$ for different cases: (a) for Case I, (b) for Case II,  and (c) for Case III.}
\label{fig: volume}
\end{figure}

In the numerical tests, we consider the following three anisotropic surface energies as follows
\begin{itemize}
  \item Case I: $\gamma(\boldsymbol{n})=1+\frac{1}{8}(n_1^3+n_2^3+n_3^3)$;
  \item Case II: $\gamma(\boldsymbol{n})=1+\frac{1}{4}(n_1^3+n_2^3+n_3^3)$;
  \item Case III: $\gamma(\boldsymbol{n})=\sqrt{(\frac{5}{2}+\frac{3}{2}\text{sign}(n_1))n_1^2+n_2^2+n_3^2}$.
\end{itemize} 
The minimal stabilizing function $k_0(\boldsymbol{n})$ is determined numerically as follows: for the interpolation points $\boldsymbol{n}_{ij}=(\cos\theta_i\cos\phi_j, \cos\theta_i\sin\phi_j, \sin\theta_i)^T$ for $\theta_i=\frac{i\pi}{10}, \phi_j=-\frac{\pi}{2}+\frac{j-1}{10}\pi, \,  i=1, 2, \ldots, 20, j=1, 2, \ldots, 21$, we solve the optimization problem \eqref{eq: def of k0} to determine $k_0(\boldsymbol{n}_{ij})$; and for the other points, $k_0(\boldsymbol{n})$ is given by the bilinear interpolation. 

\begin{figure}[htp!]
\centering
\includegraphics[width=12cm,height=14cm,angle=0,]{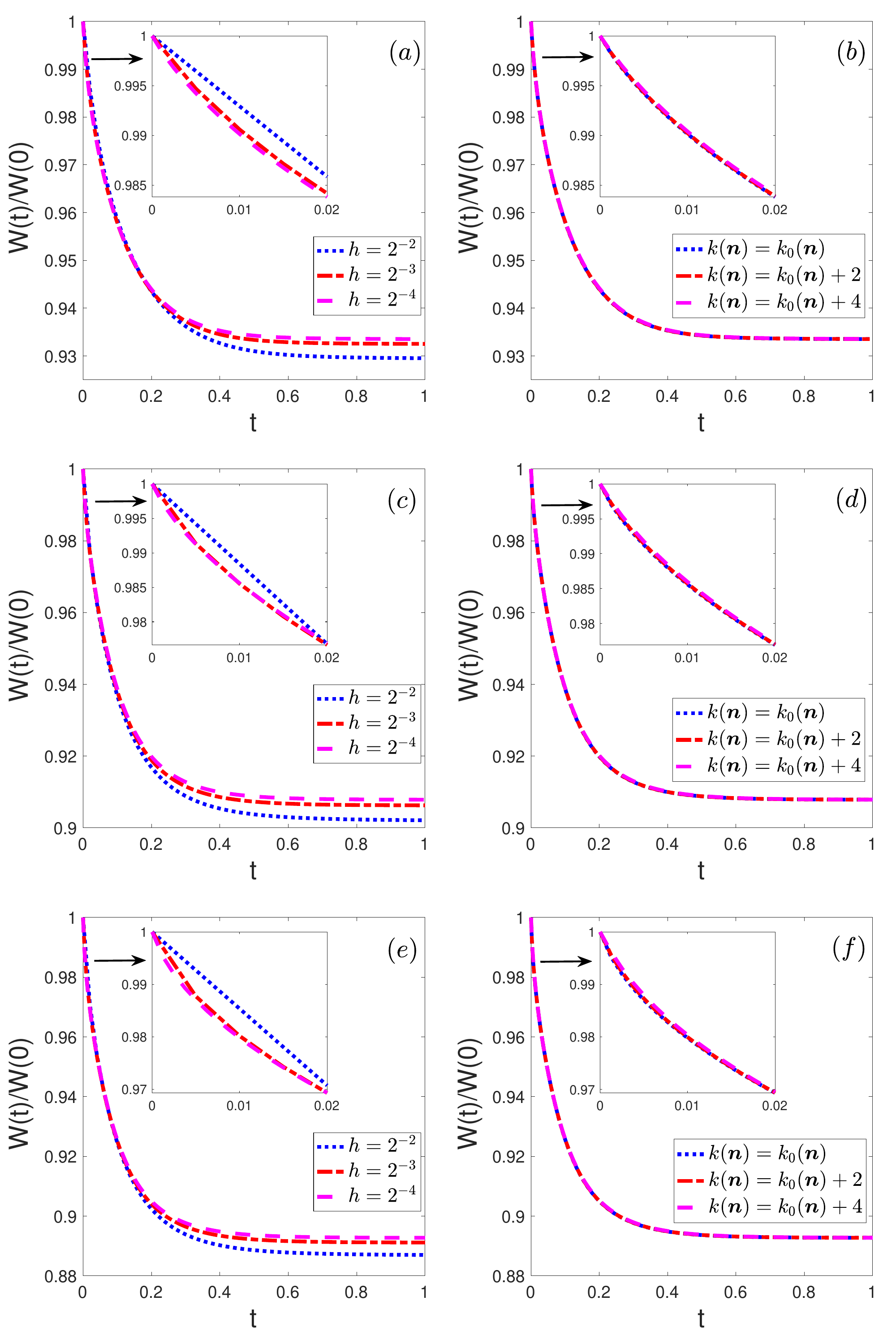}
\caption{Plot of the normalized energy $\frac{W(t)}{W(0)}$ for anisotropic energies in Case I-III with the fixed $k(\boldsymbol{n})=k_0(\boldsymbol{n})$ (left column) for different $h$ and $\tau$; or the fixed $h=2^{-4}$ and $\tau=\frac{2}{25}h^2$ with different $k(\boldsymbol{n})$ (right column). The top, middle, and bottom rows correspond to the anisotropic energies in Case I-III, respectively.}
\label{fig: energy1}
\end{figure}

To test the convergence rate, the initial surface $S_0$ is chosen as a $2\times 1\times 1$ cuboid. We denote the numerical error between the numerical solution as $\Gamma_{h, \tau}(t)$ and the exact solution $\Gamma(t)$ as $e^h(t)$. The intermediate surface $\Gamma_{h, \tau}(t)$ is defined as
\begin{equation}
  \Gamma_{h, \tau}(t):=\frac{t-t_m}{\tau}\Gamma_{h, \tau}(t_m)+\frac{t_{m+1}-t}{\tau}\Gamma_{h, \tau}(t_{m+1}), \qquad t_m\leq t<t_{m+1}.
\end{equation}
And the exact solution $\Gamma(t)$ is approximated by $S_{h_e, \tau_e}(t)$ with a small mesh size of $h_e=2^{-4}$ and a time step of $\tau_e=\frac{2}{25}h_e^2$. We adopt the manifold distance $M(S_{h, \tau}(t), \Gamma(t))$ to quantify the numerical error $e^h(t)$, which is given as
\begin{equation}
  e^h(t)=M(\Gamma_{h, \tau}(t), \Gamma(t)):=2|\Omega_1\cup \Omega_2|-|\Omega_1|-|\Omega_2|.
\end{equation}
Here $\Omega_1$ and $\Omega_2$ represent the  regions enclosed by $\Gamma_{h, \tau}(t)$ and $\Gamma(0)$, respectively.

\begin{figure}[htp!]
\centering
\includegraphics[width=1\textwidth]{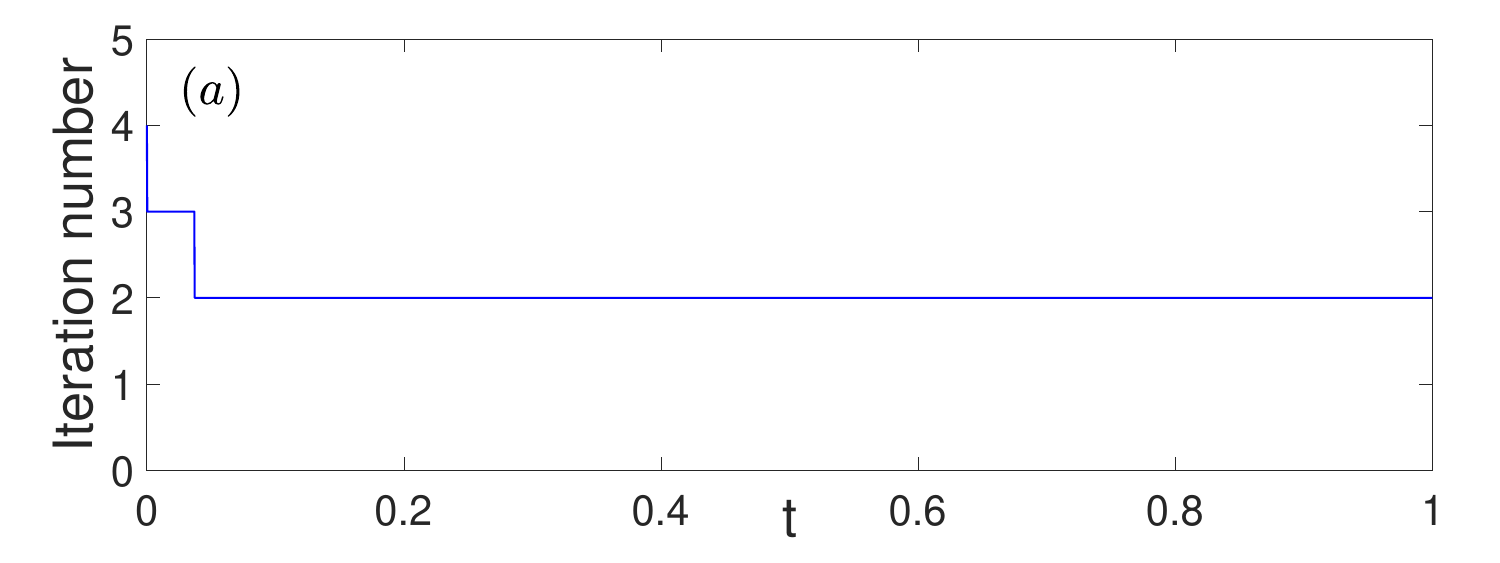}
\includegraphics[width=1\textwidth]{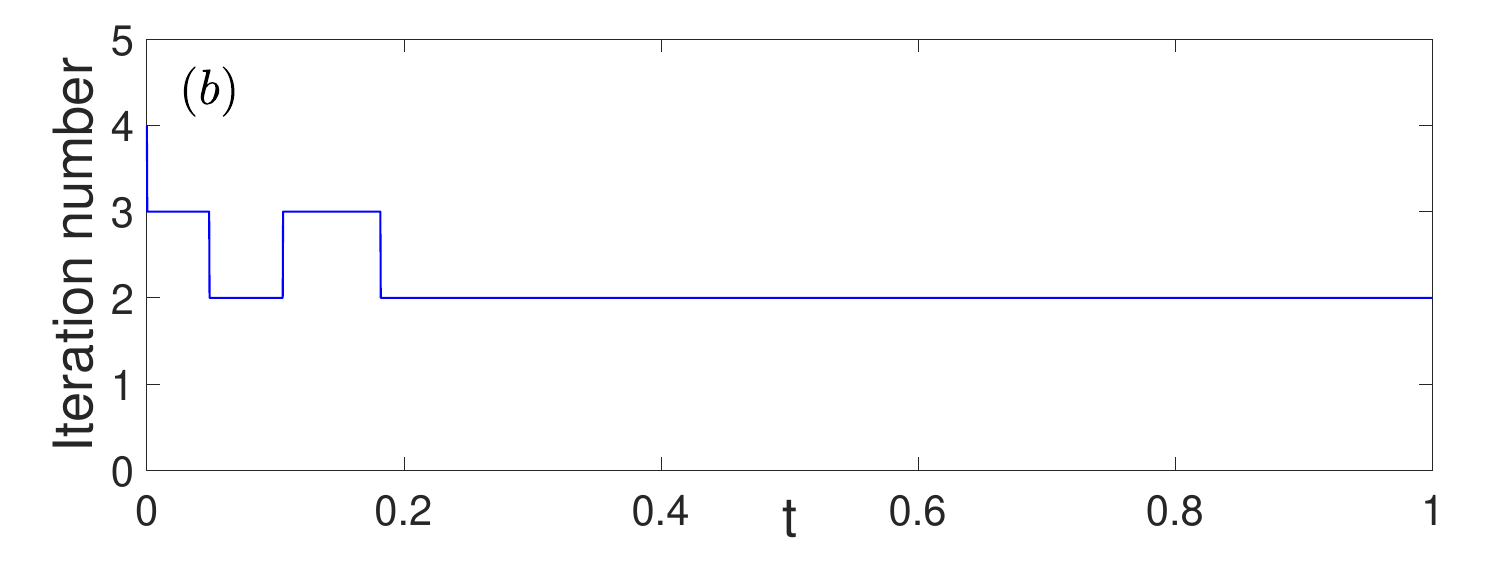}
\includegraphics[width=1\textwidth]{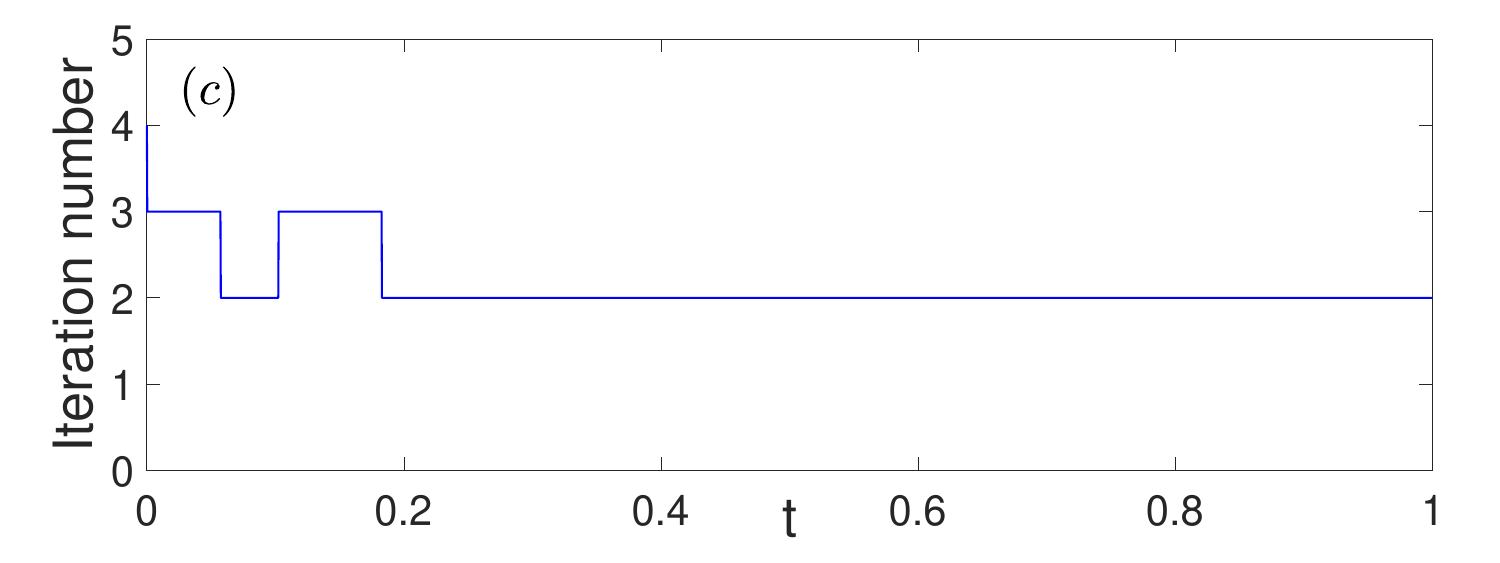}
\caption{Plot of the number of iterations per time step for anisotropies in (a) Case I, (b) Case II, and (c) Case III.}
\label{fig: CNT}
\end{figure}

\begin{figure}[hbtp!]
\centering
\includegraphics[width=1\textwidth]{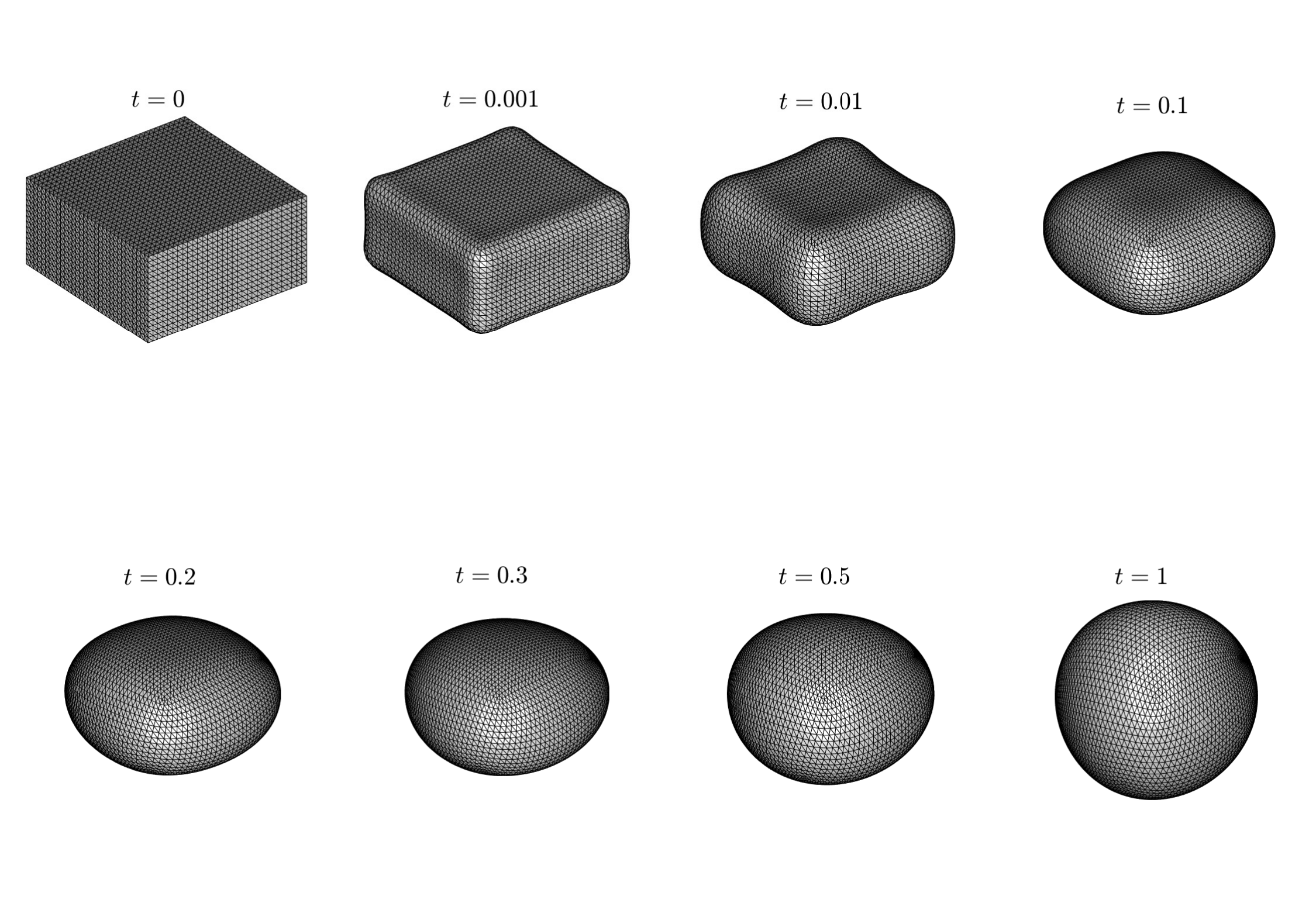}
\caption{Evolution of a $2\times 2\times 1$ cuboid by anisotropic surface diffusion with a weak anisotropy $\gamma(\boldsymbol{n})=1+\frac{1}{8}(n_1^3+n_2^3+n_3^3)$ and $k(\boldsymbol{n})=k_0(\boldsymbol{n})$ at different times.}
\label{fig: evolve1}
\end{figure}

\begin{figure}[hbtp!]
\centering
\includegraphics[width=1\textwidth]{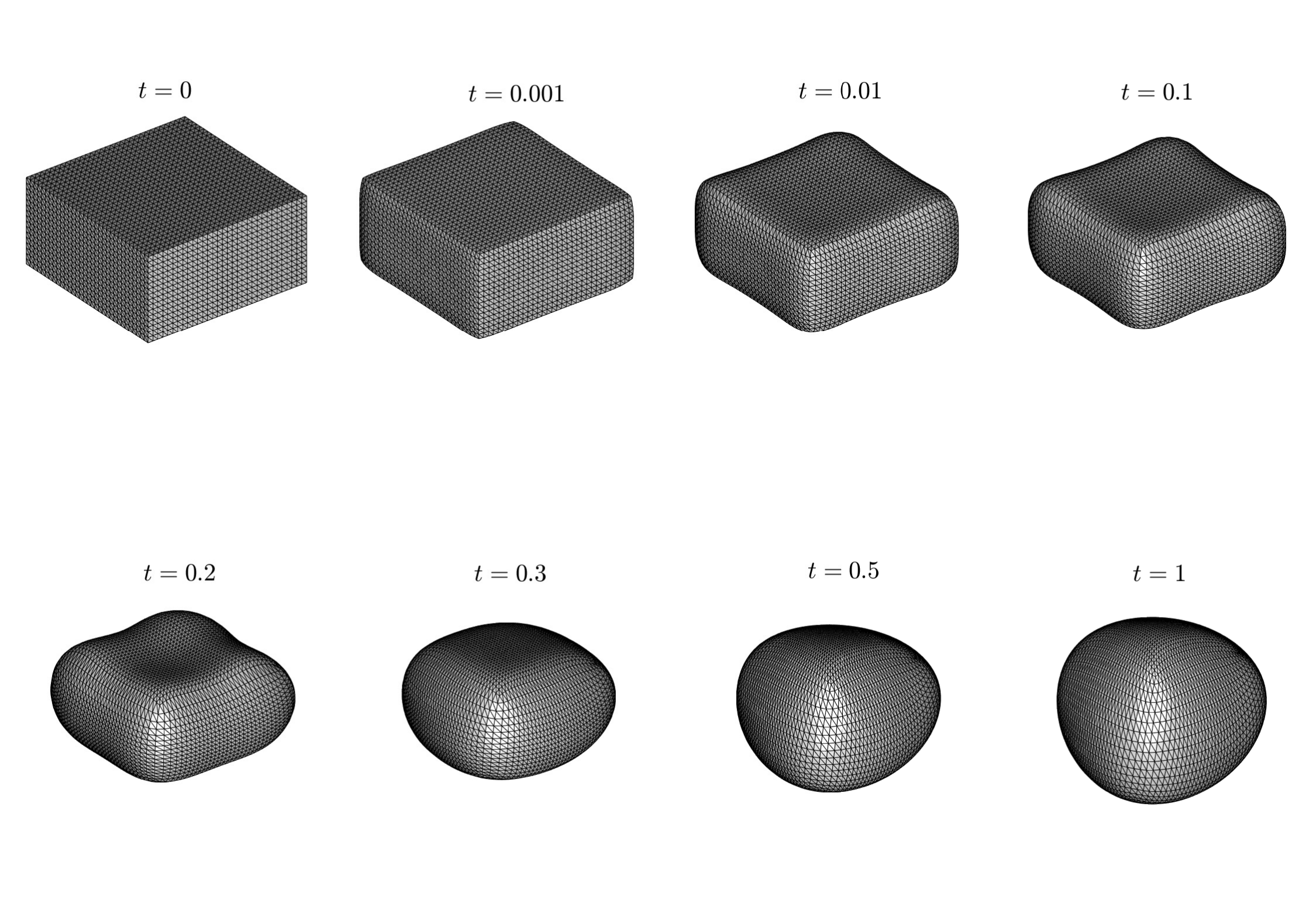}
\caption{Evolution of a $2\times 2\times 1$ cuboid by anisotropic surface diffusion with a weak anisotropy $\gamma(\boldsymbol{n})=1+\frac{1}{4}(n_1^3+n_2^3+n_3^3)$ and $k(\boldsymbol{n})=k_0(\boldsymbol{n})$ at different times.}
\label{fig: evolve2}
\end{figure}
\begin{figure}[hbtp!]
\centering
\includegraphics[width=1\textwidth]{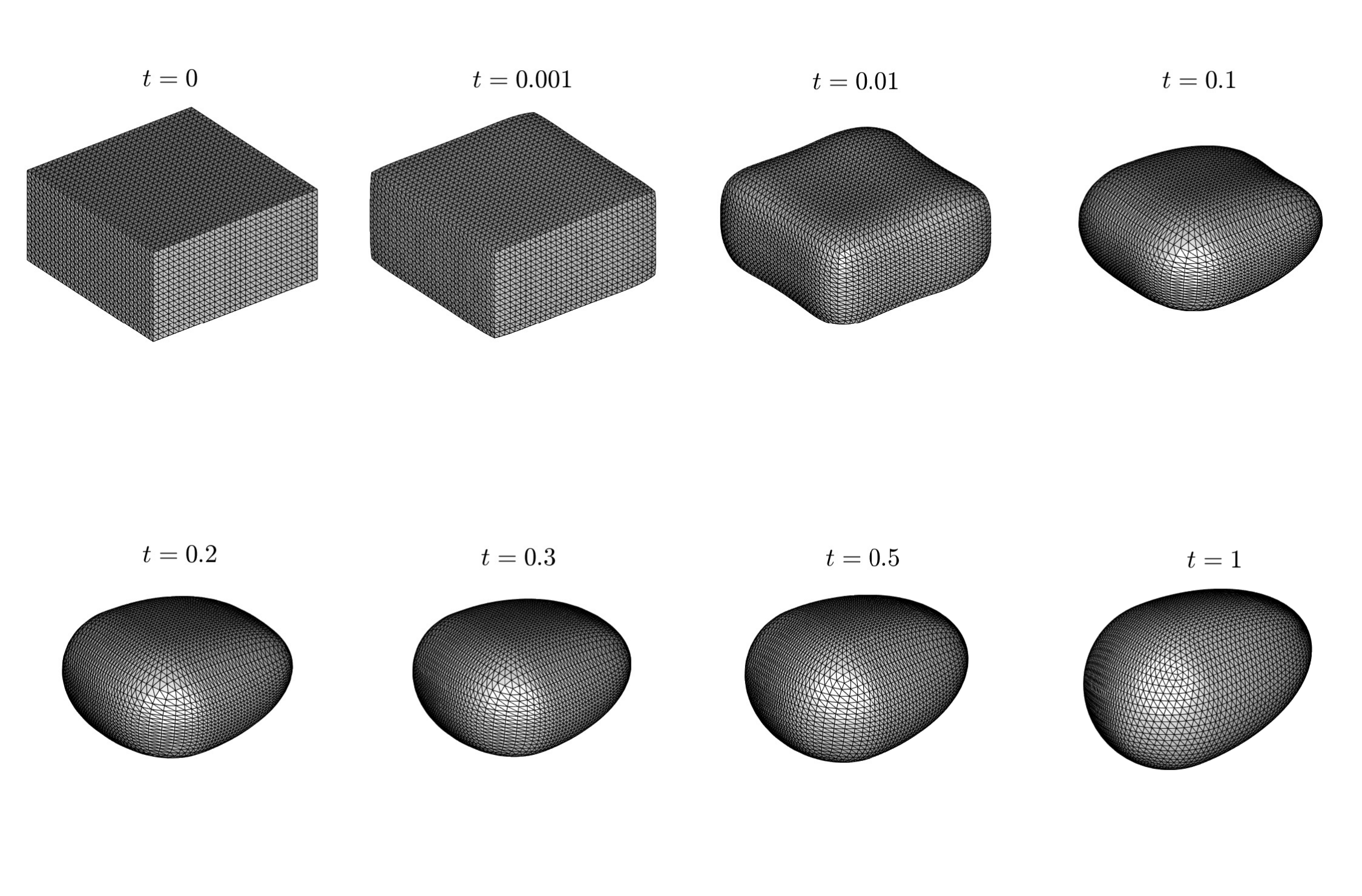}
\caption{Evolution of a $2\times 2\times 1$ cuboid by anisotropic surface diffusion with a weak anisotropy $\gamma(\boldsymbol{n})=\sqrt{(\frac{5}{2}+\frac{3}{2}\text{sign}(n_1))n_1^2+n_2^2+n_3^2}$ and $k(\boldsymbol{n})=k_0(\boldsymbol{n})$ at different times.}
\label{fig: evolve3}
\end{figure}

The numerical errors for the anisotropic energies $\gamma(\boldsymbol{n})$ in Case I-III and the stabilizing functions $k(\boldsymbol{n})=k_0(\boldsymbol{n})$ and $k(\boldsymbol{n})=\sup_{\boldsymbol{n}\in \mathbb{S}^2}k_0(\boldsymbol{n})$ are presented in Table \ref{tb: convergence rate}. Our results demonstrate that the order of convergence in $h$ is approximately 2 for all configurations, which suggests that our unified SP-PFEM \eqref{eq: full PFEM} is efficient. Additionally, we can reduce the bilinear interpolation cost by setting $k(\boldsymbol{n})=\sup_{\boldsymbol{n}\in \mathbb{S}^2}k_0(\boldsymbol{n})$ but achieve the same performance of efficiency.

To validate the volume conservation and the energy dissipation, we consider the normalized volume change $\frac{\Delta V(t)}{V(0)}$ and the normalized energy $\frac{W(t)}{W(0)}$ as follows
\begin{equation}
  \frac{\Delta V(t)}{V(0)}\Big|_{t=t_m}:=\frac{V^m-V^0}{V^0}, \qquad \frac{W(t)}{W(0)}\Big|_{t=t_m}:=\frac{W^m}{W^0}.
\end{equation}
We investigate the anisotropic energies in Case I-III with the initial $2\times 1\times 1$ elliptic and fixed mesh size $h=2^{-4}$ and time step $\tau=\frac{2}{25}h^2$. Figure \ref{fig: volume} shows the normalized volume changes with $k(\boldsymbol{n})=k_0(\boldsymbol{n})$, and Figure \ref{fig: energy1} illustrates the normalized energies with different $k(\boldsymbol{n})\geq k_0(\boldsymbol{n})$.  It can be seen in Figure \ref{fig: volume} that the normalized volume changes are in the same order of $10^{-15}$, which is almost at the machine round-off accuracy. We also observe that the normalized energies are monotonically decreasing, as shown in Figure \ref{fig: energy1}. In particular, the right column in Figure \ref{fig: energy1} indicates that the normalized energies are independent of $k(\boldsymbol{n})$.

To test the computational cost of the nonlinear system \eqref{eq: full PFEM}, we present the number of iterations per step when solving with Newton's method. We adopt the same initial mesh shape, mesh size, time step, and anisotropy as before. Figure \ref{fig: CNT}(a)-(c) show the number of iterations as functions of time corresponding to anisotropies in Case I, Case II, and Case III, respectively. We observe that, similar to the isotropic case presented in \cite{bao2021structurepreserving}, the number of iterations for \eqref{eq: full PFEM} are mostly between 2 and 4 for different anisotropies. This suggests that our unified SP-PFEM \eqref{eq: full PFEM} is highly efficient in computation.

The morphological evolutions of the $2\times 2\times 1$ cuboid under anisotropic surface diffusion are shown in Figures \ref{fig: evolve1}-\ref{fig: evolve3} for different anisotropic surface energy densities. We observe that the mesh points are well-behaved in each figure, and no mesh regularization is required. Moreover, by comparing the numerical equilibrium shapes in Figures \ref{fig: evolve1} and \ref{fig: evolve2}, we can find the corners become sharper as the anisotropic effect increases from $\frac{1}{8}$ to $\frac{1}{4}$. Finally, we note that although the regularity of $\gamma(\boldsymbol{n})$ in Case III is not $C^2$, our unified SP-PFEM \eqref{eq: full PFEM} works well for all the numerical tests, which validates our Remark \ref{remark: 2.1}.

\section{Conclusions}
We proposed a unified structure-preserving parametric finite element method (SP-PFEM) for anisotropic surface diffusion in both 2D and 3D, subject to a simple and mild condition $\gamma(-\boldsymbol{n})<(5-d)\gamma(\boldsymbol{n})$ and $\gamma(\boldsymbol{p})\in C^2(\mathbb{R}^{d}_*)$. By introducing the unified surface energy matrix $\boldsymbol{G}_k(\boldsymbol{n})$, we derived a new and unified conservative weak formulation for the chemical potential $\mu$ in all dimensions. Based on this unified conservative weak formulation, we used piecewise linear functions for spatial discretization and the implicit-explicit Euler method for the temporal discretization to obtain the unified SP-PFEM. To establish the unconditional energy stability, we introduced the minimal stabilizing function $k_0(\boldsymbol{n})$ based on the auxiliary matrix $\tilde{M}$ and $M$ in 2D and 3D, respectively. We developed a novel technique to show the existence of $k_0(\boldsymbol{n})$, which is also unified for all dimensions. Then we illustrated that the existence of $k_0(\boldsymbol{n})$ leads to local energy estimates and further unconditional energy stability. In fact, this new framework
for establishing unconditional energy stability of SP-PFEM sheds light on how to prove unconditional energy stability of other numerical methods 
for geometric partial differential equations.  
Finally, we presented numerical experiments to validate our analysis results and the efficiency of the unified SP-PFEM.



\appendix
\section*{Appendix A. A Divergence Theorem for Matrix-Valued Functions on a Closed Manifold of Codimension One}

\setcounter{theorem}{0}

\renewcommand{\thetheorem}{A.\arabic{theorem}}
\renewcommand{\theequation}{A.\arabic{equation}}
\renewcommand{\thedefinition}{\arabic{theorem}}
\renewcommand{\thelemma}{\arabic{theorem}}
\renewcommand{\theproposition}{\arabic{theorem}}
{
In this Appendix, we generalize the divergence theorem for vector fields to matrix-valued functions on a closed manifold of codimension one.

\begin{lemma}[Divergence theorem for matrix-valued functions]\label{exlem: divergence theorem for matrix-valued functions}
Let $\Gamma\subset\mathbb{R}^d$ ($d\ge2$) be a closed orientable $C^2$-manifold of codimension one with the unit normal vector $\boldsymbol{n}$, and let $\boldsymbol{F}:\Gamma \to \mathbb{R}^{d\times d}$ be a $C^1$ matrix-valued 
function. Then for any $\boldsymbol{\omega} \in C^1(\Gamma)$, we have
\begin{equation}\label{matrxintegr}
  \int_\Gamma \left(\nabla_\Gamma \cdot (\boldsymbol{F}\nabla_\Gamma \boldsymbol{X})\right) \cdot \boldsymbol{\omega}\,dA = - \int_\Gamma (\boldsymbol{F}\nabla_\Gamma \boldsymbol{X}): \nabla_\Gamma \boldsymbol{\omega} \, dA.
\end{equation}
\end{lemma}

\begin{proof}
Let $\{\boldsymbol{e}_1, \ldots, \boldsymbol{e}_d\}$ be the orthonormal basis of $\mathbb{R}^d$. The surface divergence of a matrix-valued function $\boldsymbol{F}:\Gamma \to \mathbb{R}^{d\times d}$ is defined as \cite[Definition 5 (viii)]{Barrett2020}
\begin{equation}\label{exeq: surface divergence of a matrix-valued function}
  \nabla_\Gamma \cdot \boldsymbol{F} := \sum_{i=1}^d\left(\nabla_\Gamma \cdot (\boldsymbol{F}^T \boldsymbol{e}_i)\right) \boldsymbol{e}_i.
\end{equation}
Using this definition, the left-hand side of \eqref{matrxintegr} can be written as
\begin{align*}
  \int_\Gamma \left(\nabla_\Gamma \cdot (\boldsymbol{F}\nabla_\Gamma \boldsymbol{X})\right) \cdot \boldsymbol{\omega}\,dA &= \int_\Gamma \left( \sum_{i=1}^d\left(\nabla_\Gamma \cdot ((\nabla_\Gamma \boldsymbol{X})^T \boldsymbol{F}^T \boldsymbol{e}_i)\right) \boldsymbol{e}_i\right) \cdot \boldsymbol{\omega}\,dA.
\end{align*}
Now, let $\Gamma$ be locally parametrized by $\mathfrak{X}(\boldsymbol{\theta}): U\to \mathbb{R}^d$ with $\boldsymbol{\theta}=(\theta_1, \theta_2, \ldots, \theta_{d-1})^T$ in $\mathbb{R}^{d-1}$. We use the standard notations:
\begin{equation}
  g_{ij}:=\frac{\partial \mathfrak{X}}{\partial \theta_i}\cdot \frac{\partial \mathfrak{X}}{\partial \theta_j}, \quad g:= \det(g_{ij}), \quad \left(g^{ij}\right):=\left(g_{ij}\right)^{-1}, \quad  i, j=1, 2, \ldots, d-1.
\end{equation}
Noticing $\nabla_\Gamma \boldsymbol{X} = I_d - \boldsymbol{n}\boldsymbol{n}^T$, we know $\boldsymbol{n}\cdot ((\nabla_\Gamma \boldsymbol{X})^T \boldsymbol{F}^T \boldsymbol{e}_i) = 0$, i.e., $(\nabla_\Gamma \boldsymbol{X})^T \boldsymbol{F}^T \boldsymbol{e}_i$ is a vector field, and $((\nabla_\Gamma \boldsymbol{X})\circ \mathfrak{X}) \frac{\partial \mathfrak{X}}{\partial \theta_i} = \frac{\partial \mathfrak{X}}{\partial \theta_i}$ for
$i=1,\ldots,d-1$. Therefore, using Proposition 2.46 in \cite{lee2018introduction}, the left-hand side of \eqref{matrxintegr} is further simplified as
\begin{align*}
  &\int_\Gamma \left(\nabla_\Gamma \cdot (\boldsymbol{F}\nabla_\Gamma \boldsymbol{X})\right) \cdot \boldsymbol{\omega}\,dA \\
  &= \int_U \left( \sum_{i=1}^d \frac{1}{\sqrt{g}}\sum_{j, l=1}^{d-1}\frac{\partial}{\partial \theta_j}\left(\sqrt{g}g^{jl}(((\nabla_\Gamma \boldsymbol{X})^T \boldsymbol{F}^T \boldsymbol{e}_i)\circ \mathfrak{X})\cdot \frac{\partial \mathfrak{X}}{\partial \theta_l}\right) \boldsymbol{e}_i\right) \cdot (\boldsymbol{\omega}\circ \mathfrak{X})\sqrt{g}\, d\theta \\
  &= \int_U \sum_{i=1}^d\sum_{j, l=1}^{d-1}\frac{\partial}{\partial \theta_j}\left(\sqrt{g}g^{jl}\frac{\partial \mathfrak{X}}{\partial \theta_l}^T((\boldsymbol{F}^T \boldsymbol{e}_i)\circ \mathfrak{X})\right) \left(\boldsymbol{e}_i \cdot (\boldsymbol{\omega}\circ \mathfrak{X})\right)\, d\theta \\
  &= \int_U \sum_{i=1}^d \sum_{j, l=1}^{d-1}\left(\frac{\partial}{\partial \theta_j}\left(\sqrt{g}g^{jl}(\boldsymbol{F}\circ \mathfrak{X})\frac{\partial \mathfrak{X}}{\partial \theta_l}\right)\cdot \boldsymbol{e}_i\right) \left(\boldsymbol{e}_i \cdot (\boldsymbol{\omega}\circ \mathfrak{X})\right)\, d\theta \\
  &= \int_U \sum_{j, l=1}^{d-1}\frac{\partial}{\partial \theta_j}\left(\sqrt{g}g^{jl}(\boldsymbol{F}\circ \mathfrak{X})\frac{\partial \mathfrak{X}}{\partial \theta_l}\right) \cdot (\boldsymbol{\omega}\circ \mathfrak{X})\, d\theta\\
  &= -\int_U \sum_{j, l=1}^{d-1}\left(\sqrt{g}g^{jl}(\boldsymbol{F}\circ \mathfrak{X})\frac{\partial \mathfrak{X}}{\partial \theta_l}\right) \cdot\frac{\partial (\boldsymbol{\omega}\circ \mathfrak{X})}{\partial \theta_j} \, d\theta.
\end{align*}
For the right-hand side of \eqref{matrxintegr}, we need the surface gradient of a vector-valued function $\boldsymbol{f}$, which is given by \cite[Remark 8 (i)]{Barrett2020}
\begin{equation*}
  (\nabla_\Gamma\boldsymbol{f})\circ \mathfrak{X}=\sum_{i, j=1}^{d-1}g^{ij}\frac{\partial (\boldsymbol{f}\circ \mathfrak{X})}{\partial \theta_i}\left(\frac{\partial \mathfrak{X}}{\partial {\theta_j}}\right)^T.
\end{equation*}
Using this, we can write:
\begin{align*}
  &- \int_\Gamma (\boldsymbol{F}\nabla_\Gamma \boldsymbol{X}): \nabla_\Gamma \boldsymbol{\omega} \, dA \\
  &= - \int_U \left(\sum_{i, j=1}^{d-1}g^{ij}(\boldsymbol{F}\circ \mathfrak{X})\frac{\partial \mathfrak{X}}{\partial \theta_i}\left(\frac{\partial \mathfrak{X}}{\partial \theta_j}\right)^T\right) : \left(\sum_{k, l=1}^{d-1}g^{kl}\frac{\partial (\boldsymbol{\omega}\circ \mathfrak{X})}{\partial \theta_k}\left(\frac{\partial \mathfrak{X}}{\partial \theta_l}\right)^T\right) \sqrt{g}\, d\theta \\
  &= -\int_U \sum_{i, j, k, l=1}^{d-1}g^{ij}g^{kl}\left(\frac{\partial \mathfrak{X}}{\partial \theta_j}\cdot \frac{\partial \mathfrak{X}}{\partial \theta_l}\right)\left(\left((\boldsymbol{F}\circ \mathfrak{X})\frac{\partial \mathfrak{X}}{\partial \theta_i}\right)\cdot \frac{\partial (\boldsymbol{\omega}\circ \mathfrak{X})}{\partial \theta_k}\right)\sqrt{g}\, d\theta \\
  &= -\int_U \sum_{i, j, k, l=1}^{d-1}g^{ij}g^{kl}g_{jl}\left(\left((\boldsymbol{F}\circ \mathfrak{X})\frac{\partial \mathfrak{X}}{\partial \theta_i}\right)\cdot \frac{\partial (\boldsymbol{\omega}\circ \mathfrak{X})}{\partial \theta_k}\right)\sqrt{g}\, d\theta \\
  &= -\int_U \sum_{i, k=1}^{d-1}g^{ik}\left(\left((\boldsymbol{F}\circ \mathfrak{X})\frac{\partial \mathfrak{X}}{\partial \theta_i}\right)\cdot \frac{\partial (\boldsymbol{\omega}\circ \mathfrak{X})}{\partial \theta_k}\right)\sqrt{g}\, d\theta.
\end{align*}
By combining the above two results, we complete the proof.
\end{proof}}




%
%

\end{document}